\documentclass[a4paper,legno]{amsproc}
\usepackage[english]{babel}
\usepackage{graphicx}
\usepackage[all]{xy}
\usepackage{amsthm}
\usepackage{amsfonts}
\usepackage{amssymb}
\usepackage{mathrsfs}
\usepackage{amsbsy}
\usepackage{stmaryrd}
\usepackage{amsmath}

\usepackage[usenames,dvipsnames]{color}
\usepackage{color}

\RequirePackage{color}
\definecolor{bwgreen}{rgb}{0.01,0.10,0.25}
\definecolor{bwmagenta}{rgb}{0.25,0.0,0.1}
\definecolor{bwblue}{rgb}{0.317,0.161,1}

\usepackage[usenames,dvipsnames]{xcolor}

\usepackage[colorlinks=true]{hyperref}
\hypersetup{
	bookmarksnumbered=true,
	linkcolor=bwmagenta,
	citecolor=bwgreen,
}

\newtheorem{theo}{Theorem}[section]
\newtheorem{cor}[theo]{Corollary}

\newtheorem{lemma}[theo]{Lemma}
\newtheorem{remark}[theo]{Remark}

\newtheorem{proposition}[theo]{Proposition}

\newtheorem{hyp}{Hypothesis}

\newtheorem*{main}{Theorem A}
\newtheorem*{main+}{Theorem B}
\newtheorem*{remarki}{Remark}

\newcommand{\quo}[1]{ \Z/p^{n}\Z  }

\newcommand{\iw}{\Lambda}

\newcommand{\fre}[1]{\stackrel{#1}{\rightarrow}}

\usepackage[OT2,T1]{fontenc}
\DeclareSymbolFont{cyrletters}{OT2}{wncyr}{m}{n}
\DeclareMathSymbol{\sha}{\mathalpha}{cyrletters}{"58}
\newcommand{\inlim}{\mathop{\varprojlim}\limits}

\newcommand{\dia}[1]{\left<{}#1\right>}

\newcommand{\lri}[1]{\left(#1\right)}
\newcommand{\rank}{\mathrm{rank}}

\newcommand{\ctsb}{C_{\mathrm{cont}}^{\bullet}}

\newcommand{\scob}{\widetilde{C}_{f}^{\bullet}}

\newcommand{\derco}{\widetilde{\mathbf{R}\Gamma}_{f}}
\newcommand{\exsel}{\widetilde{H}_f}
\newcommand{\divp}{\Q_p/\Z_p}

\newcommand{\Hom}[1]{\mathrm{Hom}_{#1}}

\newcommand{\xari}{\mathcal{X}^{\mathrm{arith}}}

\newcommand{\derot}[1]{\otimes_{#1}^{\mathbf{L}}}

\newcommand{\neko}{Nekov\'a\v{r}}
\newcommand{\dercts}{\mathbf{R}\Gamma_{\mathrm{cont}}}

\newcommand{\ppq}{\mathbb{T}_{\mathbf{f}}}
\newcommand{\I}{\mathbb{I}}
\newcommand{\ppql}{\mathbb{T}_{\mathbf{f},\mathfrak{p}_{f}}}
\newcommand{\Il}{\I_{\mathfrak{p}_{f}}}

\newcommand{\ppqll}{\mathbb{T}_{\mathbf{f},\mathfrak{p}_{f}}}
\newcommand{\N}{\mathbf{N}}
\newcommand{\Z}{\mathbf{Z}}
\newcommand{\Q}{\mathbf{Q}}

\newcommand{\C}{\mathbf{C}}
\newcommand{\F}{\mathbf{F}}

\author{Rodolfo Venerucci}

\begin{document}

\title{On the $p$-converse of the Kolyvagin-Gross-Zagier theorem}
\maketitle

\vspace{-6mm}

\begin{center}
\small{Universität Duisburg-Essen \\
Fakultät für Mathematik,
Mathematikcarrée \\  
Thea-Leymann-Stra\ss{}e 9, 45127 Essen. \\
\textit{E-mail address}: \texttt{rodolfo.venerucci@uni-due.de}
}\end{center}


\begin{abstract} Let $A/\Q$ be an elliptic curve having split multiplicative reduction at an odd prime $p$.
Under some mild technical assumptions, we prove the statement: 
\[
                       \rank_{\Z}A(\Q)=1\ \text{and\ }\ \#\big(\sha(A/\Q)_{p^{\infty}}\big)<\infty\ \ 
                       \implies\ \ \mathrm{ord}_{s=1}L(A/\Q,s)=1,
\]
thus providing a `$p$-converse' to a celebrated theorem of Kolyvagin-Gross-Zagier. 
\end{abstract}
\setcounter{tocdepth}{1}
\tableofcontents

\section*{Introduction}

Let $A$ be an elliptic curve defined over $\Q$, let $L(A/\Q,s)$ be its Hasse-Weil $L$-function, and let
$\sha(A/\Q)$ be its Tate-Shafarevich group.  
The (weak form of the) conjecture of Birch and Swinnerton-Dyer predicts that $\sha(A/\Q)$ is finite, and that 
the order of vanishing $\mathrm{ord}_{s=1}L(A/\Q,s)$ of $L(A/\Q,s)$ at $s=1$
equals the rank of the Mordell-Weil group $A(\Q)$.
The main result to date in support of this conjecture comes combining the 
fundamental work of Kolyvagin \cite{Koles} and Gross-Zagier \cite{Gross-Zagier} (KGZ theorem for short):
\[
          r_{\mathrm{an}}:=\mathrm{ord}_{s=1}L(A/\Q,s)\leq{}1\ \ \Longrightarrow{} \ \    
                 \mathrm{rank}_{\Z}A(\Q)=r_{\mathrm{an}} \ \ \text{and} \ \ \#\big(\sha(A/\Q)\big)<\infty.
\]

Let $p$ be a rational prime, let $r_{\mathrm{alg}}\in{}\{0,1\}$, and let $\sha(A/\Q)_{p^{\infty}}$ be the $p$-primary part of $\sha(A/\Q)$.
By the  \emph{$p$-converse of the KGZ theorem in rank $r_{\mathrm{alg}}$} we mean the conjectural statement
\[
       \mathrm{rank}_{\Z}A(\Q)=r_{\mathrm{alg}}\ \ \mathrm{and\ }\  \#\big(\sha(A/\Q)_{p^{\infty}}\big)<\infty\ \ 
       \stackrel{? \ }{\Longrightarrow{}} \ \ \mathrm{ord}_{s=1}L(A/\Q,s)=r_{\mathrm{alg}}.
\]
Thanks to the fundamental work  of Bertolini-Darmon, Skinner-Urban and their schools, we have now
(at least conceptually) all the necessary tools to attack the $p$-converse of the KGZ theorem.
Notably, assume that $p$ is a prime of \emph{good ordinary reduction} for $A/\Q$. In this case the $p$-converse of the KGZ theorem in rank $0$ follows 
by \cite{S-U}. In the preprint \cite{Skinner}, Skinner combines Wan's  Ph.D. Thesis \cite{Wan} $-$which proves,
following the ideas and the strategy used in \cite{S-U}, one divisibility in the Iwasawa main conjecture 
for Rankin-Selberg $p$-adic $L$-functions$-$
with the main results of \cite{Be-Da-Pr} and  Brooks's Ph.D. Thesis \cite{Brooks} $-$extending the results of \cite{Be-Da-Pr}$-$
to  prove many cases of the $p$-converse of the KGZ theorem  in rank $1$.
In the preprint \cite{WeiZ}, W. Zhang also proves (among other things) many cases of the $p$-converse of the KGZ theorem in rank $1$
for good  ordinary primes, combining the results of \cite{S-U}
with the results and ideas presented in Bertolini-Darmon's proof of (one divisibility in) the anticyclotomic main conjecture \cite{Be-Da-main}.
The same strategy also appears  in Berti's forthcoming Ph.D. Thesis \cite{Berti}
(see also \cite{BBV}).

The aim of this note is to prove the $p$-converse of the KGZ theorem in rank $1$ for  a prime $p$
of \emph{split multiplicative reduction} for $A/\Q$.
Our strategy is different from both the one of \cite{Skinner} and the one of  \cite{WeiZ}, and is based on the (two-variable) Iwasawa 
theory for the Hida deformation of the $p$-adic Tate module of $A/\Q$.
Together  with the results of the author's Ph.D. Thesis \cite{PhD}, and then \neko's theory of Selmer Complexes \cite{Ne}
(on which the results of \cite{PhD} rely),
the key ingredients in our approach are represented by the main results of \cite{B-D} and \cite{S-U}
(see the outline of the proof given below for more details). \footnote{After this note was written, C. Skinner communicated to the author 
that, together with W. Zhang, he extended the methods of \cite{WeiZ} to obtain (among other results) the  $p$-converse of the KGZ theorem 
in cases where  $p$ is a prime of multiplicative reduction \cite{SkinnerZhang}. While there is an overlap between 
the main result of this note and the result of Skinner-Zhang, neither subsumes the other
(cf. the end of this section). Moreover, as remarked above,
the methods of proof are substantially different.    }

\subsection*{The main result} Let $A/\Q$ be an elliptic curve having \emph{split} multiplicative reduction at an \emph{odd}
rational prime $p$. Let $N_{A}$ be the conductor of $A/\Q$,
let $j_{A}\in{}\Q$ be its $j$-invariant, and let 
$\overline{\rho}_{A,p} : G_{\Q}\fre{}\mathrm{GL}_{2}(\F_p)$  be (the isomorphism class of) the 
representation of $G_{\Q}$ on the  
$p$-torsion submodule $A[p]$ of $A(\overline{\Q})$.

\begin{main}\label{mainth} Let $A/\Q$ and $p\not=2$ be as above. Assume in addition that the following properties hold:
\begin{itemize}
\item[1.] $\overline{\rho}_{A,p}$ is irreducible; 
\item[2.] there exists a prime $q\Vert{}N_{A}$, $q\not=p$ such that $p\nmid{}\mathrm{ord}_{q}(j_{A})$;
\item[3.] $\mathrm{rank}_{\Z}A(\Q)=1$ and $\sha(A/\Q)_{p^{\infty}}$ is finite.
\end{itemize}
Then the Hasse-Weil $L$-function  $L(A/\Q,s)$ of $A/\Q$ has a simple zero at $s=1$.
\end{main}

Combined with the KGZ theorem recalled above, this  implies:

\begin{main+} Let $A/\Q$ be an elliptic curve having split multiplicative reduction at an odd rational  prime $p$.
Assume that $\overline{\rho}_{A,p}$ is irreducible, and that there exists a prime $q\Vert{}N_{A}$, $q\not=p$ such that $p\nmid{}\mathrm{ord}_{q}(j_{A})$.
Then
\[
                   \mathrm{ord}_{s=1}L(A/\Q,s)=1 \ \ \iff \ \ \mathrm{rank}_{\Z}A(\Q)=1\ \mathrm{and}
                   \ \#\big(\sha(A/\Q)_{p^{\infty}}\big)<\infty.
\]
If this is the case, the whole  Tate-Shafarevich group $\sha(A/\Q)$ is finite. 
\end{main+}

\subsection*{Outline of the proof} Let $A/\Q$ be an elliptic curve having split multiplicative reduction 
at a prime $p\not=2$, and let $f=\sum_{n=1}^{\infty}a_{n}q^{n}
\in{}S_{2}(\Gamma_{0}(N_{A}),\Z)^{\mathrm{new}}$ be the weight-two newform attached to $A$
by the modularity theorem of Wiles, Taylor-Wiles \emph{et. al.} Then $N_{A}=Np$, with $p\nmid{}N$
and $a_{p}=a_{p}(A)=+1$. Assume that $\overline{\rho}_{A,p}$ is irreducible.

Let $\mathbf{f}=\sum_{n=1}^{\infty}\mathbf{a}_{n}q^{n}\in{}\I\llbracket{}q\rrbracket$ be the \emph{Hida family}
passing through $f$. 
Here $\I$ is a normal local domain, finite and flat over Hida's weight algebra $\iw:=\mathcal{O}_{L}\llbracket{}\Gamma\rrbracket$
with $\mathcal{O}_{L}$-coefficients, where $\Gamma:=1+p\Z_{p}$ and $\mathcal{O}_{L}$
is the ring of integers of a (sufficiently large) finite extension $L/\Q_{p}$ (cf. Section $\ref{hifa}$). There is 
a natural injective morphism (Mellin transform) $\mathtt{M} : \I\hookrightarrow{}\mathscr{A}(U)$,
where $U\subset{}\Z_{p}$ is a suitable $p$-adic neighbourhood of $2$, and $\mathscr{A}(U)\subset{}L\llbracket{}k-2\rrbracket$
denotes the sub-ring of formal power series in $k-2$ which converge  in $U$
(see Section $\ref{lochidafam}$). Write $$f_{\infty}:=\sum_{n=1}^{\infty}a_{n}(k)\cdot{}q^{n}
\in{}\mathscr{A}(U)\llbracket{}q\rrbracket,$$ with $a_{n}(k)\in{}\mathscr{A}(U)$ defined as the image of $\mathbf{a}_{n}\in\I$
under $\mathtt{M}$. 
For every \emph{classical point} $\kappa\in{}U^{\mathrm{cl}}:=U\cap{}\Z^{\geq{}2}$,
the weight-$\kappa$-specialization $f_{\kappa}:=\sum_{n=1}^{\infty}a_{n}(\kappa)q^{n}$ is the $q$-expansion of a normalised 
Hecke eigenform of weight $\kappa$ and level $\Gamma_{1}(Np)$; moreover $f_{2}=f$.
For every quadratic character $\chi$ of conductor coprime with $Np$, a construction of Mazur-Kitagawa and Greenberg-Stevens \cite[Section 1]{B-D} attaches to 
$f_{\infty}$ and $\chi$   a two-variable $p$-adic analytic $L$-function $L_{p}(f_{\infty},\chi,k,s)$ on $U\times{}\Z_{p}$,
interpolating the special complex $L$-values $L(f_{\kappa},\chi,j)$, where $\kappa\in{}U^{\mathrm{cl}}$, $1\leq{}j\leq{}\kappa-1$
and $L(f_{\kappa},\chi,s)$ is the Hecke $L$-function of $f_{\kappa}$ twisted by $\chi$.
(Here $s$ is the \emph{cyclotomic variable}, and $k$ is the \emph{weight-variable}.)
Define the \emph{central critical $p$-adic $L$-function of $(f_{\infty},\chi)$}:
\[
             L_{p}^{\mathrm{cc}}(f_{\infty},\chi,k):=L_{p}(f_{\infty},\chi,k,k/2)\in{}\mathscr{A}(U)
\]
as the restriction of  the Mazur-Kitagawa $p$-adic $L$-function to the \emph{central critical line $s=k/2$} in the $(k,s)$-plane.

On the algebraic side, Hida theory attaches to $\mathbf{f}$ a \emph{central critical deformation $\ppq$} of the $p$-adic Tate module of 
$A/\Q$. $\ppq$ is a free rank-two $\I$-module, equipped with a continuous, $\I$-linear action of $G_{\Q}$,
satisfying the following interpolation property.
For every classical point $\kappa\in{}U^{\mathrm{cl}}$ (s.t. $\kappa\equiv{}2\ \mathrm{mod}\ 2(p-1)$) 
the base change $\mathbb{T}_{\mathbf{f}}\otimes_{\I,\mathrm{ev}_{\kappa}}L$ is isomorphic to the central critical twist 
$V_{f_{\kappa}}(1-\kappa/2)$ of the contragredient of the $p$-adic Deligne representation $V_{f_{\kappa}}$
of $f_{\kappa}$, where $\mathrm{ev}_{\kappa} : \I\hookrightarrow{}\mathscr{A}(U)\fre{}L$ denotes the morphism 
induced by evaluation at $\kappa$ on $\mathscr{A}(U)$. Moreover, $\mathbb{T}_{\mathbf{f}}$ is \emph{nearly-ordinary} at $p$. More precisely, let $v$ be a prime of $\overline{\Q}$ dividing $p$,
associated with an embedding $i_{v} : \overline{\Q}\hookrightarrow{}\overline{\Q}_{p}$,
and denote by  
$i_{v}^{\ast} : G_{\Q_{p}}\cong{}G_{v}\subset{}G_{\Q}$ the corresponding  decomposition group at $v$.
Then there is a short exact sequence of $\I[G_{v}]$-modules 
\[
       0\fre{}\mathbb{T}^{+}_{\mathbf{f},v}\fre{}\mathbb{T}_{\mathbf{f}}\fre{}\mathbb{T}_{\mathbf{f},v}^{-}\fre{}0,
\]
with $\mathbb{T}_{\mathbf{f},v}^{\pm}$
 free of rank one over $\I$.
For every number field $F/\Q$,  define the  (strict) Greenberg Selmer group
\[
        \mathrm{Sel}_{\mathrm{Gr}}^{\mathrm{cc}}(\mathbf{f}/F):=\ker\lri{
        H^{1}(G_{F,S},\ppq\otimes_{\I}\I^{\ast})\longrightarrow{}\prod_{v|p}H^{1}(F_{v},\mathbb{T}_{\mathbf{f},v}^{-}\otimes_{\I}\I^{\ast})}.
\]
Here $S$ is a finite set of primes of $F$ containing every prime divisor of  $N_{A}\mathrm{disc}(F)$, $G_{F,S}$ 
is the Galois group of the maximal algebraic extension of $F$
which is unramified outside $S\cup{}\{\infty\}$, 
$\I^{\ast}:=\Hom{\mathrm{cont}}(\I,\divp)$ is the Pontrjagin dual of $\I$, and the product 
runs over all the primes $v$ of $F$ which divide $p$
\footnote{$\mathrm{Sel}_{\mathrm{Gr}}^{\mathrm{cc}}(\mathbf{f}/F)$ depends on the choice of the set $S$,
even if this dependence is irrelevant for the purposes of this introduction.}. Write
\[
                   X_{\mathrm{Gr}}^{\mathrm{cc}}(\mathbf{f}/F):=\Hom{\Z_{p}}\Big(\mathrm{Sel}_{\mathrm{Gr}}^{\mathrm{cc}}(\mathbf{f}/F),
                   \divp\Big)
\]
for the Pontrjagin dual of $\mathrm{Sel}_{\mathrm{Gr}}^{\mathrm{cc}}(\mathbf{f}/F)$. It is a finitely generated $\I$-module.
We now explain the  main steps entering in the proof of Theorem A.

\subsubsection*{Step I: Skinner-Urban's divisibility.} Let $K/\Q$ be an imaginary quadratic field in which $p$ \emph{splits}.
Assume that the discriminant of $K/\Q$ is coprime to $N_{A}$, and write $N_{A}=N^{+}N^{-}$,
where $N^{+}$ (resp., $N^{-}$) is divided precisely by the prime divisors of $N_{A}$ which are split (resp., inert) in $K$.
Assume the following \emph{generalised Heegner hypothesis} and \emph{ramification hypothesis}: \vspace{2mm}

$\bullet$ $N^{-}$ is a square-free product of an \emph{odd} number of primes. \vspace{2mm}

$\bullet$ $\overline{\rho}_{A,p}$ is ramified at all prime divisors of $N^{-}$.\vspace{2mm}\\
Under some additional  technical hypotheses
on the data $(A,K,p,\dots)$ (cf.  Hypotheses $\ref{h2}$, $\ref{h3}$ and $\ref{h4}$ below), the main result of \cite{S-U},
together with some auxiliary computations, allows us to deduce the following inequality:
\begin{equation}\label{eq:suintro}
              \mathrm{ord}_{k=2}L_{p}^{\mathrm{cc}}(f_{\infty}/K,k)\leq{}
              \mathrm{length}_{\mathfrak{p}_{f}}\Big(X_{\mathrm{Gr}}^{\mathrm{cc}}(\mathbf{f}/K)\Big)+2.
\end{equation}
Here $L_{p}^{\mathrm{cc}}(f_{\infty}/K,k):=L_{p}^{\mathrm{cc}}(f_{\infty},\chi_{\mathrm{triv}},k)\cdot{}L_{p}^{\mathrm{cc}}(f_{\infty},\epsilon_{K},k)$,
where $\chi_{\mathrm{triv}}$ is the trivial character and $\epsilon_{K}$ is the quadratic character attached to $K$.
$\mathfrak{p}_{f}:=\ker\big(\mathrm{ev}_{2} : \I\hookrightarrow{}\mathscr{A}(U)\rightarrow{}L\big)$
is the kernel of the morphism induced by evaluation at $k=2$ on $\mathscr{A}(U)$; it is a height-one prime ideal of  $\I$,
so that the localisation $\Il$ is a discrete valuation ring.
Finally, $\mathrm{length}_{\mathfrak{p}_{f}}(M)$ denotes the length over $\Il$ of the localisation $M_{\mathfrak{p}_{f}}$,
for every finite $\I$-module $M$.

\begin{remarki}\emph{
The main result of  Skinner and Urban \cite{S-U} mentioned above, which 
proves one divisibility in a three variable main conjecture for $\mathrm{GL}_{2}$,
is a result \emph{over $K$}, for $K/\Q$ as above, and not over $\Q$. This is why we need to consider a base-change to 
such a $K/\Q$ in our approach to Theorem A. }
\end{remarki}

\begin{remarki}\emph{ By assumption, $A/\Q$ has split multiplicative reduction at $p$,
and as well-known this implies that $L_{p}(f_{\infty},\chi_{\mathrm{triv}},k,s)$ has a trivial zero at $(k,s)=(2,1)$
in the sense of \cite{M-T-T}. Moreover, the hypothesis $\epsilon_{K}(p)=+1$ (i.e. $p$ splits in $K$) implies that $L_{p}(f_{\infty},\epsilon_{K},k,s)$
also 
has such an exceptional zero at $(k,s)=(2,1)$ (see, e.g. \cite[Section 1]{B-D}). 
This is the reason behind the appearance of the addend $2$ on the R.H.S. of $(\ref{eq:suintro})$.
}
\end{remarki}

\begin{remarki}\emph{The generalised Heegner hypothesis gives $\epsilon_{K}(-N_{A})=-\epsilon_{K}(N^{-})=+1$. This 
implies that the Hecke $L$-series $L(f,s)=L(A/\Q,s)$ and $L(f,\epsilon_{K},s)=L(A^{K}/\Q,s)$
(where $A^{K}/\Q$ is the quadratic twist of $A$ by $K$) have the \emph{same} sign in their functional equations 
at $s=1$. The Birch and Swinnerton-Dyer conjecture then predicts that the ranks of  $A(\Q)$ and $A^{K}(\Q)\cong{}
A(K)^{-}$ have the same parity. In particular $\mathrm{rank}_{\Z}A(K)$, and then $\mathrm{ord}_{k=2}L_{p}^{\mathrm{cc}}(f_{\infty}/K,k)$
 should be \emph{even}.}
\end{remarki}

\subsubsection*{Step II: Bertolini-Darmon's exceptional-zero formula} Let $K/\Q$
be as in Step I. Assume moreover\vspace{2mm}

$\bullet$ $\mathrm{sign}(A/\Q)=-1$\vspace{2mm}\\
where $\mathrm{sign}(A/\Q)\in{}\{\pm1\}$
denotes the sign in the functional equation satisfied by the Hasse-Weil $L$-function $L(A/\Q,s)$.
As remarked above, this implies that $\mathrm{sign}(A^{K}/\Q)=-1$ too. 
The analysis carried out in \cite{G-S} and \cite{B-D} tells us that, for both $\chi=\chi_{\mathrm{triv}}$ and $\chi=\epsilon_{K}$:
\begin{equation}\label{eq:exzerointro}
                \mathrm{ord}_{k=2}L_{p}^{\mathrm{cc}}(f_{\infty},\chi,k)\geq{}2;
\end{equation}
this is once again a manifestation of the presence of an exceptional zero at $(k,s)=(2,1)$ for the Mazur-Kitagawa $p$-adic $L$-function 
$L_{p}(f_{\infty},\chi,k,s)$.
Much more deeper,  Bertolini and Darmon proved in \cite{B-D} the formula
\[
                  \frac{d^{2}}{dk^{2}}L_{p}^{\mathrm{cc}}(f_{\infty},\chi,k)_{k=2}\stackrel{\cdot{}}{=}\log_{A}^{2}(\mathbf{P_{\chi}}),
\]
where $\stackrel{\cdot{}}{=}$ denotes equality up to a non-zero factor, $\log_{A} : A(\Q_{p})\fre{}\Q_{p}$
is the formal group logarithm, and $\mathbf{P}_{\chi}\in{}A(K)^{\chi}$ is a Heegner point.
This formula implies  that
\begin{equation}\label{eq:bdintro}
                   \mathrm{ord}_{k=2}L_{p}^{\mathrm{cc}}(f_{\infty},\chi,k)=2\ \ \iff\ \ \mathrm{ord}_{s=1}L(A^{\chi}/\Q,s)=1,
\end{equation}
i.e. if and only if the Hasse-Weil $L$-function of the $\chi$-twist $A^{\chi}/\Q$ has a simple zero at $s=1$.
(Here of course $A^{\chi}=A$ is $\chi=\chi_{\mathrm{triv}}$ and $A^{\chi}=A^{K}$ if $\chi=\epsilon_{K}$.
Recall that by assumption $L(A^{\chi}/\Q,s)$ vanishes at $s=1$.)

\subsubsection*{Step III: bounding the characteristic ideal.} Let $\chi$ denote either the trivial character or a quadratic character
of conductor coprime with $Np$, and write 
$K_{\chi}:=\Q$ or $K_{\chi}/\Q$ for the quadratic field attached to $\chi$ accordingly.
Making use of \neko's theory of Selmer Complexes (especially of \neko's generalised Cassels-Tate pairings) \cite{Ne},
we are able to relate the structure of the $\Il$-module $X_{\mathrm{Gr}}^{\mathrm{cc}}(\mathbf{f}/K_{\chi})^{\chi}_{\mathfrak{p}_{f}}:
=X_{\mathrm{Gr}}^{\mathrm{cc}}(\mathbf{f}/K_{\chi})^{\chi}\otimes_{\I}\Il$ to the properties of a suitable 
\emph{\neko's half-twisted weight pairing} (see Section $\ref{htwp}$)
\[
           \dia{-,-}_{V_{f},\pi}^{\mathrm{Nek},\chi} : A^{\dag}(K_{\chi})^{\chi}\times{}A^{\dag}(K_{\chi})^{\chi}\longrightarrow{}\Q_{p},
\]
playing here the r\^ole of the canonical cyclotomic $p$-adic height pairing of Schneider, Mazur-Tate \emph{et. al.}
in cyclotomic Iwasawa theory. Here, for every $\Z[\mathrm{Gal}(K_{\chi}/\Q)]$-module 
$M$, we write $M^{\chi}$ for the submodule of $M$ on which $\mathrm{Gal}(K_{\chi}/\Q)$ acts via $\chi$,
and $A^{\dag}(K_{\chi})$ is the \emph{extended Mordell-Weil group} of $A/K_{\chi}$ introduced in \cite{M-T-T}.
$\dia{-,-}^{\mathrm{Nek},\chi}_{V_{f},\pi}$ is a bilinear and \emph{skew-symmetric} form on $A^{\dag}(K_{\chi})^{\chi}$
(see Section $\ref{mainsec}$).
Assume that the following conditions are satisfied: \vspace{2mm}

$\bullet$ $\chi(p)=1$, i.e. $p$ splits in $K_{\chi}$;\vspace{2mm}

$\bullet$ $\rank_{\Z}A(K_{\chi})^{\chi}=1$ and $\sha(A/K_{\chi})_{p^{\infty}}^{\chi}$ is finite.\vspace{2mm}\\
Then $A^{\dag}(K_{\chi})^{\chi}\otimes\Q_{p}=\Q_{p}\cdot{}q_{\chi}\oplus{}\Q_{p}\cdot{}P_{\chi}$ is a $2$-dimensional 
$\Q_{p}$-vector space generated by a non-zero point $P_{\chi}\in{}A(K_{\chi})^{\chi}\otimes\Q$ and a certain \emph{Tate's period}
$q_{\chi}\in{}A^{\dag}(K_{\chi})^{\chi}$ (which does \emph{not} come from a $K_{\chi}$-rational point of $A$). 
In the author's Ph.D. Thesis \cite{PhD} we proved that
\begin{equation}\label{eq:explexintro}
            \dia{q_{\chi},P_{\chi}}^{\mathrm{Nek},\chi}_{V_{f},\pi}\stackrel{\cdot{}}{=}\log_{A}(P_{\chi})
\end{equation}
(where $\stackrel{\cdot{}}{=}$ denotes again equality up to a non-zero multiplicative factor), which implies that 
$\dia{-,-}_{V_{f},\pi}^{\mathrm{Nek},\chi}$ is non-degenerate on $A^{\dag}(K_{\chi})^{\chi}$.
Together with the results of \neko{} mentioned above,
this allows us to deduce  that 
\begin{equation}\label{eq:mmintro}
                           X_{\mathrm{Gr}}^{\mathrm{cc}}(\mathbf{f}/K_{\chi})^{\chi}_{\mathfrak{p}_{f}}\cong{}\Il/\mathfrak{p}_{f}\Il.
\end{equation}

\begin{remarki}\emph{Let $V_{f}:=\mathrm{Ta}_{p}(A)\otimes_{\Z_{p}}\Q_{p}$
be the $p$-adic Tate module of $A/\Q$, 
and let 
$H^{1}_{f}(K_{\chi},V_{f})$
be the Bloch-Kato Selmer group of $V_{f}$ over $K_{\chi}$.
The pairing $\dia{-,-}_{V_{f},\pi}^{\mathrm{Nek},\chi}$ is naturally defined on \neko's \emph{extended Selmer group}
$\exsel^{1}(K_{\chi},V_{f})^{\chi}$, which is an extension of $H^{1}_{f}(K_{\chi},V_{f})^{\chi}$
by the $\Q_{p}$-module generated by $q_{\chi}$.
Indeed it is the non-degeneracy of $\dia{-,-}_{V_{f},\pi}^{\mathrm{Nek},\chi}$ on 
$\exsel^{1}(K_{\chi},V_{f})^{\chi}$ 
to be directly related to the structure of the $\Il$-module $X_{\mathrm{Gr}}^{\mathrm{cc}}(\mathbf{f}/K_{\chi})^{\chi}_{\mathfrak{p}_{f}}$.
On the other hand, $\exsel^{1}(K_{\chi},V_{f})^{\chi}$ contains  $A^{\dag}(K_{\chi})^{\chi}\otimes{}\Q_{p}$, and equals it precisely if the $p$-primary part of $\sha(A/K_{\chi})^{\chi}$
is finite. This explains why we need the finiteness of $\sha(A/K_{\chi})^{\chi}_{p^{\infty}}$ in order to deduce $(\ref{eq:mmintro})$.}
\end{remarki}

\begin{remarki}\emph{ The length of $X_{\mathrm{Gr}}^{\mathrm{cc}}(\mathbf{f}/K_{\chi})_{\mathfrak{p}_{f}}^{\chi}$
over $\I_{\mathfrak{p}_{f}}$
can be interpreted as the order of vanishing at $k=2$ of an algebraic 
$p$-adic $L$-function $\mathbb{L}^{\mathrm{cc}}_{p}(f_{\infty},\chi,k)\in{}\mathscr{A}(U)$,
defined as the Mellin transform of the characteristic ideal of 
$X_{\mathrm{Gr}}^{\mathrm{cc}}(\mathbf{f}/K_{\chi})^{\chi}$
(at least assuming that $\I$ is regular).
The results of \neko{} briefly mentioned above
can be used to prove 
 an  analogue in our setting of the 
algebraic $p$-adic Birch and Swinnerton-Dyer formulae of Schneider \cite{Sch} and Perrin-Riou \cite{PR},
which relates the leading coefficient of $\mathbb{L}_{p}^{\mathrm{cc}}(f_{\infty},\chi,k)$
at $k=2$ to the determinant of $\dia{-,-}^{\mathrm{Nek},\chi}_{V_{f},\pi}$,
computed on  $A^{\dag}(K_{\chi})^{\chi}/\mathrm{torsion}$.
}
\end{remarki}

\begin{remarki}\emph{Formula $(\ref{eq:explexintro})$ is crucial here. Indeed, as remarked above, it allows us to deduce the non-degeneracy of 
the weight-pairing $\dia{-,-}^{\mathrm{Nek},\chi}_{V_{f},\pi}$. The analogue of this result in cyclotomic Iwasawa theory 
(i.e. Schneider conjecture in rank-one) seems out of reach at present.}
\end{remarki}

\begin{remarki}\emph{The preceding results, and $(\ref{eq:explexintro})$ in particular, should be considered as an algebraic counterpart
of Bertolini-Darmon's exceptional zero formula (cf. Step II). This point of view is developed in \cite{Ven} (see also
Part I of the author's Ph.D. thesis 
\cite{PhD}), and leads  to the 
formulation of two-variable analogues of the Birch and Swinnerton-Dyer conjecture for 
the Mazur-Kitagawa $p$-adic $L$-function $L_{p}(f_{\infty},\chi,k,s)$. Formula $(\ref{eq:explexintro})$
$-$to be considered part of \neko's theory$-$
and  Bertolini-Darmon's exceptional zero formula, also represent crucial 
ingredients in the proof, given in \cite{Ven-2}, of the Mazur-Tate-Teitelbaum exceptional zero conjecture in rank one.
}
\end{remarki}

\subsubsection*{Step IV: conclusion of the proof} Assume that the hypotheses of Theorem A are satisfied. 
Thanks to \neko's proof of the parity conjecture \cite{Ne},  $\mathrm{sign}(A/\Q)=-1$.
By the main result 
of \cite{BFH} and hypothesis $2$ in Theorem A, we are then able to find a quadratic imaginary field 
$K/\Q$ which satisfies the hypotheses needed  in Steps I and II, with $N^{-}=q$, and such that 
$L(A^{K}/\Q,s)$ has a simple zero at $s=1$, i.e.
\begin{equation}\label{eq:jjjqqq}
      \mathrm{ord}_{s=1}L(A^{K}/\Q,s)=1.
\end{equation}
An application of the KGZ theorem gives
\[
          \mathrm{rank}_{\Z}A^{K}(\Q)=1;\ \ \#\Big(\sha(A^{K}/\Q)_{p^{\infty}}\Big)<\infty.
\]
Together with hypothesis $3$ in Theorem A, this implies that the hypotheses  needed in  Step III
are satisfied by both the trivial character $\chi=\chi_{\mathrm{triv}}$ and $\chi=\epsilon_{K}$.
Then
\[
            4\stackrel{(\ref{eq:exzerointro})\ }{\leq{}}\mathrm{ord}_{k=2}L_{p}^{\mathrm{cc}}(f_{\infty}/K,k)
            \stackrel{(\ref{eq:suintro})\ }{\leq{}}
            \mathrm{length}_{\mathfrak{p}_{f}}\Big(X_{\mathrm{Gr}}^{\mathrm{cc}}(\mathbf{f}/K)\Big)+2
            \stackrel{(\ref{eq:mmintro})}{=}4,
\]
i.e. $\mathrm{ord}_{k=2}L_{p}^{\mathrm{cc}}(f_{\infty}/K,k)=4$. Applying now Bertolini-Darmon's result
$(\ref{eq:bdintro})$ yields
\[
                \mathrm{ord}_{s=1}L(A/K,s)=2,
\]
where $L(A/K,s)=L(A/\Q,s)\cdot{}L(A^{K}/\Q,s)$ is the Hasse-Weil $L$-function of $A/K$.
Together with  $(\ref{eq:jjjqqq})$, this implies that 
$L(A/\Q,s)$ has a simple zero at $s=1$, as was to be shown.

\subsection*{Recent related results}
In the recent preprint \cite{SkinnerZhang}, Skinner and Zhang prove (among other results) a theorem similar to our Theorem A. 
More precisely, Theorem 1.1 of \emph{loc. cit.} proves  instances of the $p$-converse of the KGZ theorem in rank one, for  an elliptic curve 
with multiplicative reduction at a prime $p\geq{}5$. On the one hand, their result does not require the $p$-primary part of the Tate-Shafarevich group
to be finite, but only that  the  $p$-primary Selmer group of the elliptic curve has $\Z_{p}$-corank one.
On the other hand, together with the assumptions $1$ and $2$ of Theorem A,  the authors assume extra hypotheses in their statement.
For example, they assume that the mod-$p$ Galois representation $\overline{\rho}_{A,p}$ is not finite at $p$, that the Mazur-Tate-Teitelbaum
$L$-invariant $\mathscr{L}_{p}(A/\Q):=\frac{\log_{p}(q_{A})}{\mathrm{ord}_{p}(q_{A})}$ has $p$-adic valuation $1$
(where $q_{A}\in{}p\Z_{p}$ is the Tate period of $A/\Q_{p}$),  and require additional `$p$-indivisibility conditions'  
for the  Tamagawa factors of $A/\Q$.
(We refer to \emph{loc. cit.} for a precise list of the assumptions.) Finally, it is worth noting that our approach here
(cf. preceding Section)
is essentially different from that of \cite{SkinnerZhang}, where the authors extend the results and methods of \cite{WeiZ}
to the multiplicative setting.

\ \\
\emph{Acknowledgements.} We sincerely thank Massimo Bertolini for many  inspiring and interesting conversations,
and for his  encouragement during the preparation of this note. We thank Henri Darmon for his interest in this work.

\section[Hida Theory]{Hida Theory}
Fix for the rest of this note an elliptic curve $A/\Q$ having \emph{split} multiplicative reduction at an odd rational prime $p$.
Let  $N_{A}$ be the conductor of $A/\Q$, so that $N_{A}=Np$, with $p\nmid{}N$, and let 
$$f=\sum_{n=1}^{\infty}a_{n}q^{n}\in{}S_{2}(\Gamma_{0}(Np),\Z)^{\mathrm{new}}$$
be the weight-two newform attached to $A/\Q$
by modularity. Fix a  finite extension $L/\Q_p$, with ring of integers $\mathcal{O}_{L}$
and maximal ideal $\mathfrak{m}_{L}$,  
and an embedding $i_{p} : \overline{\Q}\hookrightarrow{}\overline{\Q}_{p}$,
under which we  identify $\overline{\Q}$ with a subfield of $\overline{\Q}_{p}$.
This also fixes a decomposition group 
$i_{p}^{\ast} : G_{\Q_{p}}\hookrightarrow{}G_{\Q}$ at $p$
(where $G_{F}:=\mathrm{Gal}(\overline{F}/F)$
for every field $F$).

 
\subsection{The Hida family $\I$}\label{hifa} 

Let $\Gamma:=1+p\Z_{p}$, let $\Z_{N,p}^{\times}:=\Gamma\times{}\lri{\Z/pN\Z}^{\times}$,
and let 
\[
      \mathcal{O}_{L}\llbracket{}\Z_{N,p}^{\times}\rrbracket[T_{n} : n\in{}\mathbf{N}]\twoheadrightarrow{}h^o(N,\mathcal{O}_{L})
\]
be Hida's universal $p$-ordinary Hecke algebra with $\mathcal{O}_{L}$-coefficients. Writing $\iw:=\mathcal{O}_{L}\llbracket{}\Gamma\rrbracket$, 
$h^o(N,\mathcal{O}_{L})$ is a finite, flat $\iw$-algebra \cite{H-2}.
Letting $\mathscr{L}:=\mathrm{Frac}(\iw)$, there is  a decomposition 
$h^o(N,\mathcal{O}_{L})\otimes_{\iw}\mathscr{L}=\prod_{j}\mathscr{K}_{j}$ as a finite product of
finite field extensions $\mathscr{K}_{j}/\mathscr{L}$.
Let $\mathscr{K}=\mathscr{K}_{j_{o}}$ be the \emph{primitive component} of $h^{o}(N,\mathcal{O}_{L})\otimes_{\iw}\mathscr{L}$
to which the $p$-ordinary newform $f$ \emph{belongs} \cite[Section 1]{H-2},
and let $\I$ be the integral closure of $\iw$ in the finite extension $\mathscr{K}/\mathscr{L}$.
For every $n\in{}\mathbf{N}$, write $\mathbf{a}_{n}\in{}\I$ for the image in $\I$ of the $n$th Hecke operator $T_{n}$.
By \cite[Corollary 1.5]{H-2}, there exists a unique morphism of $\mathcal{O}_{L}$-algebras
\[
              \phi_{f} : \I\longrightarrow{}\mathcal{O}_{L},
\]
such that $\phi_{f}(\mathbf{a}_{n})=a_{n}$ for every $n\in{}\mathbf{N}$; moreover, $\phi_{f}$ maps the image of 
$\Z_{N,p}^{\times}$ in $\I$ to $1$ (as $f$ has weight two and trivial neben type).
$\I$ is a normal local domain, finite and flat over Hida's weight algebra $\iw$. 
The domain $\I$ is called the \emph{(branch of the) Hida family passing through $f$}.
This terminology is justified as follows.

An \emph{arithmetic point} on $\I$ is a continuous morphism of $\mathcal{O}_{L}$-algebras $\psi : \I\fre{}\overline{\Q}_p$,
whose restriction to $\Gamma$ (with respect to the structural morphism $\iw\fre{}\I$) 
is of the form $\psi|_{\Gamma}(\gamma)=\gamma^{k_{\psi}-2}\cdot{}\chi_{\psi}(\gamma)$,
for an integer $k_{\psi}\geq{}2$ and a finite order character $\chi_{\psi}$ on $\Gamma$. 
We call $k_{\psi}$ and $\chi_{\psi}$  the \emph{weight} and \emph{(wild) character} of  $\psi$ respectively. 
Write $\xari(\I)$ for the set of arithmetic points on $\I$. Note that $\phi_{f}\in{}\xari(\I)$
is an arithmetic point of weight $2$ and trivial character.
Let  
$$\mathbf{f}=\sum_{n=1}^{\infty}\mathbf{a}_{n}\cdot{}q^{n}\in{}\I\llbracket{}q\rrbracket.$$
Then for every $\psi\in{}\xari(\I)$,
the \emph{specialisation of $\mathbf{f}$ at $\psi$}:  $$f_{\psi}:=\sum_{n=1}^{\infty}\psi(\mathbf{a}_n)\cdot{}
q^{n}\in{}S_{k_{\psi}}(\Gamma_{0}(Np^{c_{\psi}+1}),\xi_{\psi})$$
is a $p$-stabilised ordinary newform of tame level $N$, weight $k_{\psi}$ and character 
$\xi_{\psi}:=\chi_{\psi}\cdot{}\omega^{2-k_{\psi}}$.
Here 
$c_{\psi}\geq{}0$ is the smallest positive integer such that 
$\Gamma^{p^{c_{\psi}}}\subset{}\ker(\chi_{\psi})$, and  $\omega : \Z/(p-1)\Z\cong{}\F_p^{\times}
\fre{}\Z_p^{\times}$ is the Teichm\"uller character.  Moreover, we recover $f$ as the $\phi_{f}$-specialisation of $\mathbf{f}$, i.e.
\[
                 f_{\phi_{f}}:=\sum_{n=1}^{\infty}\phi_{f}(\mathbf{a}_{n})q^{n}=f.
\]
Let $\psi\in{}\xari(\I)$ be an arithmetic point. Denote by 
$K_{\psi}:=\mathrm{Frac}(\psi(\I))\subset{}\overline{\Q}_{p}$
the fraction field of $\psi(\I)$,
by $\mathfrak{m}_{\psi}$ its maximal ideal, and by
 $\F_{\psi}=\psi(\I)/\mathfrak{m}_{\psi}$ its residue field.
Let $\rho_{\psi} : G_{\Q}\fre{}\mathrm{GL}_{2}(K_{\psi})$ be the contragredient of the Deligne 
representation  associated with $f_{\psi}$,
and denote by $\overline{\rho}_{\psi} : G_{\Q}\fre{}\mathrm{Gal}(\F_{\psi})$
the semi-simplification of the reduction of $\rho_{\psi}$ modulo $\mathfrak{m}_{\psi}$.
Then $\overline{\rho}_{\psi}$ is unramified at every prime $\ell\nmid{}Np$,
and $\mathrm{Trace}(\overline{\rho}_{\psi}(\mathrm{Frob}_{\ell}))=\psi(\mathbf{a}_{\ell})\ (\mathrm{mod}\ \mathfrak{m}_{\psi})$
for every prime $\ell\nmid{}Np$, where $\mathrm{Frob}_{\ell}\in{}G_{\Q}$ is an arithmetic Frobenius at $\ell$.
Enlarging $L$ if necessary, one can assume $\F_{\psi}=\F:=\mathcal{O}_{L}/\mathfrak{m}_{L}$.
Then the representation $\overline{\rho}_{\psi}$ does not depend, up to isomorphism, on the 
arithmetic point $\psi$. Denote by $\overline{\rho}_{\mathbf{f}}$ this isomorphism class,
and assume throughout this note the following 

\begin{hyp}[$\mathbf{irr}$] \label{h2}$\overline{\rho}_{\mathbf{f}}$ is (absolutely) irreducible.
\end{hyp}

Under this assumption, 
it is known that $\mathbb{H}_{\mathbf{f}}:=\big(h^o(N,\mathcal{O}_{L})\otimes_{\iw}\I\big)\cap{}
\big(\mathscr{K}\times{}0)$ is a
free $\I$-module of rank one (where
we use the decomposition 
$h^{o}(N,\mathcal{O}_{L})\otimes_{\iw}\mathscr{L}=\mathscr{K}\times{}\prod_{j\not=j_{o}}\mathscr{K}_{j}$
mentioned above). 

\begin{remark}\label{remirrdist}
\emph{Taking $\psi=\phi_{f}$ in the discussion above, we deduce that $\overline{\rho}_{\mathbf{f}}$
is isomorphic to the $\F$-base change  of the mod-$p$
Galois representation $\overline{\rho}_{A,p}$ attached to the $p$-torsion submodule $A[p]$ of $A(\overline{\Q})$. 
(Indeed, Hypothesis $\ref{h2}$ is equivalent to require that  $\overline{\rho}_{A,p}$
is absolutely irreducible.) 
Since $A$ has split multiplicative reduction at $p$, Tate's theory gives us an isomorphism
(see \cite{Tate-2} or Chapter V of \cite{Sil-2})  $$\overline{\rho}_{\mathbf{f}}|_{G_{\Q_{p}}}\cong
\begin{pmatrix} \omega_{\mathrm{cy}} & \ast \\ 0 & 1\end{pmatrix},$$
where $\overline{\rho}_{\mathbf{f}}|_{G_{\Q_{p}}}$
is the restriction of $\overline{\rho}_{\mathbf{f}}$ to $G_{\Q_{p}}$
and $\omega_{\mathrm{cy}} : G_{\Q_{p}}\twoheadrightarrow{}\mathrm{Gal}(\Q_{p}(\mu_{p})/\Q_{p})\cong{}\F_{p}^{\times}$ is the mod-$p$ cyclotomic character. As $p\not=2$,
this implies  that $\overline{\rho}_{\mathbf{f}}$ is \emph{$p$-distinguished}, i.e. that condition $(\mathbf{dist})_{\mathbf{f}}$
in \cite{S-U} is satisfied.}
\end{remark}

\subsection{Hida's representations $T_{\mathbf{f}}$ and $\ppq$}\label{ppq}
Let $T_{\mathbf{f}}=(T_{\mathbf{f}},T_{\mathbf{f}}^+)$ be Hida's $p$-ordinary $\mathbb{I}$-adic representation attached to $\mathbf{f}$
(see, e.g. \cite{H-2}, \cite{S-U}). Thanks to our Hypothesis $\ref{h2}$,
$T_{\mathbf{f}}$ is a free $\mathbb{I}$-module of rank two, equipped with a continuous action 
of $G_{\Q}$ which is unramified at every prime $\ell\nmid{}Np$, and such that 
\begin{equation}\label{eq:ES}
         \det\lri{1-\mathrm{Frob}_{\ell}\cdot{}X|T_{\mathbf{f}}}=1-\mathbf{a}_{\ell}\cdot{}X+\ell[\ell]\cdot{}X^{2}
\end{equation}
for every $\ell\nmid{}Np$. Here $\mathrm{Frob}_{\ell}=\mathrm{frob}_{\ell}^{-1}$ is an arithmetic Frobenius at $\ell$
and $[\cdot{}] : \Z_{N,p}^{\times}\subset{}\mathcal{O}_{L}\llbracket\Z_{N,p}^{\times}\rrbracket\fre{}\I$ is the structural morphism. 
Write
$\chi_{\mathrm{cy},N} : G_{\Q}\twoheadrightarrow{}
\mathrm{Gal}(\Q(\mu_{Np^{\infty}})/\Q)\cong{}\Z_{N,p}^{\times}=\Gamma
\times{}\lri{\Z/Np\Z}^{\times}$,
$\chi_{\mathrm{cy}} : G_{\Q}\twoheadrightarrow{}\Z_{p}^{\times}$
for the $p$-adic cyclotomic character (i.e. the composition of $\chi_{\mathrm{cy},N}$ with projection to $\Z_{p}^{\times}=\Gamma
\times{}\lri{\Z/p\Z}^{\times}$)
and  $\kappa_{\mathrm{cy}} : G_{\Q_{p}}\fre{}\Gamma$ for the composition of $\chi_{\mathrm{cy}}$
with projection to principal units. Then $[\chi_{\mathrm{cy}}]=[\kappa_{\mathrm{cy}}]=[\chi_{\mathrm{cy},N}]$
as $\I^{\times}$-valued characters on $G_{\Q}$
(since $f$ has trivial neben type).
In particular the determinant  representation of $T_{\mathbf{f}}$ is given by
\begin{equation}\label{eq:hghghgdet}
               \det_{\mathbb{I}}T_{\mathbf{f}}\cong{}\I\big(\chi_{\mathrm{cy}}\cdot{}[\kappa_{\mathrm{cy}}]\big).
\end{equation}
$T_{\mathbf{f}}^+$ is an $\I$-direct summand of $T_{\mathbf{f}}$
of rank one, which is invariant under the action of the decomposition group $G_{\Q_p}\hookrightarrow{}G_{\Q}$ determined by $i_{p}$.
Moreover, $T_{\mathbf{f}}^-:=T_{\mathbf{f}}/T_{\mathbf{f}}^+$ is an unramified $G_{\Q_p}$-module, and the 
Frobenius  
$\mathrm{Frob}_p\in{}G_{\Q_p}/I_{\Q_{p}}$ acts on it via multiplication by 
the $p$-th Fourier coefficient $\mathbf{a}_p\in{}\mathbb{I}^{\times}$ of $\mathbf{f}$. In other words
\begin{equation}\label{eq:addequation}
                 T_{\mathbf{f}}^{+}\cong{}\I\Big(\mathbf{a}_{p}^{\ast-1}\cdot{}\chi_{\mathrm{cy}}\cdot{}[\kappa_{\mathrm{cy}}]\Big);\ \ 
                 T_{\mathbf{f}}^{-}\cong{}\I\big(\mathbf{a}_{p}^{\ast}\big)
\end{equation}
as $\I[G_{\Q_{p}}]$-modules, where $\mathbf{a}_{p}^{\ast} : G_{\Q_{p}}\twoheadrightarrow{}G_{\Q_{p}}/I_{\Q_{p}}
\fre{}\I^{\times}$ is the unramified character sending $\mathrm{Frob}_{p}$ to $\mathbf{a}_{p}$,
and we write again $\kappa_{\mathrm{cy}} : G_{\Q_{p}}\twoheadrightarrow{}\mathrm{Gal}(\Q_{p}(\mu_{p^{\infty}})/\Q_{p})
\cong{}\Z_{p}^{\times}\twoheadrightarrow{}\Gamma$ for  the $p$-adic cyclotomic character on $G_{\Q_{p}}$. 

Given an arithmetic point $\psi\in{}\xari(\I)$, let $V_{\psi}$ be the contragredient of the $p$-adic Deligne representation attached to 
the eigenform $f_{\psi}$: it is a two-dimensional vector space over $K_{\psi}=\mathrm{Frac}(\I/\mathrm{ker}(\psi))$,
equipped with a continuous $K_{\psi}$-linear action of $G_{\Q}$ which is unramified at every prime $\ell\nmid{}Np$,
and such that the trace of $\mathrm{Frob}_{\ell}$ acting on $V_{\psi}$ equals the $\ell$th  Fourier coefficient $\psi(\mathbf{a}_{\ell})=a_{\ell}(f_{\psi})$
of $f_{\psi}$, for every $\ell\nmid{}Np$. As proved by Ribet, $V_{\psi}$ is an  absolutely irreducible $G_{\Q}$-representation,
so that the Chebotarev density theorem, together with the Eichler-Shimura relations $(\ref{eq:ES})$ tell us that there exists an isomorphism 
of $K_{\psi}[G_{\Q}]$-modules
\begin{equation}\label{eq:ESdis}
                    T_{\mathbf{f}}\otimes_{\I,\psi}K_{\psi}\cong{}V_{\psi}.
\end{equation}
In other words, $T_{\mathbf{f}}$ interpolates the contragredients of the Deligne representations of the classical 
specialisations of the Hida family $\mathbf{f}$. 
(Note: $T_{\mathbf{f}}$ is the \emph{contragredient} of the representation denoted 
by the same symbol in \cite{S-U}.)

Together with the representations $T_{\mathbf{f}}$, we are  particularly interested in a certain self-dual twist $\ppq$ of it,
defined as follows. Define the \emph{critical character}
\[
           [\chi_{\mathrm{cy}}]^{1/2}=[\kappa_{\mathrm{cy}}]^{1/2} : G_{\Q}
           \twoheadrightarrow{}\mathrm{Gal}(\Q(\mu_{p^{\infty}})/\Q)\cong{}\Z_{p}^{\times}
           \twoheadrightarrow{}\Gamma\stackrel{\sqrt{\cdot}}{\longrightarrow}\Gamma\stackrel{[\cdot{}]}{\longrightarrow}\I^{\times},
\]
where the isomorphism is given by the $p$-adic cyclotomic character $\chi_{\mathrm{cy}}$.
(As $p\not=2$ by assumption, 
$\Gamma=1+p\Z_{p}$ is uniquely $2$-divisible, e.g. by Hensel's Lemma, so that $\sqrt{\cdot{}} : \Gamma\cong{}\Gamma$
is defined.) Let
\[
                  \ppq:=T_{\mathbf{f}}\otimes_{\I}[\chi_{\mathrm{cy}}]^{-1/2}\in{}_{\I[G_{\Q}]}\mathrm{Mod};\ \ 
                  \ppq^{\pm}:=T_{\mathbf{f}}^{\pm}\otimes_{\I}[\chi_{\mathrm{cy}}]^{-1/2}\in{}_{\I[G_{\Q_{p}}]}\mathrm{Mod},
\]
where we write for simplicity $[\chi_{\mathrm{cy}}]^{-1/2}$ for the inverse of $[\chi_{\mathrm{cy}}]^{1/2}$.
By $(\ref{eq:hghghgdet})$, $\ppq$ satisfies the crucial property:
\[
                  \det_{\I}\ppq\cong{}\I(1),
\]
i.e. the determinant representation of $\ppq$ is given by the $p$-adic cyclotomic character. As explained in \cite{N-P}, this implies  that there exists a skew-symmetric morphism of $\I[G_{\Q}]$-modules
\[
                    \pi : \ppq\otimes_{\I}\ppq\longrightarrow{}\I(1),
\]
inducing by adjunction isomorphisms of $\I[G_{\Q}]$- and $\I[G_{\Q_{p}}]$-modules respectively:
\[
              \mathrm{adj}(\pi) : \ppq\cong{}\mathrm{Hom}_{\I}(\ppq,\I(1));\ \ \mathrm{adj}(\pi) : \ppq^{\pm}\cong{}\mathrm{Hom}_{\I}(\ppq^{\mp},\I(1)).
\]
Let $\xari(\I)^{\prime}$ be the set of arithmetic points $\psi$ with trivial character and weight $k_{\psi}\equiv{}2\ (\mathrm{mod}\ 2(p-1))$.
Given $\psi\in{}\xari(\I)^{\prime}$, we have $\psi\circ{}[\chi_{\mathrm{cy}}]^{-1/2}(\mathrm{Frob}_{\ell})=\ell^{1-k_{\psi}/2}$
for every $\ell\nmid{}Np$. Equation
$(\ref{eq:ESdis})$ then gives: for every arithmetic point $\psi\in{}\xari(\I)^{\prime}$, 
there exists an isomorphism of $K_{\psi}[G_{\Q}]$-modules
\[
                           \ppq\otimes_{\I,\psi}K_{\psi}\cong{}V_{\psi}(1-k_{\psi}/2).
\]
In particular, $\ppq$ interpolates
the  family of self-dual, critical twists $V_{\psi}(1-k_{\psi}/2)$, for $\psi\in{}\xari(\I)^{\prime}$.

Let $v$ be a prime of $\overline{\Q}$ dividing $p$, associated with an embedding $i_{v} : \overline{\Q}
\hookrightarrow{}\overline{\Q}_{p}$. Write $i_{v}^{\ast} : G_{\Q_{p}}\hookrightarrow{}
G_{\Q}$ for the embedding determined by $i_{v}$, and $G_{v}:=i_{v}^{\ast}(G_{\Q_{p}})$
for the corresponding decomposition group at $v$.
Let $M_{\mathbf{f}}$ denote either $T_{\mathbf{f}}$ or $\mathbb{T}_{\mathbf{f}}$. 
Set $M_{\mathbf{f},v}^{\pm}:=M_{\mathbf{f}}^{\pm}\in{}_{\I[G_{\Q_{p}}]}\mathrm{Mod}$,
which we consider as $\I[G_{v}]$-modules via $i_{v}^{\ast}$.
Then there is a short exact sequence of $\I[G_{v}]$-modules 
\begin{equation}\label{eq:exact for general v}
            0\fre{}M_{\mathbf{f},v}^{+}\fre{i_{v}^{+}}M_{\mathbf{f}}\fre{p_{v}^{-}}M_{\mathbf{f},v}^{-}\fre{}0,
\end{equation}
where $i_{v}^{+}$ and $p_{v}^{-}$ are defined as follows. 
Fix $\alpha_{v}\in{}G_{\Q_{p}}$ and $\beta_{v}\in{}G_{\Q}$
such that $i_{v}=\alpha_{v}\circ{}i_{p}\circ{}\beta_{v}$. Then one sets
$i_{v}^{+}:=\beta_{v}^{-1}\circ{}i^{+}\circ{}\alpha_{v}^{-1}$ and $p_{v}^{-}:=\alpha_{v}\circ{}p^{-}\circ{}\beta_{v}$,
where $i^{+} : M_{\mathbf{f}}^{+}\subset{}M_{\mathbf{f}}$
and $p^{-} : M_{\mathbf{f}}\twoheadrightarrow{}M_{\mathbf{f}}^{-}$
denote the inclusion and  projection respectively.

\section[Skinner-Urban]{The theorem of Skinner-Urban }
The aim of this section is to state the main result of \cite{S-U} in our setting. In order to do that,
we recall  Skinner-Urban's construction 
of a three-variable $p$-adic $L$-function attached to $\mathbf{f}$ and a suitable quadratic imaginary field, 
and we  introduce the  Greenberg-style Selmer groups attached to the Hida family $\mathbf{f}$.

\subsection{Cyclotomic $p$-adic $L$-functions}\label{cpl}
For every $\psi\in{}\xari(\I)$, write $\mathcal{O}_{\psi}:=\psi(\I)$. 
Let $\Q_{\infty}/\Q$ be the $\Z_p$-extension of 
$\Q$, let $G_{\infty}:=\mathrm{Gal}(\Q_{\infty}/\Q)$,
and write $\iw_{\psi}^{\mathrm{cy}}:=\mathcal{O}_{\psi}\llbracket{}G_{\infty}\rrbracket$ for the cyclotomic Iwasawa algebra over 
$\mathcal{O}_{\psi}$.
Let $\psi\in{}\xari(\I)$, let $\epsilon$
be a quadratic Dirichlet character of conductor $C_{\epsilon}$ coprime with $Np$,
and let $S$ be a finite set of rational primes.
We say that an Iwasawa function $\mathcal{L}^{S}_{\epsilon}(f_{\psi})\in{}\iw_{\psi}^{\mathrm{cy}}$
is an $S$-primitive \emph{(cyclotomic) $p$-adic $L$-function of $f_{\psi}\otimes{}\epsilon$} if it satisfies the following interpolation property.
For every finite order character $\chi\in{}G_{\infty}\fre{}\overline{\Q}_{p}^{\ast}$ of conductor $p^{c_{\chi}}$
and every integer $1\leq{}j\leq{}k_{\psi}-1$:
\begin{align}\label{eq:intcp}
          \chi_{\mathrm{cy}}^{j-1}\chi{}\Big(\mathcal{L}^{S}_{\epsilon}(f_{\psi})\Big)= 
          \psi(\mathbf{a}_p)^{-c_{\chi}} & \cdot{} 
          \lri{1-\frac{\omega^{1-j}\epsilon\chi(p)\cdot{}p^{j-1}}{\psi(\mathbf{a}_p)}}
          \times{} \\ 
  \times    &    \frac{\lri{p^{c_{\chi}}C_{\epsilon}}^{j}{}(j-1)!\cdot{}L^{S\backslash{}\{p\}}(f_{\psi},\omega^{j-1}\chi^{-1}\epsilon,j)}
          {(-2\pi{}i)^{j-1}G\lri{\omega^{j-1}\chi^{-1}\epsilon}\cdot{}\Omega_{f_{\psi}}^{\mathrm{sgn}(\epsilon)\cdot{}(-1)^{j-1}}}
          \in{}\mathcal{O}_{\psi},    \nonumber
\end{align}
where the notations are as follows.
$L(f_{\psi},\mu,s)=L^{\emptyset}(f_{\psi},\mu,s)$ denotes the analytic continuation of the complex Hecke $L$-series 
$L(f_{\psi},\mu,s):=\sum_{n=1}^{\infty}\mu(n)\frac{\psi(\mathbf{a}_{n})}{n^{s}}=\prod_{\ell}E_{\ell}(f_{\psi}\otimes{}\mu,\ell^{-s})^{-1}$
of $f_{\psi}$ twisted by $\mu$;
for every finite set $\Sigma$ of rational primes, 
$L^{\Sigma}(f_{\psi},\mu,s):=\prod_{\ell\in{}\Sigma}E_{\ell}(f_{\psi}\otimes{}\mu,\ell^{-s})\cdot{}L(f_{\psi},\mu,s)$.
$G(\mu)$ denotes the Gauss sum of the character $\mu$.
Finally, $\Omega_{f_{\psi}}^{\pm}$ are canonical periods of $f_{\psi}$, as defined, e.g. in \cite{S-U}.
We recall that $\Omega_{f_{\psi}}^{\pm}$ is an element of $\C^{\times}$, defined only up to multiplication by a $p$-adic 
unit in $\mathcal{O}_{\psi}$, and such that the quotient appearing in the second line of the equation above lies in the number 
field $\Q\lri{\psi(\mathbf{a}_n) : n\in{}\mathbf{N}}$ generated by the Fourier coefficients of $f_{\psi}$. 
Together with the Weierstra\ss{} preparation theorem, this implies that $\mathcal{L}_{\epsilon}^{S}(f_{\psi})$, if it exists, is unique up
to multiplication a unit in $\mathcal{O}^{\times}_{\psi}$.
For a proof of the existence,
see \cite[Chapter I]{M-T-T}.

\subsection{Skinner-Urban three variable $p$-adic $L$-functions}\label{sup-adic}
Let $K/\Q$ be a quadratic \emph{imaginary} field of (absolute) discriminant $D_{K}$, let $q_{K}\nmid{}6p$
be a rational prime which splits in $K$, and let $S$ be a finite set of finite primes of $K$. We assume  that the following hypothesis is satisfied.

\begin{hyp}\label{h3}
The data $(K,p,L,q_{K},S)$ satisfy the following assumptions:

\begin{itemize}
\item[$\bullet$] $D_{K}$ is coprime with $6Np$. 
\item[$\bullet$] $p$ splits in $K$.
\item[$\bullet$] $L/\Q_p$ contains the finite extension $\Q_p\lri{D_{K}^{1/2},(-1)^{1/2},1^{1/Np}}/\Q_p$. 
\item[$\bullet$] $S$ consists of all the  primes of $K$ which divide $q_{K}D_{K}Np$.
\end{itemize}
\end{hyp}

Let $\mathcal{K}/K$ be the $\Z_{p}^{2}$-extension of $K$.
Then $\mathcal{K}=K_{\infty}\cdot{}K_{\infty}^{-}$,
where $K_{\infty}$ (resp., $K_{\infty}^-$) is the cyclotomic (resp., anticyclotomic) $\Z_p$-extension of $K$.
Denote by $G_{\infty}:=\mathrm{Gal}(K_{\infty}/K)\cong{}\mathrm{Gal}(\Q_{\infty}/\Q)$ and $D_{\infty}:=\mathrm{Gal}(K_{\infty}^-/K)$
the Galois groups of $K_{\infty}/K$ and $K_{\infty}^{-}/K$ respectively, 
so that $\mathrm{Gal}(\mathcal{K}/K)\cong{}G_{\infty}\times{}D_{\infty}$, and let $\I_{\infty}:=\I\llbracket{}G_{\infty}\rrbracket$.
Section $12$ of \cite{S-U} constructs an element 
\[
         \mathcal{L}_{K}^{S}(\mathbf{f})\in{}
         \I\llbracket{}G_{\infty}\times{}D_{\infty}\rrbracket=\I_{\infty}\llbracket{}D_{\infty}\rrbracket{},
\]
satisfying the following property: given $\psi\in{}\xari(\I)$, write 
$\psi^{\mathrm{cy}} : \I\llbracket{}G_{\infty}\times{}D_{\infty}\rrbracket\fre{}
\iw_{\psi}^{\mathrm{cy}}=\psi(\I)\llbracket{}G_{\infty}\rrbracket$ for the morphism of $\mathcal{O}_{L}\llbracket{}G_{\infty}\rrbracket$-algebras  whose restriction to 
$\I$ is $\psi$, and s.t. $\psi^{\mathrm{cy}}(D_{\infty})=1$.
Moreover, fix canonical periods $\Omega_{\psi}^{\pm}:=\Omega_{f_{\psi}}^{\pm}$ for $f_{\psi}$.
Then, for every $\psi\in{}\xari(\I)$, there exists $\lambda_{\psi}\in{}\mathcal{O}_{\psi}^{\times}$ such that
\begin{equation}\label{eq:spesu}
       \psi^{\mathrm{cy}}\lri{\mathcal{L}_{K}^{S}(\mathbf{f})}=\lambda_{\psi}\cdot{}
       \mathcal{L}^{S}(f_{\psi})\cdot{}\mathcal{L}_{\epsilon_{K}}^{S}(f_{\psi}),
\end{equation}
where 
$\mathcal{L}^{S}(f_{\psi}):=\mathcal{L}^{S}_{1}(f_{\psi})$ (resp., 
$\mathcal{L}_{\epsilon_{K}}^{S}(f_{\psi})$) is an 
$S$-primitive cyclotomic $p$-adic $L$-function 
of $f_{\psi}$ (resp., of 
$f_{\psi}\otimes{}\epsilon_{K}$), computed with respect to the periods $\Omega_{\psi}^{\pm}$.
Here $\epsilon_{K} : \lri{\Z/D_{K}\Z}^{\times}\fre{}\overline{\Q}_{p}^{\ast}$ is the primitive quadratic character 
attached to $K/\Q$, and we write for simplicity $\mathcal{L}^{S}_{\ast}(f_{\psi}):=\mathcal{L}^{S_{o}}_{\ast}(f_{\psi})$,
where $S_{o}:=\{\ell\ \text{prime} : \ell|q_{K}D_{K}Np\}$ is the set of rational primes lying below the primes in $S$.
More precisely, such a $p$-adic $L$-function $\mathcal{L}_{K}^{S}(\mathbf{f})=\mathcal{L}_{K}^{S}(\mathbf{f};1_{\mathbf{f}})$ 
is attached to every generator $1_{\mathbf{f}}$ of the free rank-one $\I$-module $\mathbb{H}_{\mathbf{f}}$
(mentioned at the end of Section $\ref{hifa}$),
and it is a well defined element of $\I_{\infty}\llbracket{}D_{\infty}\rrbracket$ only
up to multiplication by a unit in $\I$.
We refer to \cite[Theorems 12.6 and 12.7 and Proposition 12.8]{S-U} for the proofs of these facts, and for the 
interpolation property characterizing $\mathcal{L}_{K}^{S}(\mathbf{f})$. 

\begin{remark}\emph{Recall that Hypothesis $\ref{h2}$ (denoted $(\mathbf{irred})_{\mathbf{f}}$ in \cite{S-U}) 
is in order, i.e. that the residual representation $\overline{\rho}_{\mathbf{f}}$ is assumed to be  (absolutely) irreducible. 
As explained in Remark $\ref{remirrdist}$, we also know that $\overline{\rho}_{\mathbf{f}}$ is $p$-distinguished, i.e. that 
condition $(\mathbf{dist})_{\mathbf{f}}$ in \cite{S-U} is satisfied. These two hypotheses are used by Skinner and Urban in their construction of
$\mathcal{L}_{K}^{S}(\mathbf{f})$
(cf. Section 3.4.5 and Theorems 12.6 and 12.7 of \cite{S-U}).}
\end{remark}

\subsection{Greenberg Selmer groups}\label{greeselgen} 
Let $F/\Q$ be a number field, and let $\mathcal{F}/F$ be a $\Z_{p}$-power extension of $F$,
i.e. $\mathrm{Gal}(\mathcal{F}/F)\cong{}\Z_{p}^{r}$ for some $r\geq{}0$.
Write
$\mathbb{I}_{\mathcal{F}}:=\mathbb{I}\llbracket
\mathrm{Gal}(\mathcal{F}/K)\rrbracket$ and 
\[
     T_{\mathbf{f}}(\mathcal{F}):=T_{\mathbf{f}}\otimes_{\mathbb{I}}\mathbb{I}_{\mathcal{F}}(\varepsilon_{\mathcal{F}}^{-1})\in{}
     _{\mathbb{I}_{\mathcal{F}}[G_{F}]}\mathrm{Mod},
\]     
where  $\varepsilon_{\mathcal{F}} : G_{F}\twoheadrightarrow{}\mathrm{Gal}(\mathcal{F}/F)\subset{}\mathbb{I}_{\mathcal{F}}^{\times}$
is the tautological representation.
Let $v$ of be a prime of $F$ dividing $p$,
associated with an embedding $i_{v} : \overline{\Q}\hookrightarrow{}\overline{\Q_{p}}$,
and let $i_{v}^{\ast} : G_{F_{v}}\hookrightarrow{}G_{F}$
denote the corresponding decomposition group at $v$. Define  
\[
           T_{\mathbf{f}}(\mathcal{F})_{v}^{\pm}:=T_{\mathbf{f},v}^{\pm}\otimes_{\I}\I_{\mathcal{F}}
           (\varepsilon_{\mathcal{F},v}^{-1})\in{}_{\I_{\mathcal{F}}[G_{F_{v}}]}\mathrm{Mod},
\]
where $\varepsilon_{\mathcal{F},v}:=\varepsilon_{\mathcal{F}}\circ{}i_{v}^{\ast} : G_{F_{v}}\fre{}\I_{\mathcal{F}}^{\times{}}$.
The exact sequence $(\ref{eq:exact for general v})$ then 
induces a short exact sequence of $\I_{\mathcal{F}}[G_{F_{v}}]$-modules 
\begin{equation}\label{eq:hhhnnn}
                 0\fre{}T_{\mathbf{f}}(\mathcal{F})_{v}^{+}\fre{i_{v}^{+}}T_{\mathbf{f}}(\mathcal{F})\fre{p_{v}^{-}}T_{\mathbf{f}}(\mathcal{F})_{v}^{-}\fre{}0.
\end{equation}

Let $S$ be a finite set of primes of $F$, containing all the prime divisors of $NpD_{F}$
(where $D_{F}:=\mathrm{disc}(F/\Q)$ is the discriminant of $F/\Q$), and let 
$G_{F,S}:=\mathrm{Gal}(F_{S}/F)$ be the Galois group of the maximal algebraic extension $F_{S}/F$
which is unramified at every finite  prime $v\notin{}S$ of $F$.
As $\mathcal{F}/F$ (being a  $\Z_{p}$-power extension) is unramified outside $p$, 
$T_{\mathbf{f}}(\mathcal{F})$ is unramified at every finite prime $v\notin{}S$ of $F$,
i.e. $T_{\mathbf{f}}(\mathcal{F})$ is a $\I_{\mathcal{F}}[G_{F,S}]$-module.  
Let
$\mathfrak{a}\in{}\mathrm{Spec}(\I_{\mathcal{F}})$,
and write $\mathbb{I}_{\mathcal{F}}^{\ast}:=\Hom{\mathrm{cont}}(\mathbb{I}_{\mathcal{F}},\divp)$ for the Pontrjagin dual of 
$\mathbb{I}_{\mathcal{F}}$,
so that $\I_{\mathcal{F}}^{\ast}[\mathfrak{a}]$ is the Pontrjagin dual of $\I_{\mathcal{F}}/\mathfrak{a}$.
Define the (discrete) \emph{non-strict Greenberg Selmer group}:
\begin{equation}\label{eq:seldef}
     \mathrm{Sel}_{\mathcal{F}}^{S}(\mathbf{f},\mathfrak{a}):=\ker
     \lri{H^{1}\big(G_{F,S},T_{\mathbf{f}}(\mathcal{F})\otimes_{\mathbb{I}_{\mathcal{F}}}\mathbb{I}_{\mathcal{F}}^{\ast}[\mathfrak{a}]\big)\stackrel{}{\longrightarrow{}}
     \prod_{v|p}H^{1}\big(I_{v},T_{\mathbf{f}}(\mathcal{F})_{v}^-\otimes_{\mathbb{I}_{\mathcal{F}}}\mathbb{I}_{\mathcal{F}}^{\ast}[\mathfrak{a}]\big)
       }
\end{equation}
where  
$I_{v}=I_{F_{v}}\subset{}G_{F_{v}}$ is
the inertia subgroup and the arrow is defined by $\prod_{v|p}p_{v\ast}^{-}\circ{}\mathrm{res}_{v}$,
$p_{v\ast}^{-}$ being the morphism induced in cohomology by  $p_{v}^{-} : T_{\mathbf{f}}(\mathcal{F})\twoheadrightarrow{}T_{\mathbf{f}}(\mathcal{F})_{v}^{-}$.
It is a cofinitely generated $\mathbb{I}_{\mathcal{F}}/\mathfrak{a}$-module, i.e. its Pontrjagin dual
\[
              X_{\mathcal{F}}^{S}(\mathbf{f},\mathfrak{a}):=\Hom{\mathbb{I}_{\mathcal{F}}}
              \lri{\mathrm{Sel}_{\mathcal{F}}^{S}(\mathbf{f},\mathfrak{a}),\mathbb{I}_{\mathcal{F}}^{\ast}[\mathfrak{a}]}
              \cong{}\Hom{\Z_{p}}\Big(\mathrm{Sel}_{\mathcal{F}}^{S}(\mathbf{f},\mathfrak{a}),\divp\Big)
\]
is a finitely-generated $\mathbb{I}_{\mathcal{F}}/\mathfrak{a}$-module. If $\mathfrak{a}=0$, write more simply
$$\mathrm{Sel}_{\mathcal{F}}^{S}(\mathbf{f}):=\mathrm{Sel}_{\mathcal{F}}^{S}(\mathbf{f},0);\ \ 
X_{\mathcal{F}}^{S}(\mathbf{f}):=X_{\mathcal{F}}^{S}(\mathbf{f},0).$$
By construction there are  natural morphisms of $\I_{\mathcal{F}}/\mathfrak{a}$-modules
\begin{equation}\label{eq:controleq}
     \mathrm{Sel}_{\mathcal{F}}^{S}(\mathbf{f},\mathfrak{a})\fre{}
     \mathrm{Sel}_{\mathcal{F}}^{S}(\mathbf{f})[\mathfrak{a}];\ \   
      X_{\mathcal{F}}^{S}(\mathbf{f})\otimes_{\I_{\mathcal{F}}}\I_{\mathcal{F}}/\mathfrak{a}
            \fre{}X_{\mathcal{F}}^{S}(\mathbf{f},\mathfrak{a}).
\end{equation}

Since $\I$ is a normal domain,
so is $\I_{\mathcal{F}}\cong{}\I\llbracket{}X_{1},\dots,X_{r}\rrbracket$ (with $\mathrm{Gal}(\mathcal{F}/F)\cong{}\Z_{p}^{r}$). 
Write $\mathrm{Ch}_{\mathcal{F}}^{S}(\mathbf{f})\subset{}\I_{\mathcal{F}}$ for the characteristic ideal of the $\I_{\mathcal{F}}$-module 
$X^{S}_{\mathcal{F}}(\mathbf{f})$ (cf. Section 3 of \cite{S-U}):
\[
      \mathrm{Ch}_{\mathcal{F}}^{S}(\mathbf{f}):=\left\{ x\in{}\I_{\mathcal{F}} : 
      \mathrm{ord}_{\mathfrak{a}}(x)\geq{}\mathrm{length}_{\mathfrak{a}}\lri{X^{S}_{\mathcal{F}}(\mathbf{f})},\ 
      \text{for every} \ \mathfrak{a}\in{}\mathrm{Spec}(\I_{\mathcal{F}})\ \text{s.t.} \ \mathrm{height}(\mathfrak{a})=1\right\}.
\]
Here $\mathrm{ord}_{\mathfrak{a}} : \mathrm{Frac}(\I_{\mathcal{F}})\fre{}\Q\cup\{\infty\}$ is the (normalised) discrete valuation attached to the 
height-one prime $\mathfrak{a}$, and $\mathrm{length}_{\mathfrak{a}} : \lri{_{\I_{\mathcal{F}}}\mathrm{Mod}}_{\mathrm{ft}}
\fre{}\Z\cup\{\infty\}$
is defined by sending a finite $\I_{\mathcal{F}}$-module $M$ to the 
length over $\lri{\I_{\mathcal{F}}}_{\mathfrak{a}}$ of the localization $M_{\mathfrak{a}}$ of $M$ at $\mathfrak{a}$.

\begin{remark}\emph{Assume that $\mathcal{F}/F$ contains the cyclotomic $\Z_{p}$-extension 
$F_{\infty}\subset{}F(\mu_{p^{\infty}})$ of $F$. Thanks to the work of Kato \cite{Kateul}, we know that  
$X_{\mathcal{F}}^{S}(\mathbf{f})$ is a \emph{torsion} $\mathbb{I}_{\mathcal{F}}$-module (see also Section $3$ of \cite{S-U}), 
so that $\mathrm{Ch}_{\mathcal{F}}^{S}(\mathbf{f})$ is a \emph{non-zero} divisorial ideal
(which is principal if $\I$ is a unique factorization domain).}
\end{remark}

\subsection{The main result of \cite{S-U}}\label{mainsusec} Let $(K,p,L,q_{K},S)$
be as in Section $\ref{sup-adic}$, and assume (as in \emph{loc. cit.})
that this data satisfies Hypothesis $\ref{h3}$. In particular, $K/\Q$ is an \emph{imaginary} quadratic field 
in which $p$ \emph{splits}. Let
$\mathcal{K}=K_{\infty}\cdot{}K_{\infty}^{-}$ be the $\Z_{p}^{2}$-extension of $K$,
and let $\mathcal{L}_{K}^{S}(\mathbf{f})\in{}\I_{\mathcal{K}}=\I\llbracket\mathrm{Gal}(\mathcal{K}/K)\rrbracket$
be Skinner-Urban's three variable $p$-adic $L$-function.
Together with Hypotheses $\ref{h2}$ and $\ref{h3}$, we have to consider:

\begin{hyp}[$\mathbf{ram}$]\label{h4} Decompose  $N=N^{+}N^{-}$, where $N^{+}=N^{+}_{K}$ (resp., $N^{-}=N_{K}^{-}$)
is divided precisely by the prime divisors of $N=N_{A}/p$ which are split (resp., inert) in $K$. Then:
\begin{itemize}
\item[$\bullet$] $N^{-}$ is square-free,
and has an \emph{odd} number of prime divisors.
\item[$\bullet$] The residual representation $\overline{\rho}_{\mathbf{f}}$ is ramified at every prime $\ell\Vert{}N^{-}$.
\end{itemize}
\end{hyp}

The following fundamental and deep result is Theorem 3.26 of \cite{S-U}. 

\begin{theo}[Skinner-Urban \cite{S-U}]\label{msu}
Assume that Hypotheses $\ref{h2}$, $\ref{h3}$ and $\ref{h4}$ hold. Then
$$\mathrm{Ch}_{\mathcal{K}}^S(\mathbf{f})\subseteq{}
\lri{\mathcal{L}_{K}^{S}(\mathbf{f})}.$$
\end{theo}

\section[Central Critical]{Restricting to the central critical line}\label{rccl}
The aim of this section is to specialise Skinner-Urban's result to the \emph{(cyclotomic) central critical line} in the \emph{weight-cyclotomic}
space. More precisely, we use Theorem $\ref{msu}$ to compare the order of vanishing of a  certain \emph{central-critical $p$-adic $L$-function}
of the weight variable with the structure of a certain \emph{central-critical Selmer group} attached to Hida's half-twisted 
representation $\ppq$. 

In this section, the notations and hypotheses of Section  $\ref{mainsusec}$ are in order.
In particular, we assume that Hypotheses $\ref{h2}$, $\ref{h3}$ and $\ref{h4}$ are satisfied. 

\subsection{The (localised) Hida family}\label{lochidafam} Let $\phi_{f}\in{}\mathcal{X}^{\mathrm{arith}}(\mathbb{I})$ be the arithmetic point 
of weight $2$ and trivial character introduced in Section $\ref{hifa}$, with associated $p$-stabilised weight-two newform
$f\in{}S_{2}(\Gamma_{0}(Np),\Z)^{\mathrm{new}}$.
Write $\mathfrak{p}_{f}:=\ker\lri{\phi_{f}}\in{}\mathrm{Spec}(\mathbb{I})$. By \cite[Corollary 1.4]{H-2},  the localisation 
$\mathbb{I}_{\mathfrak{p}_{f}}$ is a discrete valuation ring, unramified over the localisation of $\iw=\mathcal{O}_{L}\llbracket{}\Gamma\rrbracket$
at the prime $\widetilde{\mathfrak{p}}=\mathfrak{p}_{f}\cap{}\iw$. Fix a topological generator 
$\gamma_{\mathrm{wt}}\in{}\Gamma=1+p\Z_{p}$, and write $\varpi_{\mathrm{wt}}:=\gamma_{\mathrm{wt}}-1$.
Then $\varpi_{\mathrm{wt}}$ is a generator of the prime $\widetilde{\mathfrak{p}}$, so that 
\begin{equation}\label{eq:unif}
                      \mathfrak{p}_{f}\cdot{}\mathbb{I}_{\mathfrak{p}_{f}}=\varpi_{\mathrm{wt}}\cdot{}\mathbb{I}_{\mathfrak{p}_{f}},
\end{equation}
i.e.  $\varpi_{\mathrm{wt}}\in{}\iw$ is a uniformiser of the discrete valuation ring  $\mathbb{I}_{\mathfrak{p}_{f}}$.

Let $W\subset{}\Z_{p}$ be a non-empty open neighbourhood of $2$. 
Denote by $\mathscr{A}(W)\subset{}\overline{\Q}_{p}\llbracket{}k-2\rrbracket$
the subring of formal power series in $k-2$ which converge for every $k\in{}W$.
As explained in \cite{G-S} (see also \cite{N-P}), there exist an open neighbourhood $U=U_{f}\subset{}\Z_{p}$ of $2$,
and  a natural morphism (the \emph{Mellin transform centred at $\phi_{f}$}) 
$$\mathtt{M} : \mathbb{I}\longrightarrow{}\mathscr{A}(U),$$ characterised by the following properties: for every 
$x\in{}\mathbb{I}$ write $\mathtt{M}_{x}(k):=\mathtt{M}(x)(k)\in{}\mathscr{A}(U)$. Then:
$(i)$ for every $x\in{}\mathbb{I}$, $\mathtt{M}_{x}(2)=\phi_{f}(x)$ and $(ii)$ for every $\gamma\in{}\Gamma\subset{}\mathbb{I}^{\times}$,
$\mathtt{M}_{[\gamma]}(k)=\gamma^{k-2}:=\exp_{p}\lri{(k-2)\cdot{}\log_{p}(\gamma)}\in{}\mathscr{A}(\Z_{p})$
($[\cdot{}] : \iw\fre{}\I$ being the structural  morphism).
For every positive integer $n$, write $a_{n}(k):=\mathtt{M}(\mathbf{a}_{n})\in{}\mathscr{A}(U)$ for the image of the $n$-th Hecke operator 
$\mathbf{a}_{n}\in{}\mathbb{I}$ under $\mathtt{M}$, and consider the formal $q$-expansion with coefficients in $\mathscr{A}(U)$:
\[
                      f_{\infty}:=\sum_{n=1}^{\infty}a_{n}(k)q^{n}\in{}\mathscr{A}(U)\llbracket{}q\rrbracket.
\]
This is the `portion' of the Hida family $\mathbf{f}$ we are mostly interested in. More precisely, 
let $$U^{\mathrm{cl}}:=\big\{k\in{}U\cap{}\Z : k\geq{}2; k\equiv{}2\ \big(\mathrm{mod}\ 2(p-1)\big)\big\}$$
be the subset of \emph{classical points},
which is a dense subset of $U$. For every classical point $\kappa\in{}U^{\mathrm{cl}}$, the composition 
$\phi_{\kappa} : \mathbb{I}\stackrel{\mathtt{M}}{\longrightarrow{}}\mathscr{A}(U)\stackrel{\mathrm{ev}_{\kappa}\ }{\longrightarrow{}}\overline{\Q}_{p}$ (where $\mathrm{ev}_{\kappa}$ is  evaluation at $\kappa$) is an arithmetic point of weight $\kappa$ and  trivial character, and the 
\emph{weight-$\kappa$ specialisation} $f_{\kappa}:=f_{\phi_{\kappa}}=\sum_{n=1}^{\infty}a_{n}(\kappa)q^{n}\in{}S_{\kappa}(\Gamma_{0}(Np))$
is a $p$-ordinary normalised eigenform of weight $\kappa$ and level $\Gamma_{0}(Np)$. By construction: $f=f_{2}$.
Moreover, $N$ divides the conductor of $f_{\kappa}$ for every $\kappa\in{}U^{\mathrm{cl}}$ (and 
$f_{\kappa}$ is old at $p$ for  $\kappa>2$, i.e. $f_{\kappa}$ is the $p$-stabilisation of a newform of level $\Gamma_{0}(N)$ when $\kappa>2$ \cite{H-2}).

\subsection{The central critical $p$-adic $L$-function}
Let $\mathscr{A}(U\times{}\Z_{p}\times{}\Z_{p})\subset{}\overline{\Q}_{p}\llbracket
k-2,s-1,r-1\rrbracket$ be the subring of formal power series converging 
for every  $(k,s,r)\in{}U\times{}\Z_{p}\times{}\Z_{p}$.
Let $\chi_{\mathrm{cy}} : G_{\infty}\cong{}1+p\Z_{p}$ be the $p$-adic cyclotomic character,
and fix an isomorphism $\chi_{\mathrm{acy}} : D_{\infty}\cong{}1+p\Z_{p}$.
We can uniquely  extend the Mellin transform $\mathtt{M}$ to  a morphism of rings
\[
               \widetilde{\mathtt{M}} : \mathbb{I}\llbracket{}G_{\infty}\times{}D_{\infty}\rrbracket\longrightarrow{}\mathscr{A}(U\times{}\Z_{p}
               \times{}\Z_{p}),
\]  
by mapping  every $\sigma\in{}D_{\infty}$ (resp., $\sigma\in{}G_{\infty}$) to the analytic function on $\Z_{p}$
represented by the power series 
$\widetilde{\mathtt{M}}(\sigma):=\chi_{\mathrm{acy}}(\sigma)^{r-1}
=\exp_{p}\lri{(r-1)\cdot{}\log_{p}\big(\chi_{\mathrm{acy}}(\sigma)\big)}$
(resp., $\widetilde{\mathtt{M}}(\sigma):=\chi_{\mathrm{cy}}(\sigma)^{s-1}$). We then define the \emph{$S$-primitive 
analytic three-variable $p$-adic $L$-function of $f_{\infty}/K$}:
\[
                   L_{p}^{S}(f_{\infty}/K,k,s,r):=\widetilde{\mathtt{M}}\big(\mathcal{L}_{K}^{S}(\mathbf{f})\big)\in{}\mathscr{A}\lri{U\times{}\Z_{p}
                   \times{}\Z_{p}}.             
\]
In the rest of this note, the \emph{(cyclotomic) central critical line} $\ell^{\mathrm{cc}}:=\{(k,s,r)\in{}U\times{}\Z_{p}\times{}\Z_{p}
 : r=1;\ s=k/2\}$ will play a key role.
Let $\mathfrak{l}$ be a prime of $K$ contained in $S$, which does not divide $p$.
Let $\ell\not=p$ be the rational prime lying below it: $\mathfrak{l}\cap{}\Z=\ell\Z$.
Define the \emph{central critical $\mathfrak{\ell}$-Euler factor of $f_{\infty}/K$} as
\[
      E_{\ell}(f_{\infty}/K,k):=\lri{1-\frac{a_{\ell}(k)}{\dia{\ell}^{k/2}\omega(\ell)}+
      \frac{\mathbf{1}_{N}(\ell)}{\ell}}\cdot{}
      \lri{1-\frac{\epsilon_{K}(\ell)a_{\ell}(k)}{\dia{\ell}^{k/2}\omega(\ell)}+\frac{\mathbf{1}_{ND_{K}}(\ell)}
      {\ell}}\in{}\mathscr{A}(\Z_{p}),
\]
where $\dia{\ell}:=\omega(\ell)^{-1}\ell\in{}1+p\Z_{p}$ is the projection of $\ell$ to principal units and 
$\mathbf{1}_{M}$ denotes the trivial Dirichlet character modulo $M$, for every $M\in{}\N$.
Then 
\[
          E_{\ell}(f_{\infty}/K,\kappa)=E_{\ell}(f_{\kappa},\ell^{-\kappa/2})\cdot{}
          E_{\ell}(f_{\kappa}\otimes\epsilon_{K},\ell^{-\kappa/2})
\]
for every classical point $\kappa\in{}U^{\mathrm{cl}}$, where 
$E_{\ell}(\ast,X)$ is the $\ell$-th Euler factor of the eigenform $\ast$,
so that the Hecke $L$-series of $\ast$ is given by the product 
$L(\ast,s)=\prod_{q\ \text{prime}}E_{q}(\ast,q^{-s})^{-1}$ (cf. Section $\ref{cpl}$).
Define the \emph{central critical  $S$-Euler factors} of $f_{\infty}/K$ by
\[                
                E_{S}(f_{\infty}/K,k):=\prod_{\ell|q_{K}ND_{K}}E_{\ell}(f_{\infty}/K,k),
\]
where the product runs over the rational primes lying below a prime $\mathfrak{l}\nmid{}p$ of $S$
(cf. Hypothesis $\ref{h3}$).
One has 
$E_{\ell}(f_{\infty}/K,2)\not=0$ for every $\ell|ND_{K}q_{K}$,
so that, up to shrinking the $p$-adic disc $U$ if necessary, 
one can assume that $E_{S}(f_{\infty}/K,k)\in{}\mathscr{A}(U)^{\times}$. 
Define finally  the \emph{central critical $p$-adic $L$-function} of $f_{\infty}/K$:
\begin{equation}\label{eq:deflcc}
             L_{p}^{\mathrm{cc}}(f_{\infty}/K,k):=E_{S}(f_{\infty}/K,k)^{-1}\cdot{}
             L_{p}^{S}(f_{\infty}/K,k,k/2,1)\in{}\mathscr{A}(U).
\end{equation}
Note that, while the definition of $L_{p}^{S}(f_{\infty}/K,k,s,r)$ depends on the choice of the isomorphism 
$\chi_{\mathrm{acy}} : D_{\infty}\cong{}1+p\Z_{p}$, the analytic function  $L_{p}^{\mathrm{cc}}(f_{\infty}/K,k)$
is independent of this choice.

\subsection{The central critical Selmer group: a Control Theorem}\label{criticalgrsec} 
Fix topological generators $\gamma_{+}\in{}G_{\infty}$, $\gamma_{-}\in{}D_{\infty}$
and $\gamma_{\mathrm{wt}}\in{}\Gamma$, and write $\varpi_{?}:=\gamma_{?}-1$. 
We can (and will) assume that $\chi_{\mathrm{cy}}(\gamma_{+})=\gamma_{\mathrm{wt}}$,
where we write again $\chi_{\mathrm{cy}} : G_{\infty}\cong{}1+p\Z_{p}=\Gamma\subset{}\I^{\times}$
for the isomorphism induced by the $p$-adic cyclotomic character. Let
\[
               \Theta_{K}^{+} : \mathrm{Gal}(\mathcal{K}/K)=G_{\infty}\times{}D_{\infty}\twoheadrightarrow{}G_{\infty}
               \stackrel{\chi_{\mathrm{cy}}}{\cong}\Gamma\stackrel{\sqrt{\cdot{}}}{\longrightarrow}\Gamma\stackrel{[\cdot{}]}{\longrightarrow}\I^{\times}
\]  
be the \emph{cyclotomic central critical Greenberg character}. We can extend uniquely 
$\Theta_{K}^{+}$ to a morphism of $\I$-algebras, denoted again by the same symbol,
$\Theta_{K}^{+} : \I_{\mathcal{K}}\fre{}\I$. As easily seen, its kernel $\mathfrak{P}^{\mathrm{cc}}$ is given by
\[
              \mathfrak{P}^{\mathrm{cc}}:=\ker\big(\Theta_{K}^{+} : \I_{\mathcal{K}}\twoheadrightarrow{}\I\big)=
              (\varpi_{\mathrm{cc}},\varpi_{-})\cdot{}\I_{\mathcal{K}};
              \ \ \varpi_{\mathrm{cc}}:=[\gamma_{\mathrm{wt}}]-\gamma_{+}^{2}\in{}\I_{\mathcal{K}},
\] 
i.e. $\mathfrak{P}^{\mathrm{cc}}$ is generated by $\varpi_{-}$ and $\varpi_{\mathrm{cc}}$. In analogy with the definitions above,
we define the \emph{(cyclotomic) $S$-primitive central critical (non-strict) Greenberg Selmer group of $\mathbf{f}/K$} by
\[
              \mathrm{Sel}^{S,\mathrm{cc}}_{\Q_{\infty}}(\mathbf{f}/K):=
              \ker\lri{H^{1}(G_{K,S},\mathbb{T}_{\mathbf{f}}\otimes_{\I}\I^{\ast})
              \longrightarrow{}\prod_{v|p}H^{1}(I_{v},\mathbb{T}_{\mathbf{f},v}^{-}\otimes_{\I}\I^{\ast})}.
\] 
Here $\mathbb{T}_{\mathbf{f}}=(\mathbb{T}_{\mathbf{f}},\mathbb{T}_{\mathbf{f}}^{+})$ is Hida's half-twisted representation defined in Section 
$(\ref{ppq})$ and $S$ is as in Section $\ref{sup-adic}$.
Moreover, 
the arrow refers again to $\prod_{v|p}p_{v\ast}^{-}\circ{}\mathrm{res}_{v}$,
where $p_{v}^{-} : \mathbb{T}_{\mathbf{f}}\twoheadrightarrow{}\mathbb{T}_{\mathbf{f},v}^{-}$
is the projection introduced in equation $(\ref{eq:exact for general v})$
\footnote{We should keep in mind that the cyclotomic variable plays a non trivial role in the definition 
of Hida's half-twisted representation $\mathbb{T}_{\mathbf{f}}$. 
This explains the appearance of the subscript $\Q_{\infty}$  
in the notation  $\mathrm{Sel}^{S,\mathrm{cc}}_{\Q_{\infty}}(\mathbf{f}/K)$.}.
Denote by $X_{\Q_{\infty}}^{S,\mathrm{cc}}(\mathbf{f}/K)$ the Pontrjagin dual of $\mathrm{Sel}_{\Q_{\infty}}^{S,\mathrm{cc}}(\mathbf{f}/K)$:
\[
              X^{S,\mathrm{cc}}_{\Q_{\infty}}(\mathbf{f}/K):=\Hom{\Z_{p}}\lri{\mathrm{Sel}_{\Q_{\infty}}^{S,\mathrm{cc}}
              (\mathbf{f}/K),\divp}.
\]
With these notations, and the ones introduced in Section  $\ref{greeselgen}$, we have the following perfect control theorem.

\begin{proposition}\label{speccgr} There exists a canonical isomorphism of $\I$-modules
\[
                     X_{\mathcal{K}}^{S}(\mathbf{f})\otimes_{\I_{\mathcal{K}}}\I_{\mathcal{K}}\big/\mathfrak{P}^{\mathrm{cc}}\cong{}
                     X_{\Q_{\infty}}^{S,\mathrm{cc}}(\mathbf{f}/K).
\]
\end{proposition} 
\begin{proof} Let $\mathfrak{a}_{1}=(\varpi_{-})\in{}\mathrm{Spec}(\mathbb{I}_{\mathcal{K}})$ and  $\mathfrak{a}_{2}:=
(\varpi_{\mathrm{cc}})\in{}\mathrm{Spec}(\mathbb{I}_{K_{\infty}})$. (We remind that
$\mathcal{K}=K_{\infty}\cdot{}K_{\infty}^{-}$ is the $\Z_{p}^{2}$-extension of $K$ and  $K_{\infty}/K$ is the cyclotomic 
$\Z_{p}$-extension).
As $\mathbb{I}_{\mathcal{K}}/\mathfrak{a}_{1}\cong{}\mathbb{I}_{K_{\infty}}$
and 
$T_{\mathbf{f}}(\mathcal{K})/\mathfrak{a}_{1}\cong{}T_{\mathbf{f}}(K_{\infty})$:
\[
         T_{\mathbf{f}}(\mathcal{K})\otimes_{\mathbb{I}_{\mathcal{K}}}\mathbb{I}_{\mathcal{K}}^{\ast}[\mathfrak{a}_{1}]
         \cong{}T_{\mathbf{f}}(\mathcal{K})/\mathfrak{a}_{1}\otimes_{\mathbb{I}_{\mathcal{K}}/\mathfrak{a}_{1}}
         \mathbb{I}_{K_{\infty}}^{\ast}
         \cong{}T_{\mathbf{f}}(K_{\infty})\otimes_{\mathbb{I}_{K_{\infty}}}\mathbb{I}_{K_{\infty}}^{\ast},
\]
and similarly $T_{\mathbf{f}}(\mathcal{K})_{v}^{-}\otimes_{\I_{\mathcal{K}}}\I_{\mathcal{K}}^{\ast}[\mathfrak{a}_{1}]
\cong{}T_{\mathbf{f}}(K_{\infty})_{v}^{-}\otimes_{\I_{K_{\infty}}}\I_{K_{\infty}}^{\ast}$
for every $v|p$.  In particular $\mathrm{Sel}_{\mathcal{K}}^{S}(\mathbf{f},\mathfrak{a}_{1})$
is canonically isomorphic to $\mathrm{Sel}_{K_{\infty}}^{S}(\mathbf{f})$.
Moreover, by \cite[Proposition 3.9]{S-U}, the maps $(\ref{eq:controleq})$ induce  isomorphisms
\begin{equation}\label{eq:ggg}
            \mathrm{Sel}_{K_{\infty}}^{S}(\mathbf{f})
            \cong{}\mathrm{Sel}_{\mathcal{K}}^{S}(\mathbf{f})[\mathfrak{a}_{1}];\ \  \ 
            X_{\mathcal{K}}^{S}(\mathbf{f})\otimes_{\I_{\mathcal{K}}}\I_{\mathcal{K}}/\mathfrak{a}_{1}\cong{}X_{K_{\infty}}^{S}(\mathbf{f}).
\end{equation}
Similarly, $\Theta_{K}^{+}$ induces an isomorphism:  $\mathbb{I}_{K_{\infty}}/\mathfrak{a}_{2}\cong{}\mathbb{I}$, an isomorphism 
of $\I[G_{K,S}]$-modules: 
$T_{\mathbf{f}}(K_{\infty})/\mathfrak{a}_{2}\cong{}\mathbb{T}_{\mathbf{f}}$
and isomorphisms of $\I[G_{K_{v}}]$-modules: $T_{\mathbf{f}}(K_{\infty})_{v}^{\pm}/\mathfrak{a}_{2}\cong{}\mathbb{T}_{\mathbf{f},v}^{\pm}$
for every $v|p$.
(Indeed, write $\Theta_{K} : \I_{\infty}=\I_{K_{\infty}}\twoheadrightarrow{}\I$
for the `restriction' of $\Theta_{K}^{+}$ to $\I_{\infty}$.
Then $\Theta_{K}\circ{}\varepsilon_{K_{\infty}}^{-1}=[\chi_{\mathrm{cy}}]^{-1/2}$
on $G_{K,S}$,
so that 
$$T_{\mathbf{f}}(K_{\infty})/\mathfrak{a}_{2}\cong{}T_{\mathbf{f}}(K_{\infty})\otimes_{\I_{\infty},\Theta_{K}}\I=
T_{\mathbf{f}}\otimes_{\I}\I_{\infty}\big(\varepsilon_{K_{\infty}}^{-1}\big)\otimes_{\I_{\infty},\Theta_{K}}\I
\cong{}T_{\mathbf{f}}\otimes_{\I}[\chi_{\mathrm{cy}}]^{-1/2}=\mathbb{T}_{\mathbf{f}}.$$
The same argument justify the statement for the $\pm$-parts at a prime $v|p$.)
As above (i.e. retracing the definitions), this gives a canonical isomorphism of Selmer groups
\begin{equation}\label{eq:445}
                      \mathrm{Sel}_{\Q_{\infty}}^{S,\mathrm{cc}}(\mathbf{f}/K)\cong{}\mathrm{Sel}_{K_{\infty}}^{S}(\mathbf{f},\mathfrak{a}_{2}).
\end{equation}
Let us consider the following  commutative diagram with (tautological) exact rows:
\[\small\small
                \xymatrix{      0 \ar[r]   & \mathrm{Sel}_{K_{\infty}}^{S}(\mathbf{f},\mathfrak{a}_{2}) \ar[r] \ar[d]_{\alpha}
                                                & H^{1}(G_{K,S},T_{\mathbf{f}}(K_{\infty})\otimes_{\I_{K_{\infty}}}\I_{K_{\infty}}^{\ast}[\mathfrak{a}_{2}]) \ar[d]_{\beta}
                                                \ar[r] & \prod_{v|p}H^{1}(I_{v},T_{\mathbf{f}}(K_{\infty})_{v}^{-}
                                                \otimes_{\I_{K_{\infty}}}\I_{K_{\infty}}^{\ast}[\mathfrak{a}_{2}])
                                                \ar[d]_{\gamma} 
                                                \\
                                                0\ar[r] & \lri{\mathrm{Sel}_{K_{\infty}}^{S}(\mathbf{f})}[\mathfrak{a}_{2}] \ar[r]
                                                 &  \Big(H^{1}\big(G_{K,S},T_{\mathbf{f}}(K_{\infty})\otimes_{\I_{K_{\infty}}}\I_{K_{\infty}}^{\ast}\big)\Big)[\mathfrak{a}_{2}]   
                                                 \ar[r] & \lri{\prod_{v|p}
                                                 H^{1}\big(I_{v},T_{\mathbf{f}}(K_{\infty})_{v}^{-}
                                                \otimes_{\I_{K_{\infty}}}\I_{K_{\infty}}^{\ast}\big)
                                                }[\mathfrak{a}_{2}],
                                                            }
\]  
where the vertical maps are the natural ones induced by the inclusion $\I_{K_{\infty}}^{\ast}[\mathfrak{a}_{2}]\subset{}\I_{K_{\infty}}^{\ast}$
(cf. $(\ref{eq:controleq})$). We claim that $\alpha$ is an isomorphism of $\I$-modules:
\begin{equation}\label{eq:333999}
             \alpha : \mathrm{Sel}_{K_{\infty}}^{S}(\mathbf{f},\mathfrak{a}_{2})\cong{}\mathrm{Sel}_{K_{\infty}}^{S}(\mathbf{f})[\mathfrak{a}_{2}].
\end{equation}
The map $\beta$ sits into a short exact sequence 
(arising form $0\fre{}\I_{K_{\infty}}^{\ast}[\mathfrak{a}_{2}]\fre{}\I_{K_{\infty}}^{\ast}\fre{\varpi_{\mathrm{cc}}}\I^{\ast}_{K_{\infty}}\fre{}0$):
\begin{align*}
  0\fre{}     H^{0}(G_{K,S},T_{\mathbf{f}}(K_{\infty})\otimes_{\I_{K_{\infty}}}\I_{K_{\infty}}^{\ast})\big/\varpi_{\mathrm{cc}}
                 \fre{}  &
                 H^{1}(G_{K,S},T_{\mathbf{f}}(K_{\infty})\otimes_{\I_{K_{\infty}}}\I_{K_{\infty}}^{\ast}[\mathfrak{a}_{2}]) \\
   \stackrel{\beta}{\longrightarrow{}}H^{1}(G_{K,S},T_{\mathbf{f}}(K_{\infty})&\otimes_{\I_{K_{\infty}}}\I_{K_{\infty}}^{\ast})[\mathfrak{a}_{2}]\fre{}0.
\end{align*}
Hypotheses $\ref{h2}$ and  $\ref{h3}$ 
imply that the restriction of $\overline{\rho}_{\mathbf{f}}$ to $G_{K}$ is irreducible. Then the first $H^{0}$ vanishes,
and $\beta$ is an isomorphism. By the Snake Lemma, the morphism $\alpha$ is injective, and its cockerel is a 
sub-module of $\ker(\gamma)$. To prove the claim $(\ref{eq:333999})$ it is then sufficient to show that
\begin{equation}\label{eq:nnn777}
                        \ker(\gamma)=0.
\end{equation}
Looking again at the exact $I_{v}$-cohomology sequence arising from  $0\fre{}\I_{K_{\infty}}^{\ast}[\mathfrak{a}_{2}]\fre{}
\I_{K_{\infty}}^{\ast}\fre{\varpi_{\mathrm{cc}}}\I_{K_{\infty}}^{\ast}\fre{}0$, we have
\begin{equation}\label{eq:nbnbnbnbnbnb}
             \ker(\gamma)\cong{}\prod_{v|p}
             H^{0}(I_{v},T_{\mathbf{f}}(K_{\infty})_{v}^{-}\otimes_{\I_{K_{\infty}}}\I_{K_{\infty}}^{\ast})
             \otimes_{\I_{K_{\infty}}}\I_{K_{\infty}}/\varpi_{\mathrm{cc}}
             .
\end{equation}
Note that $T_{\mathbf{f}}(K_{\infty})_{v}^{-}\otimes_{\I_{K_{\infty}}}\I_{K_{\infty}}^{\ast}
\cong{}\I_{K_{\infty}}^{\ast}\big(\mathbf{a}_{p}^{\ast}\cdot{}\varepsilon_{K_{\infty}}^{-1}\big)$ (cf. Sec. $\ref{ppq}$).
Since  
$\I_{K_{\infty}}/(\gamma_{+}-1)\I_{K_{\infty}}\cong{}\I$, one finds  
\[
      H^{0}(I_{v},T_{\mathbf{f}}(K_{\infty})_{v}^{-}\otimes_{\I_{K_{\infty}}}\I_{K_{\infty}}^{\ast})=
\I_{K_{\infty}}^{\ast}(\mathbf{a}_{p}^{\ast})[\gamma_{+}-1]=\I^{\ast}(\mathbf{a}_{p}^{\ast})
\]
(recall that $\mathbf{a}_{p}^{\ast}$ is the unramified character on $G_{\Q_{p}}$
sending an arithmetic Frobenius to $\mathbf{a}_{p}$). Finally, note that $\varpi_{\mathrm{cc}}:=[\gamma_{\mathrm{wt}}]-\gamma_{+}^{2}$
acts as $\varpi_{\mathrm{wt}}=[\gamma_{\mathrm{wt}}]-1$ on $\I^{\ast}=\I^{\ast}_{K_{\infty}}[\gamma_{+}-1]$, 
so that $\I^{\ast}$ is $\varpi_{\mathrm{cc}}$-divisible, and hence 
\[
 H^{0}(I_{v},T_{\mathbf{f}}(K_{\infty})_{v}^{-}\otimes_{\I_{K_{\infty}}}\I_{K_{\infty}}^{\ast})
             \otimes_{\I_{K_{\infty}}}\I_{K_{\infty}}/\varpi_{\mathrm{cc}}=0
\]
for every prime $v|p$ of $K$. Together with $(\ref{eq:nbnbnbnbnbnb})$, this implies that   $(\ref{eq:nnn777})$
holds true, and then proves  the claim $(\ref{eq:333999})$. 
When combined with the isomorphism  $(\ref{eq:445})$,  this gives  canonical isomorphisms of $\I$-modules
\[
              \mathrm{Sel}_{\Q_{\infty}}^{S,\mathrm{cc}}(\mathbf{f}/K)\cong{}\mathrm{Sel}_{K_{\infty}}^{S}(\mathbf{f})[\mathfrak{a}_{2}];\ \ \ 
              X_{\Q_{\infty}}^{S,\mathrm{cc}}(\mathbf{f}/K)\cong{}X_{K_{\infty}}^{S}(\mathbf{f})/\mathfrak{a}_{2}.
\]
Since $\mathfrak{P}^{\mathrm{cc}}=(\mathfrak{a}_{1},\mathfrak{a}_{2})\cdot{}\I_{\mathcal{K}}$,
combined  with the second isomorphism in $(\ref{eq:ggg})$, this concludes the proof.
\end{proof}

\subsection{Specialising Skinner-Urban to the central critical line}
We can finally state the following corollary of the theorem of Skinner-Urban. For every $f(k)\in{}\mathscr{A}(U)$, write 
$\mathrm{ord}_{k=2}f(k)\in{}\mathbf{N}$ to denote the order of vanishing of $f(k)$ at $k=2$. 
Given a finite $\mathbb{I}$-module  $M$, write as usual $\mathrm{length}_{\mathfrak{p}_{f}}(M)$
for the length of the localisation $M_{\mathfrak{p}_{f}}$ over the discrete valuation ring $\mathbb{I}_{\mathfrak{p}_{f}}$.

\begin{cor}\label{corsumain} Assume that Hypotheses $\ref{h2}$, $\ref{h3}$ and $\ref{h4}$ are satisfied. Then 
$$\mathrm{ord}_{k=2}L_{p}^{\mathrm{cc}}(f_{\infty}/K,k)\leq{}\mathrm{length}_{\mathfrak{p}_{f}}
\big(X_{\Q_{\infty}}^{S,\mathrm{cc}}(\mathbf{f}/K)\big).$$
\end{cor}
\begin{proof} Combining  Skinner-Urban's Theorem $\ref{msu}$ with Proposition $\ref{speccgr}$, we easily deduce that 
the characteristic ideal of $X_{\Q_{\infty}}^{S,\mathrm{cc}}(\mathbf{f}/K)$ is contained in the principal ideal generated 
by the projection $\mathcal{L}_{K}^{S}(\mathbf{f})\ \mathrm{mod}\ \mathfrak{P}^{\mathrm{cc}}$ (cf. 
the proof of \cite[Corollary 3.8]{S-U}). In other words 
\[
       \Big\{\mathrm{Characteristic\ ideal\ of\ } X_{\Q_{\infty}}^{S,\mathrm{cc}}(\mathbf{f}/K) \Big\} \subset{}\Big(\mathcal{L}_{K}^{S}(\mathbf{f})\ \mathrm{mod}\ \mathfrak{P}^{\mathrm{cc}}\Big).
\]
In particular, writing $\mathrm{ord}_{\mathfrak{p}_{f}} : \mathrm{Frac}(\mathbb{I})\fre{}\Z\cup\{\infty\}$ for the valuation 
 attached to $\mathfrak{p}_{f}$,  
\[
                      \mathrm{ord}_{\mathfrak{p}_{f}}\lri{\mathcal{L}_{K}^{S}(\mathbf{f})\ \mathrm{mod}\ \mathfrak{P}^{\mathrm{cc}}}\leq{}
                      \mathrm{length}_{\mathfrak{p}_{f}}\lri{X_{\Q_{\infty}}^{S,\mathrm{cc}}(\mathbf{f}/K)}.
\]
Write for simplicity $\mathcal{L}_{\Q_{\infty}}^{S,\mathrm{cc}}(\mathbf{f}/K):=\mathcal{L}_{K}^{S}(\mathbf{f})\ \mathrm{mod}\ \mathfrak{P}^{\mathrm{cc}}$.
To conclude the proof it remains to verify that
\begin{equation}\label{eq:todocor}
               \mathrm{ord}_{\mathfrak{p}_{f}}\mathcal{L}_{\Q_{\infty}}^{S,\mathrm{cc}}(\mathbf{f}/K)=
               \mathrm{ord}_{k=2}L_{p}^{S}(f_{\infty}/K,k,k/2,1).
\end{equation} 
Note that, by the definition of the Mellin transform  $\widetilde{\mathtt{M}}$ (and the normalisation 
$\chi_{\mathrm{cy}}(\gamma_{+})=\gamma_{\mathrm{wt}}$) we have
\begin{equation}\label{eq:lolo}
                          \widetilde{\mathtt{M}}(\varpi_{\mathrm{cc}})(k,s,r)=\gamma_{\mathrm{wt}}^{k-2}-\gamma_{\mathrm{wt}}^{2(s-1)}
                          =\gamma_{\mathrm{wt}}^{2(s-1)}\lri{\gamma_{\mathrm{wt}}^{2(k/2-s)}-1}
                          \equiv{}0\ \mathrm{mod}\ (s-k/2)\cdot{}\mathscr{A}(U\times{}\Z_{p}\times{}\Z_{p}),
\end{equation}
and then $\widetilde{\mathtt{M}}(\varpi_{\mathrm{cc}})(k,k/2,1)=0$. Similarly, writing 
$\ell_{\mathrm{wt}}:=\log_{p}(\gamma_{\mathrm{wt}})$  and $\ell_{-}:=\log_{p}(\chi_{\mathrm{acy}}(\gamma_{-}))$, we have
\begin{equation}\label{eq:p-p-}
           \mathtt{M}(\varpi_{\mathrm{wt}})(k)\equiv{}\ell_{\mathrm{wt}}\cdot{}(k-2)\ \mathrm{mod}
           \ (k-2)^{2};\ \ 
           \widetilde{\mathtt{M}}(\varpi_{-})(k,s,r)\equiv{}\ell_{-}\cdot{}(r-1)\ \mathrm{mod}\ (r-1)^{2}.
\end{equation}
Assume now that $\mathcal{L}_{\Q_{\infty}}^{S,\mathrm{cc}}(\mathbf{f}/K)\in{}
\mathfrak{p}_{f}^{m}\mathbb{I}_{\mathfrak{p}_{f}}-\mathfrak{p}_{f}^{m+1}\mathbb{I}_{\mathfrak{p}_{f}}$, for some integer $m\geq{}0$, so that 
$\mathrm{ord}_{\mathfrak{p}_{f}}\mathcal{L}_{\Q_{\infty}}^{S,\mathrm{cc}}(\mathbf{f}/K)=m$.
Since $\mathfrak{p}_{f}\mathbb{I}_{\mathfrak{p}_{f}}$ is a principal ideal generated by $\varpi_{\mathrm{wt}}$ $(\ref{eq:unif})$,
equation $(\ref{eq:p-p-})$ gives
\[
                    \mathrm{ord}_{k=2}\mathtt{M}\lri{\mathcal{L}_{\Q_{\infty}}^{S,\mathrm{cc}}(\mathbf{f}/K)}(k)=
                    \mathrm{ord}_{\mathfrak{p}_{f}}\mathcal{L}_{\Q_{\infty}}^{S,\mathrm{cc}}(\mathbf{f}/K).
\]
On the other hand, we have by construction 
$\mathcal{L}_{K}^{S}(\mathbf{f})\equiv{}\mathcal{L}_{\Q_{\infty}}^{S,\mathrm{cc}}(\mathbf{f})\ \mathrm{mod}\ \mathfrak{P}^{\mathrm{cc}}$,
so that equations $(\ref{eq:lolo})$ and $(\ref{eq:p-p-})$ give
\[
         L_{p}^{S}(f_{\infty}/K,k,k/2,1):=\widetilde{\mathtt{M}}\lri{\mathcal{L}_{K}^{S}(\mathbf{f})}(k,k/2,1)=
         \mathtt{M}\lri{\mathcal{L}_{\Q_{\infty}}^{S,\mathrm{cc}}(\mathbf{f}/K)}(k).
\]
Combining the preceding two equations, we deduce that $(\ref{eq:todocor})$ holds in this case.
Assume finally that  $\mathcal{L}_{K}^{S}(\mathbf{f})\in{}\mathfrak{P}^{\mathrm{cc}}$,
i.e. $\mathcal{L}_{\Q_{\infty}}^{S,\mathrm{cc}}(\mathbf{f}/K)=0$.
(This is the case `$m=\infty$'.) Then $L_{p}^{S}(f_{\infty}/K,k,k/2,1)\equiv{}0$ by $(\ref{eq:lolo})$
and $(\ref{eq:p-p-})$,
 so that $(\ref{eq:todocor})$ holds also in this case (giving $\infty=\infty$).
\end{proof}

\section[Bertolini Darmon]{Bertolini-Darmon's exceptional zero formula}\label{bdsec}
Throughout this section, the notations and assumptions are as in Section $\ref{rccl}$.
In particular, we assume that Hypotheses $\ref{h2}$--$\ref{h4}$ 
are satisfied.

Let $\kappa\in{}U^{\mathrm{cl}}$ be a classical point in $U$, let $\phi_{k}\in{}\xari(\mathbb{I})$ be the associated arithmetic point
(of weight $\kappa$ and trivial character), and let $f_{\kappa}\in{}S_{\kappa}(\Gamma_{0}(Np))$ be the corresponding $p$-stabilised
newform (cf. Section $\ref{lochidafam}$). 
Write  $\phi_{\kappa}^{\dag}=\phi_{\kappa}\times{}\chi_{\mathrm{cy}}^{\kappa/2-1}\times{}1 : \I\llbracket{}G_{\infty}\times{}D_{\infty}\rrbracket\fre{}
\overline{\Q}_{p}$
for the morphism of $\mathcal{O}_{L}$-algebras 
such that $\phi_{\kappa}^{\dag}(\sigma\times{}h)=\chi_{\mathrm{cy}}(\sigma)^{\kappa/2-1}$
for every $\sigma\times{}h\in{}G_{\infty}\times{}D_{\infty}$,
and such that $\phi_{\kappa}^{\dag}(x)=\phi_{\kappa}(x)$ for every $x\in{}\I$.
Since $\kappa\equiv{}2\ \mathrm{mod}\ 2(p-1)$, $p\not=2$, and $p$ splits in $K$ (i.e. $\epsilon_{K}(p)=1$), equations $(\ref{eq:intcp})$ and $(\ref{eq:spesu})$ yield
\[
    \phi^{\dag}_{\kappa}\lri{\mathcal{L}_{K}^{S}(\mathbf{f})}=\lambda_{\kappa}
    D_{K}^{\frac{\kappa-2}{2}}
          \lri{1-\frac{p^{\frac{\kappa}{2}-1}}{a_{p}(\kappa)}}^{2}
          \frac{(\kappa/2-1)!\cdot{}L^{S\backslash\{p\}}(f_{\kappa},\kappa/2)}
          {(-2\pi{}i)^{\kappa/2-1}\Omega^{+}_{\phi_{\kappa}}}\cdot{}
          \frac{G(\epsilon_{K})(\kappa/2-1)!\cdot{}L^{S\backslash\{p\}}(f_{\kappa},\epsilon_{K},\kappa/2)}
          {(-2\pi{}i)^{\kappa/2-1}\Omega^{-}_{\phi_{\kappa}}}
\]
By the very definition of the central critical $p$-adic $L$-function $L_{p}^{\mathrm{cc}}(f_{\infty}/K,k)$ we then deduce: for every 
$\kappa\in{}U^{\mathrm{cl}}$
\[
           L_{p}^{\mathrm{cc}}(f_{\infty}/K,\kappa)=\lambda_{\kappa}D_{K}^{\frac{\kappa-2}{2}}\lri{1-\frac{p^{\frac{\kappa}{2}-1}}{a_{p}(\kappa)}}^{2}\cdot{}
           \frac{(\kappa/2-1)!L(f_{\kappa},\kappa/2)}{(-2\pi{}i)^{\kappa/2-1}\Omega_{\phi_{\kappa}}^{+}}\cdot{}
           \frac{G(\epsilon_{K})(\kappa/2-1)!L(f_{\kappa},\epsilon_{K},\kappa/2)}{(-2\pi{}i)^{\kappa/2-1}\Omega_{\phi_{\kappa}}^{-}}.
\]
Since $U^{\mathrm{cl}}$ is a dense subset of $U$,  if we compare this formula with
\cite[Theorem 1.12]{B-D}, we obtain a factorisation
\begin{equation}\label{eq:compl}
          L_{p}^{\mathrm{cc}}(f_{\infty}/K,k)=D_{K}^{\frac{k-2}{2}}
          L_{p}(f_{\infty},k,k/2){ }L_{p}(f_{\infty},\epsilon_{K},k,k/2).
\end{equation}
Here, for every quadratic Dirichlet character $\chi$ of conductor coprime with $Np$, 
$L_{p}(f_{\infty},\chi,k,s)\in{}\mathscr{A}(U\times{}\Z_{p})$
is a Mazur-Kitagawa two-variable $p$-adic  $L$-function attached to $f_{\infty}$ and $\chi$ in \cite[Section 1]{B-D},
and we write simply $L_{p}(f_{\infty},k,s):=L_{p}(f_{\infty},\chi_{\mathrm{triv}},k,s)$ when $\chi=\chi_{\mathrm{triv}}$ is the trivial character. 
Like $L_{p}^{\mathrm{cc}}(f_{\infty}/K,s)$  (once the periods $\Omega_{\phi_{\kappa}}^{\pm}$ are fixed for $\kappa\in{}U^{\mathrm{cl}}$),
$L_{p}(f_{\infty},\chi,k,s)$ is characterised by its interpolation property (namely \cite[Theorem 1.12]{B-D}) up to multiplication 
by a nowhere-vanishing analytic function on $U$,
so the preceding equality has to be interpreted up to multiplication by such a unit in $\mathscr{A}(U)$.

The following  \emph{exceptional-zero} formula is the main result (Theorem 5.4) of \cite{B-D}, where it is proved 
under a  technical assumption (namely the existence of a prime $q\Vert{}N$) subsequently removed by Mok in \cite{Mok}. 
Write $\mathrm{sign}(A/\Q)
\in{}\{\pm1\}$ for  the sign in the functional equation
satisfied by the Hecke $L$-series $L(A/\Q,s)=L(f,s)$.

\begin{theo}[Bertolini-Darmon \cite{B-D}]  Let $\chi$ be a   quadratic Dirichlet character 
of conductor coprime with $N_{A}=Np$, such that 
\[
                    \chi(-N)=-\mathrm{sign}(A/\Q);\ \ \chi(p)=a_{p}(A)=+1.   
\]
If $\chi$ is non-trivial (resp., $\chi=1$), let $K_{\chi}/\Q$ be the quadratic extension attached to $\chi$
(resp., let $K_{\chi}:=\Q$). Then

$1.$ $L_{p}(f_{\infty},\chi,k,k/2)$ vanishes to order at least $2$ at $k=2$.

$2.$ There exists a global  point $\mathbf{P}_{\chi}\in{}A(K_{\chi})^{\chi}$ \footnote{By $A(K_{\chi})^{\chi}$
we mean the subgroup of $A(K_{\chi})$ on which $\mathrm{Gal}(K_{\chi}/\Q)$ acts via $\chi$.} such that 
\[
                       \frac{d^{2}}{dk^{2}}L_{p}(f_{\infty},\chi,k,k/2)_{k=2}\stackrel{\cdot}{=}\log_{A}^{2}(\mathbf{P}_{\chi}),
\]
where $\log_{A} : A(\overline{\Q}_{p})\fre{}\overline{\Q}_{p}$ is the formal group logarithm \footnote{Writing 
$\Phi_{\mathrm{Tate}} : \overline{\Q}_{p}^{\times}/q_{A}^{\Z}\cong{}A(\overline{\Q}_{p})$
for the Tate $p$-adic uniformization of $A/\Q_{p}$
(see Section $\ref{exalg}$ below), one can define 
$\log_{A}:=\log_{q_{A}}\circ{}\Phi_{\mathrm{Tate}}^{-1} : A(\overline{\Q}_{p})\fre{}\overline{\Q}_{p}$,
where $\log_{q_{A}}$ is the branch of the $p$-adic logarithm vanishing at the Tate period 
$q_{A}\in{}p\Z_{p}$ of $A/\Q_{p}$.}, and 
$\stackrel{\cdot{}}{=}$
denotes equality up to multiplication by a non-zero
(explicit)  factor in $\Q_{p}^{\times}$. 

$3.$ $\mathbf{P}_{\chi}$ has infinite order if and only if  the Hecke $L$-series $L(f,\chi,s)$ has a simple zero at $s=1$.
\end{theo}

In the preceding result, $\chi$ is allowed to be a generic Dirichelt character 
of conductor coprime with $Np$. Applying the theorem to both $\chi=\chi_{\mathrm{triv}}$
and $\chi=\epsilon_{K}$, we obtain the following corollary. 

\begin{cor}\label{corbdmain} Assume that
$\mathrm{sign}(A/\Q)=-1$, and that 
Hypotheses  $\ref{h2}$, $\ref{h3}$ and $\ref{h4}$ are satisfied.
Denote by 
$L(A/K,s):=L(f,s)\cdot{}L(f,\epsilon_{K},s)$ the complex Hasse-Weil $L$-function of $A/K$.
Then 
$L_{p}^{\mathrm{cc}}(f_{\infty}/K,k)$ vanishes to order at least $4$ at $k=2$, and
\[
    \mathrm{ord}_{k=2}L_{p}^{\mathrm{cc}}(f_{\infty}/K,k)=4\ \ \iff
\ \ \mathrm{ord}_{s=1}L(A/K,s)=2.
\]
\end{cor}
\begin{proof} 
Since $\mathrm{sign}(A/\Q)=-1$, the hypotheses of the preceding theorem are satisfied by $\chi=\chi_{\mathrm{triv}}$.
Moreover, since $p$ splits in $K$ by Hypothesis $\ref{h3}$,
$\epsilon_{K}(p)=+1$, and $\epsilon_{K}(-N)=-\epsilon(N^{-})=+1$ by Hypothesis $\ref{h4}$.
Then $\chi=\epsilon_{K}$ also satisfies the hypotheses of the theorem. 
The corollary then follows by applying the theorem to both $\chi=\chi_{\mathrm{triv}}$
and $\chi=\epsilon_{K}$, and using the factorisation $(\ref{eq:compl})$.
\end{proof}

\section[Bounding]{Bounding  the characteristic ideal via \neko's duality}\label{mainsec} 

Recall the  arithmetic prime $\phi_{f}\in{}\xari(\mathbb{I})$ 
defined   in Section $\ref{lochidafam}$, and write as above $\mathfrak{p}_{f}:=\ker(\phi_{f})$, which is a height-one prime ideal of $\mathbb{I}$.
Let $\chi$ be a  quadratic Dirichlet character of conductor coprime with $Np$.
If $\chi$ is non-trivial (resp., $\chi=1$), let $K_{\chi}/\Q$ be the corresponding quadratic extension
(resp., let $K_{\chi}:=\Q$), and let $D_{\chi}$ be the discriminant of $K_{\chi}$.
Fix a finite set $S$  of primes of $K_{\chi}$ containing all the prime divisors of $NpD_{\chi}$,
and decomposition groups $G_{K_{\chi,w}}:=\mathrm{Gal}(\overline{\Q}_{\ell}/K_{\chi,w})\hookrightarrow{}G_{K_{\chi}}$
at $w$, for every  $w\in{}S$ dividing the rational prime $\ell$ (where $K_{\chi,w}$ 
denotes the completion of $K_{\chi}$ at $w$).
Define  the \emph{strict Greenberg Selmer group} of $\mathbb{T}_{\mathbf{f}}/K_{\chi}$
(cf. Section $\ref{ppq}$):
\[
                 \mathrm{Sel}_{\mathrm{Gr}}^{\mathrm{cc}}(\mathbf{f}/K_{\chi}):=
                 \ker\lri{H^{1}(G_{K_{\chi},S},\mathbb{T}_{\mathbf{f}}\otimes_{\I}\mathbb{I}^{\ast})
                 \longrightarrow{}\prod_{v|p}H^{1}(K_{\chi,v},\mathbb{T}_{\mathbf{f},v}^{-}\otimes_{\I}\mathbb{I}^{\ast})},
\]
where $G_{K_{\chi},S}$ denotes as usual 
the Galois group of the maximal algebraic extension of $K_{\chi}$ which is unramified 
outside $S\cup{}\{\infty\}$.
Let
\[
                 X_{\mathrm{Gr}}^{\mathrm{cc}}(\mathbf{f}/K_{\chi}):=\Hom{\Z_{p}}\Big(
                 \mathrm{Sel}_{\mathrm{Gr}}^{\mathrm{cc}}(\mathbf{f}/K_{\chi}),\divp\Big)\   \footnote{ The Selmer groups already defined depend in general on the choice of the set $S$.
On the other hand, we are interested here only in the structure of the localisation of $X_{\mathrm{Gr}}^{\mathrm{cc}}(\mathbf{f}/K_{\chi})$ at $\mathfrak{p}_{f}$,
and such a localisation does not depend, up to canonical isomorphism, on the choice of $S$.}.
\]
 For every $\Z[\mathrm{Gal}(K_{\chi}/\Q)]$-module $M$,
write $M^{\chi}$ for the submodule of $M$ on which $\mathrm{Gal}(K_{\chi}/\Q)$ acts via $\chi$
(so that $M^{\chi}:=M$ is $\chi$ is trivial, and $M^{\chi}$ is the submodule of $M$ on which the nontrivial automorphism 
of $\mathrm{Gal}(K_{\chi}/\Q)$ acts as $-1$ if $\chi$ is nontrivial). 
The aim of this section is to prove the following theorem.

\begin{theo}\label{bchar} Let $\chi$ be a quadratic Dirichlet 
character of conductor coprime with $Np$.
Assume that:

$(i)$ $\chi(p)=1$, i.e. $p$ splits in $K_{\chi}$;

$(ii)$ $\rank_{\Z}A(K_{\chi})^{\chi}=1$;

$(iii)$ the $p$-primary subgroup $\sha(A/K_{\chi})^{\chi}_{p^{\infty}}$ 
of $\sha(A/K_{\chi})^{\chi}$ is finite. \\
Then  the localisation at $\mathfrak{p}_{f}$ of $X_{\mathrm{Gr}}^{\mathrm{cc}}(\mathbf{f}/K_{\chi})^{\chi}$ is isomorphic
to the residue field of the discrete valuation ring $\Il$:
\[
      X_{\mathrm{Gr}}^{\mathrm{cc}}(\mathbf{f}/K_{\chi})^{\chi}\otimes_{\I}\Il\cong{}\Il/\mathfrak{p}_{f}\Il.
\]
\end{theo}

\subsection{\neko's theory}\label{sc}  In this section we recall the needed results from 
\neko's theory of Selmer complexes \cite{Ne}.
Unless explicitly specified, all notations and conventions are as in \emph{loc. cit.}

\subsubsection{\neko's Selmer complexes}\label{genenek} Given a ring $R$,
write $\mathrm{D}(R):=\mathrm{D}(_{R}\mathrm{Mod})$ for the derived category of complexes of $R$-modules,
and $\mathrm{D}_{\mathrm{ft}}^{b}(R)\subset{}\mathrm{D}(R)$ (resp., $\mathrm{D}_{\mathrm{cf}}^{b}(R)\subset{}\mathrm{D}(R)$) for the 
subcategory of cohomologically bounded complexes, with cohomology 
of finite (resp., cofinite) type over $R$. 

Recall the self-dual, ordinary $\I$-adic representation $\ppq=(\ppq,\ppq^{\pm})$, defined in Section $\ref{ppq}$. Denote by
\[
       \mathbb{A}_{\mathbf{f}}:=\Hom{\mathrm{cont}}(\ppq,\mu_{p^{\infty}}); \ \ 
       \mathbb{A}_{\mathbf{f}}^{\pm}:=\Hom{\mathrm{cont}}(\ppq^{\mp},\mu_{p^{\infty}})
\]
the Kummer dual $p$-ordinary representation.
Set $T_{f}:=\mathbb{T}_{\mathbf{f}}/\mathfrak{p}_{f}\mathbb{T}_{\mathbf{f}}$ and 
$T_{f}^{\pm}:=\mathbb{T}_{\mathbf{f}}^{\pm}/\mathfrak{p}_{f}\mathbb{T}_{\mathbf{f}}^{\pm}$.  Then one has 
\[
      A_{f}:=\Hom{\mathrm{cont}}(T_{f},\mu_{p^{\infty}})\cong{}\mathbb{A}_{\mathbf{f}}[\mathfrak{p}_{f}];\ \ 
       A_{f}^{\pm}:=\Hom{\mathrm{cont}}(T_{f}^{\mp},\mu_{p^{\infty}})\cong{}\mathbb{A}^\pm_{\mathbf{f}}[\mathfrak{p}_{f}].
\]
Given a multiplicative subset $\mathscr{S}$ of a ring $R$, and an $R$-module $M$,
write as usual $\mathscr{S}^{-1}M$ for the localisation of $M$ at $\mathscr{S}$.
Fix a multiplicative subset  $\mathscr{S}$ of $\I$ or $\mathcal{O}_{L}$, let 
$$X\in{}\{\mathscr{S}^{-1}\mathbb{T}_{\mathbf{f}}, \mathscr{S}^{-1}T_{f}, \mathbb{A}_{\mathbf{f}}, A_{f}\}$$ and let 
$R_{X}\in{}\{\mathscr{S}^{-1}\I,\mathscr{S}^{-1}\mathcal{O}_{L},\I,\mathcal{O}_{L}\}$ be the corresponding `coefficient ring'. 
For every prime $v|p$ of $K_{\chi}$,
set  $X_{v}^{+}:=\mathscr{S}^{-1}\mathbb{T}_{\mathbf{f}}^{+}$ (resp., $\mathscr{S}^{-1}T_{f}^{+}$, $\mathbb{A}_{\mathbf{f}}^{+}$,
$A_{f}^{+}$) if $X=\mathscr{S}^{-1}\mathbb{T}_{\mathbf{f}}$ 
(resp., $\mathscr{S}^{-1}T_{f}$, $\mathbb{A}_{\mathbf{f}}$, $A_{f}$),
and  $X_{v}^{-}:=X/X_{v}^{+}$.
The exact sequence $(\ref{eq:exact for general v})$
then induces  short exact sequences of $R_{X}[G_{K_{\chi},v}]$-modules
\[
             0\fre{}X_{v}^{+}\stackrel{i_{v}^{+}\ }{\rightarrow}X\fre{p_{v}^{-}\ }X_{v}^{-}\fre{}0.
\]
(Recall that $\mathbb{T}_{\mathbf{f},w}^{+}:=\mathbb{T}_{\mathbf{f}}^{+}$ for every prime $w|p$ of 
$\overline{\Q}$, cf. equation $(\ref{eq:exact for general v})$.)

As in \cite[Section 6]{Ne}, define local conditions $\Delta_{S}(X)=\left\{\Delta_{v}(X)\right\}_{v\in{}S}$ 
for $X/K_{\chi}$ as follows \footnote{Let $R$ be a local complete Noetherian ring with finite residue field of characteristic $p$, and let $T$
be an $R$-module of finite or cofinite type, equipped with a continuous, linear action of $G_{K_{\chi},S}$.
For every $w\in{}S$, fix a decomposition group $G_{w}$ at $w$, 
i.e. $G_{w}:=G_{K_{\chi},w}\hookrightarrow{}G_{K_{\chi}}\twoheadrightarrow{}G_{K_{\chi},S}$.
According to \neko's theory of Selmer complexes, a \emph{local condition} at $w\in{}S$ for $T$
is the choice $\Delta_{w}(T)$ of a complex of $R$-modules $U_{w}^+(T)$,
together with a morphism of complexes
$i_{w}^{+}(T) : U_{w}^+(T)\fre{}\ctsb(K_{\chi,w},T)$. For $G=G_{K_{\chi},S}$ or $G_{w}$ ($w\in{}S$),
$\ctsb(G,T)$ (also denoted $\ctsb(K_{\chi,w},T)$ when $G=G_{w}$) is the complex of continuous (non-homogeneous) $T$-valued cochains on $G$.
If $\mathscr{R}$ is a localisation of $R$, and $\mathscr{T}:=T\otimes_{R}\mathscr{R}$,
set $\ctsb(\ast,\mathscr{T}):=\ctsb(\ast,T)\otimes_{R}\mathscr{R}$.
Then a \emph{local condition} for $\mathscr{T}$ at $w\in{}S$
is a morphism $i_{w}^{+}(T)\otimes\mathscr{R} : U_{w}^{+}(T)\otimes_{R}\mathscr{R}\fre{}\ctsb(K_{\chi,w},\mathscr{T})$, obtained as the base change of a local condition $i_{w}^{+}(T)$ for $T$ at $w$.
}.
For a prime $v\in{}S$ dividing $p$, let $\Delta_{v}(X)$ be the morphism
\[
  i_{v}^+(X) :  U_{v}^+(X):=\ctsb(K_{\chi,v},X_{v}^+)
                  \longrightarrow{}\ctsb(K_{\chi,v},X),
\]
i.e. $\Delta_{v}(X)$ is the Greenberg local condition attached to the $R_{X}[G_{K_{\chi,v}}]$-submodule  
$i_{v}^{+} : X_{v}^+\subset{}X$.  For every $S\ni{}w\nmid{}p$, we define $\Delta_{w}(X)$ to be the \emph{full} local condition:
$i_{w}^+(X) : U_{w}^+(X):=0\fre{}\ctsb(K_{\chi,w},X)$ (resp., the \emph{empty} local condition:
$i_{w}^+(X)=\mathrm{id} : U_{w}^+(X):=\ctsb(K_{\chi,w},X)\fre{}\ctsb(K_{\chi,w},X)$) in case $X\in{}\{\mathscr{S}^{-1}\ppq,\mathscr{S}^{-1}T_{f}\}$
(resp., $X\in{}\{\mathbb{A}_{\mathbf{f}},A_{f}\}$).
The associated \emph{\neko's Selmer complex} \cite{Ne} is defined as the complex of $R_{X}$-modules
\[
\scob(K_{\chi},X)=\scob(G_{K_{\chi},S},X;\Delta_{S}(X)):=
     \mathrm{Cone}\lri{\ctsb(G_{K_{\chi},S},X)\oplus\bigoplus_{v\in{}S}U_{v}^+(X)\stackrel{\mathrm{res}_{S}-i_{S}^+}{\longrightarrow{}}
     \bigoplus_{v\in{}S}\ctsb(K_{\chi,v},X)}[-1],
\]
where $\mathrm{res}_S=\oplus_{v\in{}S}\mathrm{res}_{v}$ and $i_{S}^+=\oplus_{v\in{}S}i_{v}^+(X)$. 
It follows by standard results on continuous Galois cohomology groups \cite[Section 4]{Ne}
(essentially due to Tate \cite{Tate}) that $\scob(K_{\chi},X)$ is cohomologically bounded, with cohomology of finite (resp., cofinite)
type over $R_{X}$ if $X$ is of finite (resp., cofinite) type over $R_{X}$. 
Let
\[     
     \derco(K_{\chi},X)\in{}\mathrm{D}^{b}_{\mathrm{ft}, (\text{resp., cf})}(R_{X});\ \ 
     \exsel^{\ast}(K_{\chi},X):=H^{\ast}\Big(\derco(K_{\chi},X)\Big)
     \in{}\lri{_{R_{X}}\mathrm{Mod}}_{\mathrm{ft}, (\text{resp., cf})}
\]     
be the image of $\scob(K_{\chi},X)$ in the derived category and its cohomology respectively. 
If $\mathcal{X}\in{}\{\ppq,T_{f}\}$ and 
$R_{\mathcal{X}}\in{}\{\I,\mathcal{O}_{L}\}$ is the corresponding coefficient ring, then
\[
       \derco(K_{\chi},X)\cong{}\derco(K_{\chi},\mathcal{X})\otimes_{R_{\mathcal{X}}}R_{X};\ \ 
       \exsel^{\ast}(K_{\chi},X)\cong{}\exsel^{\ast}(K_{\chi},\mathcal{X})\otimes_{R_{\mathcal{X}}}R_{X},
\]
which we consider as equalities in what follows.

Let $X\in{}\{\mathscr{S}^{-1}\mathbb{T}_{\mathbf{f}},\mathscr{S}^{-1}T_{f}\}$
(resp., $X\in{}\{\mathbb{A}_{\mathbf{f}},A_{f}\}$), and let $S\ni{}w\nmid{}p$.
Define the $R_{X}[G_{K_{\chi},w}]$-module $X_{w}^{-}:=X$
(resp., $X_{w}^{-}:=0$).
By the definition of \neko's Selmer complexes, there is  a long exact cohomology sequence
of $R_{X}$-modules \cite[Section 6]{Ne}:
\[
           \cdots\fre{}\bigoplus_{w\in{}S}H^{q-1}(K_{\chi,w},X_{w}^{-})\fre{}\exsel^{q}(K_{\chi},X)\fre{}H^{q}(G_{K_{\chi},S},X)\fre{}\bigoplus_{w\in{}S}
           H^{q}(K_{\chi,w},X_{w}^{-})\fre{}\cdots.
\]
In particular this gives an exact sequence of $R_{X}$-modules
\begin{equation}\label{eq:exselmer}
      X^{G_{K_{\chi},S}}\fre{}\bigoplus_{w\in{}S}H^{0}(K_{\chi,w},X_{w}^{-})\fre{}\exsel^{1}(K_{\chi},X)\fre{}\mathfrak{S}(K_{\chi},X)\fre{}0.
\end{equation}
Here $\mathfrak{S}(K_{\chi},X)=\mathfrak{S}(K_{\chi,S},X)$ is the  \emph{($S$-primitive, strict) Greenberg Selmer group 
of $X/K_{\chi}$}, defined by
\[
          \mathfrak{S}(K_{\chi},X):=\ker\lri{H^{1}(G_{K_{\chi},S},X)\longrightarrow{}\prod_{w\in{}S}
          H^{1}(K_{\chi,w},X_{w}^{-})}.
\]

\subsubsection{A control theorem}\label{control} We know that $\mathbb{I}_{\mathfrak{p}_{f}}$ is a discrete valuation ring,
and that its maximal ideal $\mathfrak{p}_{f}\mathbb{I}_{\mathfrak{p}_{f}}$ is generated by $\varpi_{\mathrm{wt}}:=\gamma_{\mathrm{wt}}-1\in{}\iw$
(see $(\ref{eq:unif})$). Write $V_{f}:=T_{f}\otimes_{\mathcal{O}_{L}}L$ and $\ppqll:=\mathbb{T}_{\mathbf{f}}\otimes_{\mathbb{I}}\Il$.
By \cite[Propositions 3.4.2 and 3.5.10]{Ne}, the arithmetic point $\phi_{f}\in{}\xari(\I)$
induces an exact triangle in $\mathrm{D}^{b}_{\mathrm{ft}}(\I_{\mathfrak{p}_{f}})$:
$$\derco(K_{\chi},\ppqll)\stackrel{\varpi_{\mathrm{wt}}}{\longrightarrow{}}\derco(K_{\chi},\ppqll)\stackrel{\phi_{f\ast}\ \ }{\longrightarrow{}}
\derco(K,V_{f}),$$ 
and then an isomorphism in $\mathrm{D}_{\mathrm{ft}}^{b}(L)$:
\begin{equation}\label{eq:controlder}
        c_{f}  :   \mathbf{L}\phi_{f}^{\ast}\Big(\derco(K_{\chi},\ppqll)\Big)\cong{}\derco(K_{\chi},V_{f}),
\end{equation}
where $\mathbf{L}\phi_{f}^{\ast} : \mathrm{D}^-(\Il)\fre{}\mathrm{D}(L)$ 
is the left derived functor of the base-change functor 
$\phi_{f}^{\ast}(\cdot):=\cdot\otimes_{\I,\phi_{f}}L$.
(Note that, since $f=f_{2}$ has integral Fourier coefficients,  the residue field $\Il/\mathfrak{p}_{f}\Il$ of $\mathbb{I}_{\mathfrak{p}_{f}}$ equals $L$.)
 This induces in cohomology short exact sequences of $L$-modules
\begin{equation}\label{eq:speco}
           0\fre{}\exsel^{q}(K_{\chi},\mathbb{T}_{\mathbf{f},\mathfrak{p}_{f}})/\varpi_{\mathrm{wt}}\fre{}\exsel^{q}(K_{\chi},V_{f})
           \fre{i_{\mathrm{wt}}}\exsel^{q+1}(K_{\chi},\ppqll)[\varpi_{\mathrm{wt}}]\fre{}0.
\end{equation}

\subsubsection{\neko's duality I: global cup-products}\label{glocp} 
Let $\mathcal{X}\in{}\{\ppq,T_{f}\}$, and let $\mathcal{R}\in{}\{\I,\mathcal{O}_{L}\}$ be the corresponding coefficient ring.
For $\mathscr{S}\in{}\{\I-\mathfrak{p}_{f},\mathcal{O}_{L}-\mathfrak{m}_{L}\}$
(where $\mathfrak{m}_{L}$ is the maximal ideal of $\mathcal{O}_{L}$), write
$X:=\mathscr{S}^{-1}\mathcal{X}\in{}\{\ppql,V_{f}\}$ and  $R_{X}:=\mathscr{S}^{-1}\mathcal{R}\in{}\{\Il,L\}$.
Let $$\pi_{X} : X\otimes_{R_{X}}X\fre{}R_{X}(1)$$ be the localization at $\mathscr{S}$
of the  perfect duality $\pi : \ppq\otimes_{\I}\ppq\fre{}\I(1)$ if $\mathcal{X}=\ppq$,
or the localisation at $\mathscr{S}$ of its $\phi_{f}$-base change $\pi_{f}:=\phi_{f}^{\ast}(\pi) : T_{f}\otimes_{\mathcal{O}_{L}}T_{f}
\fre{}\mathcal{O}_{L}(1)$ if $\mathcal{X}=T_{f}$ (see Section $\ref{ppq}$).
As a manifestation of \neko's wide generalization of Poitou-Tate duality, Section 6 of \cite{Ne} attaches 
to  $\pi_{X}$ 
a morphism in $\mathrm{D}^{b}_{\mathrm{ft}}(R_{X})$: 
\[
       \boldsymbol{\cup^\mathrm{Nek}_{\pi_{X}}}  : 
       \derco(K_{\chi},X)\derot{R_{X}}\derco(K_{\chi},X)\longrightarrow{}\mathbf{}
       \tau_{\geq{}3}\mathbf{R}\Gamma_{c,\mathrm{cont}}(K_{\chi},R_{X}(1))\cong{}
        R_{X}[-3],
\]
where $\mathbf{R}\Gamma_{c,\mathrm{cont}}(K_{\chi},-)$ denotes the complex of \emph{cochains with compact support}
\cite[Section 5]{Ne},
and the isomorphism comes (essentially) by the sum of the invariant maps of local class field theory for $v\in{}S$. 
The pairings  $\boldsymbol{\cup^{\mathrm{Nek}}_{\pi}}$ on $\derco(K_{\chi},\ppql)$ and 
$\boldsymbol{    \cup^{\mathrm{Nek}}_{\pi_{f}}}$ on $\derco(K_{\chi},V_{f})$
are compatible with respect to the isomorphism
$c_{f}: \mathbf{L}\phi_{f}^{\ast}\big(\derco(K_{\chi},\ppqll)\big)\cong{}\derco(K_{\chi},V_{f})$
in $\mathrm{D}(L)$
described  in $(\ref{eq:controlder})$.

The global cup-product pairing $\boldsymbol{   \cup^{\mathrm{Nek}}_{\pi_{X}}   }$ gives in cohomology pairings
\begin{equation}\label{eq:nekgcp}
        _{q}\cup^{\mathrm{Nek}}_{\pi_{X}} : \exsel^{q}(K_{\chi},X)\otimes_{R_{X}}\exsel^{3-q}(K_{\chi},X)\longrightarrow{}R_{X}
\end{equation}
(for every $q\in{}\Z$). Writing $\mathscr{R}_{X}:=\mathrm{Frac}(R_{X})$,
they  induce by adjunction  isomorphisms 
\begin{equation}\label{eq:ptcom}
       \mathrm{adj}\lri{{_{q}\cup^{\mathrm{Nek}}_{\pi_{X}}}}   : \exsel^{q}(K_{\chi},X)\otimes_{R_{X}}
       \mathscr{R}_{X}\cong{}
       \Hom{\mathscr{R}_{X}}\lri{\exsel^{3-q}(K_{\chi},X)\otimes_{R_{X}}\mathscr{R}_{X},\mathscr{R}_{X}},
\end{equation}
as follows  from \cite[Proposition 6.6.7]{Ne}, since $\dercts(K_{\chi,w},X)\cong{}0$ is acyclic
for every prime $w\nmid{}p$ of $K_{\chi}$.
(See also \cite[Propositions 12.7.13.3 and 12.7.13.4]{Ne}.)

\subsubsection{\neko's duality II: generalised Pontrjagin duality}
Let $X$ denote either $\ppq$ or $T_{f}$, let $R_{X}$ be either $\I$ or $\mathcal{O}_{L}$ (accordingly),
and let $\mathbb{A}_{X}:=\Hom{\mathrm{cont}}(X,\mu_{p^{\infty}})$
be the (discrete) Kummer dual of $X$. 
Appealing again to  \neko's  generalised  Poitou-Tate duality, we have  Pontrjagin dualities
\begin{equation}\label{eq:nekdu1}
              \exsel^{3-q}(K_{\chi},\mathbb{A}_{X})\cong{}\Hom{\mathrm{cont}}\lri{\exsel^{q}(K_{\chi},X),\divp}
              =:\exsel^{q}(K_{\chi},X)^{\ast}.
\end{equation}
We refer the reader to \cite[Section 6]{Ne} for the details. 

\subsubsection{\neko's duality III: generalised Cassels-Tate pairings}\label{ctsec}
Section $10$ of \cite{Ne}  --which provides a  generalisation
of a construction of Flach \cite{Fl}-- attaches to $\pi : \ppq\otimes_{\I}\ppq\fre{}\I(1)$ a  \emph{skew-symmetric} 
pairing
\[
         \boldsymbol{\cup^{\mathrm{CT}}_{\pi}} : \exsel^{2}(K_{\chi},\ppq)_{\mathrm{tors}}\otimes_{\I}
         \exsel^{2}(K_{\chi},\ppq)_{\mathrm{tors}}\longrightarrow{}\mathrm{Frac}(\I)/\mathbb{I},
\]
where $M_{\mathrm{tors}}=\ker\lri{M\fre{i}M\otimes_{\I}\mathrm{Frac}(\I)}$ denotes the $\I$-torsion submodule of $M$.
Denote by
\begin{equation}\label{eq:nekdu2}
         \cup^{\mathrm{CT}}_{\pi} : \exsel^{2}(K_{\chi},\ppqll)_{\mathrm{tors}}\otimes_{\Il}
         \exsel^{2}(K_{\chi},\ppqll)_{\mathrm{tors}}\longrightarrow{}\mathrm{Frac}(\I_{\mathfrak{p}_{f}})/\mathbb{I}_{\mathfrak{p}_{f}}
\end{equation}
its localization at $\mathfrak{p}_{f}$, $N_{\mathrm{tors}}:=N[\varpi_{\mathrm{wt}}^{\infty}]$ denoting now the 
$\I_{\mathfrak{p}_{f}}$-torsion 
submodule of $N$ (see $(\ref{eq:unif})$).
As proved in \cite[Proposition 12.7.13.3]{Ne}, $\cup_{\pi}^{\mathrm{CT}}$ is a \emph{perfect} pairing, i.e. its adjoint 
\begin{equation}\label{eq:ctnek}
           \mathrm{adj}\lri{\cup_{\pi}^{\mathrm{CT}}} : \exsel^{2}(K_{\chi},\ppqll)_{\mathrm{tors}}\cong{}\Hom{\I_{\mathfrak{p}_{f}}}
           \Big(\exsel^{2}(K_{\chi},\ppqll)_{\mathrm{tors}},\mathrm{Frac}(\Il)/\Il\Big)
\end{equation}
is an isomorphism.
We call $\cup_{\pi}^{\mathrm{CT}}$ \emph{\neko{} (localized) Cassels-Tate pairing} on $\ppqll$.
This is the pairing denoted $\cup_{\pi(\mathfrak{p}_{f}),0,2,2}$ in \emph{loc. cit.}
We refer to Sections 2.10.14, 10.2 and 10.4 of \cite{Ne} for the  definition of $\boldsymbol{\cup_{\pi}^{\mathrm{CT} } }$.

\subsubsection{Comparison with Bloch-Kato Selmer groups}\label{compselclassicsec} 
Recall that  $V_{f}:=T_{f}\otimes_{\mathcal{O}_{L}}L$,
and  $V_{f,v}^{\pm}:=T_{f,v}^{\pm}\otimes_{\mathcal{O}_{L}}L$
for $v|p$.
Then $V_{f}\cong{}\ppql\otimes_{\Il,\phi_{f}}L$ is isomorphic to the $\phi_{f}$-base change of the localisation 
$\ppql$, and similarly $V_{f,v}^{\pm}$ is isomorphic to the $\phi_{f}$-base change of the localisation 
of $\mathbb{T}_{\mathbf{f},v}^{\pm}$ at $\mathfrak{p}_{f}$. 
By $(\ref{eq:ES})$, combined with the Chebotarev density theorem and \cite[Chapters V and VII]{Sil-1}, there is 
an  isomorphism of $L[G_{K_{\chi},S}]$-modules (cf. Section $\ref{ppq}$)
\begin{equation}\label{eq:isoV_{p}}
                     V_{f}\cong{}V_{p}(A)\otimes_{\Q_{p}}L,
\end{equation}
where $V_{p}(A):=\mathrm{Ta}_{p}(A)\otimes_{\Z_{p}}\Q_{p}$
is the $p$-adic Tate module of $A/\Q$ with $\Q_{p}$-coefficients.
We fix from now on such an isomorphism, and we will use it to identify $V_{f}$ with $V_{p}(A)\otimes_{\Q_{p}}L$. 

Consider the classical (or 
 Bloch-Kato \cite{B-K}) Selmer group  attached to $V_{p}(A)/K_{\chi}$ via Kummer theory:
\[
           \mathrm{Sel}_{p}(A/K_{\chi}):=\ker\lri{H^{1}(K_{\chi,S},V_{p}(A))\longrightarrow{}\prod_{v|p}\frac{H^{1}(K_{\chi,v},V_{p}(A))}
           {A(K_{\chi,v})\widehat{\otimes}\Q_{p}}}
\]
(it is easily verified using Tate local duality and \cite[Chapter VII]{Sil-1} that $H^{1}(K_{\chi,w},V_{p}(A))=0$ for $w\nmid{}p$),
sitting in a short exact sequence 
\begin{equation}\label{eq:kummera}
     0\fre{}A(K_{\chi})\widehat{\otimes}\Q_{p}\fre{}\mathrm{Sel}_{p}(A/K_{\chi})\fre{}V_{p}\lri{\sha(A/K_{\chi})}\fre{}0,
\end{equation}      where 
$\sha(A/K_{\chi})$ is the Tate-Shafarevich group of $A/K_{\chi}$ and $V_{p}(\cdot{}):=\inlim{}_{n\geq{}1}(\cdot{})_{p^{n}}\otimes_{\Z_{p}}\Q_{p}$ is the $p$-adic Tate module of the abelian group $(\cdot{})$ with $\Q_{p}$-coefficients.
R. Greenberg \cite{Gr-1} has proved that 
\[
                 \mathrm{Sel}_{p}(A/K_{\chi})\otimes_{\Q_{p}}L=\mathfrak{S}(K_{\chi},V_{f}).
\]
Since $a_{p}=a_{p}(A)=+1$ (as $A/\Q_{p}$ has split multiplicative reduction), the $G_{\Q_{p}}$-representation 
$V_{f}=V_{p}(A)\otimes_{\Q_{p}}L$ is a Kummer extension of the trivial representation $L$, i.e. 
$V_{f,v}^{+}\cong{}L(1)$ and $V_{f,v}^{-}\cong{}L$ for every $v|p$
(where $L$ is the trivial representation of $G_{K_{\chi},v}$
and $L(1):=L\otimes_{\Q_{p}}\Q_{p}(1)$ is its Tate twist).
As $H^{0}(G_{K_{\chi},S},V_{f})\subset{}H^{0}(G_{K_{\chi},w},V_{f})=0$
for every $w\nmid{}p$ (by \cite[Chapter VII]{Sil-1} and local Tate duality),  $(\ref{eq:exselmer})$ gives rise to an exact sequence
\begin{equation}\label{eq:exselclassic}
               0\fre{}\bigoplus_{v|p}L\fre{}\exsel^{1}(K_{\chi},V_{f})\fre{}\mathrm{Sel}_{p}(A/K_{\chi})\otimes_{\Q_{p}}L\fre{}0.
\end{equation}
(See Section $\ref{exalg}$ below for more details.)
\subsubsection{Galois conjugation}\label{galconj} Let $X$ be as in Section $\ref{genenek}$.
Section $8$ of \cite{Ne} defines 
a natural action of $\mathrm{Gal}(K_{\chi}/\Q)$ on $\exsel^{q}(K_{\chi},X)$, making it a 
$R_{X}[\mathrm{Gal}(K_{\chi}/\Q)]$-module.
If $\tau$ is a nontrivial automorphism on $K_{\chi}$,
we will write $\tau(x)$ or $x^{\tau}$ for its action on $x\in{}\exsel^{q}(K_{\chi},X)$.
To be short, all the relevant constructions we discussed above commute with the action of $\mathrm{Gal}(K_{\chi}/\Q)$.
In particular, we mention the following properties.

\neko's global cup products $_{q}\cup^{\mathrm{Nek}}_{\pi_{X}}$ (defined in $(\ref{eq:nekgcp})$) 
are $\mathrm{Gal}(K_{\chi}/\Q)$-equivariant
\cite[Section 8]{Ne}.

\neko's Pontrjagin duality isomorphisms $(\ref{eq:nekdu1})$ are $\mathrm{Gal}(K_{\chi}/\Q)$-equivariant 
\cite[Prop. 8.8.9]{Ne}.

The abstract Cassels-Tate pairing $\cup_{\pi}^{\mathrm{CT}}$ is $\mathrm{Gal}(K_{\chi}/\Q)$-equivariant \cite[Section 10.3.2]{Ne}.

The exact sequences $(\ref{eq:exselmer})$, $(\ref{eq:speco})$ and $(\ref{eq:exselclassic})$ are  $\mathrm{Gal}(K_{\chi}/\Q)$-equivariant.
(In case $K_{\chi}/\Q$ is quadratic and $p$ splits in $K_{\chi}$, the action of the non-trivial element $\tau\in{}\mathrm{Gal}(K_{\chi}/K)$ 
on the first term $\bigoplus_{v|p}L=L\oplus{}L$ in $(\ref{eq:exselclassic})$ is given by permutation of the factors:
$(q,q^{\prime})^{\tau}=(q^{\prime},q)$ for every $q,q^{\prime}\in{}L$.)

\subsection{The half-twisted weight pairing}\label{htwp} Define \emph{\neko's half-twisted weight pairing} by the composition
\begin{align*}
              \dia{-,-}_{V_{f},\pi}^{\mathrm{Nek}} : \exsel^{1}(K_{\chi},V_{f})\otimes_{L}\exsel^{1}(K_{\chi},V_{f})
              \stackrel{i_{\mathrm{wt}}\otimes{}i_{\mathrm{wt}}}{\longrightarrow{}} & 
              \exsel^{2}(K_{\chi},\ppql)[\varpi_{\mathrm{wt}}]\otimes_{\Il}
              \exsel^{2}(K_{\chi},\ppql)[\varpi_{\mathrm{wt}}] \\
              \stackrel{\cup_{\pi}^{\mathrm{CT}}}{\longrightarrow{}}  
              \Big(\mathrm{Frac}(\Il)/\Il\Big)&[\varpi_{\mathrm{wt}}]\stackrel{\theta_{\mathrm{wt}}}{\cong}
              \Il/\mathfrak{p}_{f}\Il\stackrel{\phi_{f}}{\cong{}}L\stackrel{\times{}\ell_{\mathrm{wt}}}{\cong}L,
\end{align*}
where the notations are as follows. The morphism $i_{\mathrm{wt}} : \exsel^{1}(K_{\chi},V_{f})\fre{}\exsel^{2}(K_{\chi},\ppql)[\varpi_{\mathrm{wt}}]$
is the one appearing in the exact sequence  $(\ref{eq:speco})$ (taking $q=1$).
$\cup^{\mathrm{CT}}_{\pi}$ is \neko's  Cassels-Tate pairing attached to $\pi : \ppq\otimes_{\I}\ppq\fre{}\I(1)$,
and defined in Section  $\ref{ctsec}$. $\theta_{\mathrm{wt}} : \big(\mathrm{Frac}(\Il)/\Il\big)[\varpi_{\mathrm{wt}}]\cong{}\Il/\mathfrak{p}_{f}\Il$
is defined by $\theta_{\mathrm{wt}}\big(\frac{a}{\varpi_{\mathrm{wt}}}\ \mathrm{mod}\ \Il\big):=a\ \mathrm{mod}\ \mathfrak{p}_{f}$,
for every $a\in{}\Il$.
(We remind  that $\varpi_{\mathrm{wt}}\in{}\iw$ is a uniformiser of $\Il$ by $(\ref{eq:unif})$).
Finally, $\ell_{\mathrm{wt}}:=\log_{p}(\gamma_{\mathrm{wt}})$ (where $\varpi_{\mathrm{wt}}:=\gamma_{\mathrm{wt}}-1$).
Note that both the morphisms $i_{\mathrm{wt}}$ and $\theta_{\mathrm{wt}}$ depend on the choice of the uniformiser 
$\varpi_{\mathrm{wt}}$.
Multiplication by $\ell_{\mathrm{wt}}$ serves the purposes of removing the dependence on this choice.

Since $\cup_{\pi}^{\mathrm{CT}}$ is a skew-symmetric, $\mathrm{Gal}(K_{\chi}/\Q)$-equivariant pairing,
and since $i_{\mathrm{wt}}$ is a $\mathrm{Gal}(K_{\chi}/\Q)$-equivariant morphism (cf. Section $\ref{galconj}$),
$\dia{-,-}_{V_{f},\pi}^{\mathrm{Nek}}$ is a \emph{skew-symmetric, $\mathrm{Gal}(K_{\chi}/\Q)$-equivariant
pairing}. (Of course, here we consider on $L$ the trivial $\mathrm{Gal}(K_{\chi}/\Q)$-action.) 

The aim of  this section is to prove the following key proposition,
whose proof uses all the power of \neko's results mentioned above. 
Let $\chi$ be (as above) a quadratic Dirichlet character of conductor coprime with $Np$.
Write $\dia{-,-}_{V_{f},\pi}^{\mathrm{Nek},\chi}$
for the restriction of $\dia{-,-}_{V_{f},\pi}^{\mathrm{Nek}}$ to 
$\exsel^{1}(K_{\chi},V_{f})^{\chi}\otimes_{L}\exsel^{1}(K_{\chi},V_{f})^{\chi}$.
(Of course, if $\chi$ is the trivial character, i.e. if $K_{\chi}=\Q$, we are defining nothing new.)
Given an $\I$-module $M$,
we say that $M$ is \emph{semi-simple at $\mathfrak{p}_{f}$} if $M_{\mathfrak{p}_{f}}$
is a semi-simple $\Il$-module,
and we write 
$\mathrm{length}_{\mathfrak{p}_{f}}(M)$
to denote the length of $M_{\mathfrak{p}_{f}}$ over $\Il$.

\begin{proposition}\label{nondeg} Let $\chi$ be a quadratic Dirichlet character of conductor coprime with $Np$,
and assume that $p$ splits in $K_{\chi}$. 
Then the following conditions are equivalent:

$1.$ $\dia{-,-}_{V_{f},\pi}^{\mathrm{Nek},\chi}$ is a  non-degenerate $L$-bilinear form on $\exsel^{1}(K_{\chi},V_{f})^{\chi}$.

$2.$ $$\mathrm{length}_{\mathfrak{p}_{f}}\Big(\exsel^{2}(K_{\chi},\ppq)^{\chi}\Big)=\dim_{L}
\Big(\exsel^{1}(K_{\chi},V_{f})^{\chi}\Big).$$

$3.$ $\exsel^{2}(K_{\chi},\ppq)^{\chi}$ is a torsion $\I$-module, which is semi-simple at $\mathfrak{p}_{f}$.\\
If these properties are satisfied, then $X_{\mathrm{Gr}}^{\mathrm{cc}}(\mathbf{f}/K_{\chi})^{\chi}$
is a torsion $\I$-module, which is semi-simple at $\mathfrak{p}_{f}$, and 
$$\mathrm{length}_{\mathfrak{p}_{f}}\Big(X_{\mathrm{Gr}}^{\mathrm{cc}}(\mathbf{f}/K_{\chi})^{\chi}\Big)
=\dim_{\Q_{p}}\Big(\mathrm{Sel}_{p}(A/K_{\chi})^{\chi}\Big).$$
\end{proposition}

The proposition will be an immediate consequence of the following three lemmas (in which we will prove separately 
the equivalences  $1\iff{}3$, $3\iff{}2$ and the last statement  respectively).

\begin{lemma}\label{seminondeg} $\dia{-,-}^{\mathrm{Nek},\chi}_{V_{f},\pi}$ is non-degenerate 
if and only if 
$\exsel^{2}(K_{\chi},\ppql)^{\chi}$ is a torsion, semi-simple $\Il$-module.
\end{lemma}
\begin{proof} Taking the $\chi$-component of the exact sequence $(\ref{eq:speco})$, we see that the restrictions
\[
               i_{\mathrm{wt}}^{\chi}=i_{\mathrm{wt}}^{q,\chi} : \exsel^{q}(K_{\chi},V_{f})^{\chi}\longrightarrow{}\exsel^{q+1}(K_{\chi},\ppql)^{\chi}[\varpi_{\mathrm{wt}}]
\]
of the morphisms $i_{\mathrm{wt}}=i_{\mathrm{wt}}^{q}$ defined in 
$(\ref{eq:speco})$ are surjective. Since $\exsel^{0}(K_{\chi},V_{f})\subset{}H^{0}(G_{K_{\chi,S}},V_{f})=0$,
this implies in particular that $\exsel^{1}(K_{\chi},\ppql)^{\chi}$ is torsion free,
and $i_{\mathrm{wt}}^{1,\chi}$
is  injective if and only if $\exsel^{1}(K_{\chi},\ppql)^{\chi}=0$.
Moreover, since $\chi$ is quadratic and $_{q}\cup^{\mathrm{Nek}}_{\pi}$
is $\mathrm{Gal}(K_{\chi}/\Q)$-equivariant,  the  duality  isomorphism $(\ref{eq:ptcom})$ shows that 
the latter condition is equivalent 
to the fact that 
$\exsel^{2}(K_{\chi},\ppql)^{\chi}$ is a torsion $\Il$-module. 

Write for simplicity $N:=\exsel^{2}(K_{\chi},\ppql)_{\mathrm{tors}}$ for the $\Il$-torsion submodule of $\exsel^{2}(K_{\chi},\ppql)$.
Since $\cup_{\pi}^{\mathrm{CT}}$ is  $\mathrm{Gal}(K_{\chi}/\Q)$-equivariant, $p\not=2$ and $\chi$
is quadratic, the isomorphism $(\ref{eq:ctnek})$
restricts to an isomorphism
\[
                 \mathrm{adj}\lri{\cup_{\pi}^{\mathrm{CT}}} : N^{\chi}\cong{}\Hom{\Il}(N^{\chi},\mathrm{Frac}(\Il)/\Il).
\]
Let $\cup_{\pi,\varpi_{\mathrm{wt}}}^{\mathrm{CT},\chi} : N^{\chi}[\varpi_{\mathrm{wt}}]
\otimes{}N^{\chi}[\varpi_{\mathrm{wt}}]\fre{}\lri{\mathrm{Frac}(\Il)/\Il}[\varpi_{\mathrm{wt}}]$ denote the restriction of $\cup^{\mathrm{CT}}_{\pi}$
to the $\varpi_{\mathrm{wt}}$-torsion of $N^{\chi}$.
It follows by the preceding isomorphism that the right (or left) 
radical of $\cup_{\pi,\varpi_{\mathrm{wt}}}^{\mathrm{CT},\chi}$ equals 
$\mathcal{N}^{\chi}:=\varpi_{\mathrm{wt}}N^{\chi}\cap{}N^{\chi}[\varpi_{\mathrm{wt}}]$.
In other words, $\cup_{\pi,\varpi_{\mathrm{wt}}}^{\mathrm{CT},\chi}$ is non-degenerate if and only if 
$\mathcal{N}^{\chi}=0$.
On the other hand, as $\varpi_{\mathrm{wt}}$ is a uniformiser for $\Il$, 
the structure theorem for finite modules over discrete valuation rings gives
an isomorphism of $\Il$-modules $N^{\chi}\cong{}
\bigoplus_{j=0}^{\infty}\lri{\Il/(\varpi_{\mathrm{wt}})^{j}}^{e_{j}}$, for positive integers $e_{j}$
such that $e_{j}=0$ for $j\gg0$.
Then $\mathcal{N}^{\chi}=0$ if and only if $e_{j}=0$ for every $j>1$, i.e. 
if and only if $N^{\chi}$ is semi-simple. 

Since $i_{\mathrm{wt}}^{\chi}=i_{\mathrm{wt}}^{1,\chi}$ is surjective, it follows by the definitions that  $\dia{-,-}^{\mathrm{Nek},\chi}_{V_{f},\pi}$
is non-degenerate (i.e. has trivial right$=$left radical) if and only if $i_{\mathrm{wt}}^{\chi}$ is injective 
and $\cup_{\pi,\varpi_{\mathrm{wt}}}^{\mathrm{CT},\chi}$ has trivial radical. Together 
with the preceding discussion, this concludes the proof of the lemma.
\end{proof}

\begin{lemma}\label{lala} $\mathrm{length}_{\mathfrak{p}_{f}}\Big(\exsel^{2}(K_{\chi},\mathbb{T}_{\mathbf{f}})^{\chi}
\Big)\geq{}\dim_{L}\Big(\exsel^{1}(K_{\chi},V_{f})^{\chi}\Big)$,
and equality holds if and only if $\exsel^{2}(K_{\chi},\ppql)^{\chi}$ is a torsion, semi-simple $\Il$-module. 
\end{lemma}
\begin{proof} Write for simplicity $\varpi:=\varpi_{\mathrm{wt}}$, $M_{\ast}:=\exsel^{\ast}(K_{\chi},\ppql)^{\chi}$, and $\mathscr{M}_{\ast}:=
\exsel^{\ast}(K_{\chi},V_{f})^{\chi}$,
so that there are short exact sequences of $L$-modules $(\ref{eq:speco})$:
$0\fre{}M_{q}/\varpi\fre{}\mathscr{M}_{q}\fre{}M_{q+1}[\varpi]\fre{}0.$ 
We can assume that $M_{2}$ is a torsion $\Il$-module, hence $M_{1}=0$
by the duality isomorphism $(\ref{eq:ptcom})$
(cf. the preceding proof). Then  
$\mathscr{M}_{1}\cong{}M_{2}[\varpi]$ and
\begin{equation}\label{eq:gqr1}
         \dim_{L} \mathscr{M}_{1}=\dim_{L} M_{2}[\varpi].
\end{equation}
The structure theorem for finite, torsion modules over principal ideal domains yields an isomorphism 
\[
                        M_{2}=\bigoplus_{j=1}^{\infty}\lri{\Il/\varpi^{j}}^{m(j)},
\]
where  $m : \mathbf{N}\fre{}\mathbf{N}$ is a function such that $m(j)=0$ for $j\gg0$.  Since 
$\lri{\Il/\varpi^{j}}[\varpi]\cong{}\Il/\varpi$ for $j\geq{}1$:
\[
             \mathrm{length}_{\mathfrak{p}_{f}} M_{2}=\sum_{j=0}^{\infty}m(j)\cdot{}j
             =\sum_{j=1}^{\infty}m(j)+\sum_{j=2}^{\infty}m(j)\cdot{}(j-1)
             =\dim_{L} M_{2}[\varpi]+\sum_{j=2}^{\infty}m(j)\cdot{}(j-1).
\]
Together with $(\ref{eq:gqr1})$, this gives $\mathrm{length}_{\mathfrak{p}_{f}} M_{2}\geq{}\dim_{L} \mathscr{M}_{1}$,
with equality if and only if $m(j)=0$ for every $j\geq{}2$, i.e. if and only if $M_{2}$ is a semi-simple $\Il$-module. 
\end{proof}

\begin{lemma}\label{lalaf} Assume that $\exsel^{2}(K_{\chi},\ppql)^{\chi}$ is a torsion, semi-simple $\Il$-module.
Then  $X_{\mathrm{Gr}}^{\mathrm{cc}}(\mathbf{f}/K_{\chi})^{\chi}\otimes_{\I}\Il$
is  a torsion, semi-simple $\Il$-module, and 
\[
      \mathrm{length}_{\mathfrak{p}_{f}}\Big( X_{\mathrm{Gr}}^{\mathrm{cc}}(\mathbf{f}/K_{\chi})^{\chi}\Big)=
\dim_{\Q_{p}} \Big(\mathrm{Sel}_{p}(A/K_{\chi})^{\chi}\Big).
\]
\end{lemma}
\begin{proof} Since $\mathrm{adj}(\pi) : \mathbb{T}_{\mathbf{f}}\cong{}\Hom{\I}(\mathbb{T}_{\mathbf{f}},\I(1))$
and $\ppq$ is a free $\I$-module,
there is  a canonical isomorphism of $\I[G_{K_{\chi},S}]$-modules
$$\mathbb{T}_{\mathbf{f}}\otimes_{\I}\I^{\ast}\cong{}\Hom{\I}(\ppq,\I(1))
\otimes_{\I}\Hom{\mathrm{cont}}(\I,\divp)\cong{}\Hom{\mathrm{cont}}(\ppq,\mu_{p^{\infty}})=:\mathbb{A}_{\mathbf{f}},$$
the second isomorphism being defined by composition: $\psi\otimes\mu\mapsto{}\mu\circ{}\psi$.  Similarly,
the isomorphism of $\I[G_{\Q_{p}}]$-modules 
$\mathrm{adj}(\pi) : \ppq^{-}\cong{}\Hom{\I}(\ppq^{+},\I(1))$
gives  an isomorphism 
of $\I[G_{\Q_p}]$-modules
$\mathbb{T}_{\mathbf{f}}^-\otimes_{\I}\I^{\ast}\cong{}\mathbb{A}_{\mathbf{f}}^-$. (Recall here that $\mathbb{A}_{\mathbf{f}}$
and $\mathbb{A}_{\mathbf{f}}^-$ are the Kummer duals of $\ppq$ and $\ppq^+$ respectively.)
This implies that $\mathrm{Sel}_{\mathrm{Gr}}^{\mathrm{cc}}(\mathbf{f}/K_{\chi})=\mathfrak{S}\lri{K_{\chi},\mathbb{A}_{\mathbf{f}}}$.
(Note that $\mathbb{A}_{\mathbf{f},w}^{-}:=0$ for every $S\ni{}w\nmid{}p$, so that we impose no condition at $w\nmid{}p$
in both the definitions of $\mathrm{Sel}_{\mathrm{Gr}}^{\mathrm{cc}}(\mathbf{f}/K_{\chi})$ and $\mathfrak{S}(K_{\chi},\mathbb{A}_{\mathbf{f}})$.)
By $(\ref{eq:exselmer})$ one then obtains an exact sequence
\begin{equation}\label{eq:rio1}
                 H^{0}(G_{K_{\chi},S},\mathbb{A}_{\mathbf{f}})\fre{}\bigoplus_{v|p}H^{0}(K_{\chi,v},\mathbb{A}_{\mathbf{f},v}^{-})
                 \fre{}\exsel^{1}(K_{\chi},\mathbb{A}_{\mathbf{f}})\fre{}\mathrm{Sel}_{\mathrm{Gr}}^{\mathrm{cc}}(\mathbf{f}/K_{\chi})\fre{}0.
\end{equation}
We claim that the localisation at $\mathfrak{p}_{f}$ of the Pontrjagin dual of $H^{0}(G_{K_{\chi},S},\mathbb{A}_{\mathbf{f}})$ vanishes, i.e.
\begin{equation}\label{eq:rio2}
                    H^{0}(G_{K_{\chi},S},\mathbb{A}_{\mathbf{f}})^{\ast}_{\mathfrak{p}_{f}}:=
                    \Hom{\Z_{p}}\Big(H^{0}(G_{K_{\chi},S},\mathbb{A}_{\mathbf{f}}),\divp\Big)\otimes_{\I}\Il=0.
\end{equation}
Indeed, let $w$ be a prime of $K_{\chi}$. By Tate local duality, $H^{0}(K_{\chi,w},\mathbb{A}_{\mathbf{f}})$
is the Pontrjagin dual of $H^{2}(K_{\chi,w},\mathbb{T}_{\mathbf{f}})$,
so that the inclusion $H^{0}(G_{K_{\chi},S},\mathbb{A}_{\mathbf{f}})\subset{}H^{0}(K_{\chi,w},\mathbb{A}_{\mathbf{f}})$
induces a  surjection $H^{2}(K_{\chi,w},\ppql)\twoheadrightarrow{}H^{0}(G_{K_{\chi},S},\mathbb{A}_{\mathbf{f}})^{\ast}_{\mathfrak{p}_{f}}$ on
(localised) Pontrjagin duals. As $\dercts(K_{\chi,w},\ppql)\cong{}0\in{}\mathrm{D}(\Il)$ is acyclic for every 
prime $w\nmid{}p$ (as easily proved, cf. \cite[Proposition 12.7.13.3(i)]{Ne}),  the claim $(\ref{eq:rio2})$ follows.  Since $\exsel^{1}(K_{\chi},\mathbb{A}_{\mathbf{f}})$ is the Pontrjagin dual of $\exsel^{2}(K_{\chi},\ppq)$
by \neko's duality isomorphism $(\ref{eq:nekdu1})$, applying first $\Hom{\Z_{p}}(-,\divp)$ and then $-\otimes_{\I}\Il$
to $(\ref{eq:rio1})$, and using $(\ref{eq:rio2})$, yield a short exact sequence of $\Il$-modules
\begin{equation}\label{eq:rio3}
          0\fre{}X_{\mathrm{Gr}}^{\mathrm{cc}}(\mathbf{f}/K_{\chi})\otimes_{\I}\Il\fre{}\exsel^{2}(K_{\chi},\ppql)\fre{}\bigoplus_{v|p}
          H^{2}(K_{\chi,v},\mathbb{T}_{\mathbf{f},v}^{+}\otimes_{\I}\Il)\fre{}0,
\end{equation}
where we used once again local Tate duality to rewrite the Pontrjagin dual of $H^{0}(K_{\chi,v},\mathbb{A}_{\mathbf{f},v}^{-})$
as $H^{2}(K_{\chi,v},\mathbb{T}_{\mathbf{f},v}^{+})$.
Lemma $\ref{sublemma}$ below gives an  isomorphism of $\Il$-modules
\[
                 H^{2}(K_{\chi,v},\mathbb{T}_{\mathbf{f},v}^{+}\otimes_{\I}\Il)\cong{}H^{2}(\Q_{p},\mathbb{T}_{\mathbf{f}}^{+}\otimes
                 _{\I}\Il)\cong{}\Il/\mathfrak{p}_{f}\Il,
\]
for every $v|p$. Since $p$ splits in (the at most quadratic field) $K_{\chi}$,
taking the $\chi$-component of $(\ref{eq:rio3})$ gives a short exact sequence of $\Il$-modules
\[
                    0\fre{}X_{\mathrm{Gr}}^{\mathrm{cc}}(\mathbf{f}/K_{\chi})^{\chi}\otimes_{\I}\Il
                    \fre{}\exsel^{2}(K_{\chi},\ppql)^{\chi}\fre{}\Il/\mathfrak{p}_{f}\Il\fre{}0.
\]
(Note that, if  $\chi$ is nontrivial, the nontrivial automorphism of $\mathrm{Gal}(K_{\chi}/\Q)$
acts by permuting the factors in the sum
$H^{2}(K_{\chi,v_{1}},\mathbb{T}_{\mathbf{f},v_{1}}^{+}\otimes_{\I}\Il)\oplus{}H^{2}(K_{\chi,v_{2}},\mathbb{T}_{\mathbf{f},v_{2}}^{+}\otimes_{\I}\Il)
=:V\oplus{}V$, where $\{v|p\}=\{v_{1},v_{2}\}$. Then the  $\epsilon$-component of $V\oplus{}V$ 
 is equal to either the subspace $\{(v,v) : v\in{}V\}\cong{}V$ if $\epsilon=1$ or to 
$\{(v,-v) : v\in{}V\}\cong{}V$ if $\epsilon=\chi$.)
In particular,  $X_{\mathrm{Gr}}^{\mathrm{cc}}(\mathbf{f}/K_{\chi})^{\chi}$ is a torsion module, 
which is semi-simple at $\mathfrak{p}_{f}$
if $\exsel^{2}(K_{\chi},\ppql)^{\chi}$ is. Moreover, if $\exsel^{2}(K_{\chi},\ppql)^{\chi}$
is indeed semi-simple, the preceding equation and  Lemma $\ref{lala}$ give
\[
              \mathrm{length}_{\mathfrak{p}_{f}}\Big(X_{\mathrm{Gr}}^{\mathrm{cc}}(\mathbf{f}/K_{\chi})^{\chi}\Big)=
              \mathrm{length}_{\mathfrak{p}_{f}}\Big(\exsel^{2}(K_{\chi},\ppq)^{\chi}\Big)-1=
              \dim_{L}\Big(\exsel^{1}(K_{\chi},V_{f})^{\chi}\Big)-1.
\]
Since 
  $\dim_{L}\exsel^{1}(K_{\chi},V_{f})^{\chi}=\dim_{\Q_{p}}\mathrm{Sel}_{p}(A/K_{\chi})^{\chi}+1$
by $(\ref{eq:exselclassic})$, this concludes the proof of the lemma.
\end{proof}

\begin{lemma}\label{sublemma} $H^{2}(\Q_{p},\mathbb{T}_{\mathbf{f}}^{+}\otimes_{\I}\Il)\cong{}\Il/\mathfrak{p}_{f}\Il$.
\end{lemma}
\begin{proof} Write $\varpi:=\varpi_{\mathrm{wt}}$.
Since $\mathbb{T}_{\mathbf{f}}^{+}\otimes_{\I}\Il/\varpi\cong{}V_{f}^{+}\cong{}L(1)$ as $G_{\Q_{p}}$-modules 
(see Section $\ref{compselclassicsec}$), there are  short exact sequences of $L$-modules
\begin{equation}\label{eq:676754}
        0\fre{}H^{j}(\Q_{p},\mathbb{T}_{\mathbf{f}}^{+}\otimes_{\I}\Il)/\varpi\fre{}H^{j}(\Q_{p},
        \Q_{p}(1))\otimes_{\Q_{p}}L
        \fre{}H^{j+1}(\Q_{p},\mathbb{T}_{\mathbf{f}}^{+}\otimes_{\I}\Il)[\varpi]\fre{}0.
\end{equation}
Taking $j=0$ one  finds $H^{1}(\Q_{p},\mathbb{T}_{\mathbf{f}}^{+}\otimes_{\I}\Il)[\varpi]=0$,
i.e. $H^{1}(\Q_{p},\mathbb{T}_{\mathbf{f}}^{+}\otimes_{\I}\Il)$ is a free $\Il$-module.  
It is immediately seen by the explicit description of $\mathbb{T}_{\mathbf{f}}^{\pm}$ given in $(\ref{eq:addequation})$
that $H^{0}(\Q_{p},\mathbb{T}_{\mathbf{f}}^{+})=0$ and $H^{0}(\Q_{p},\mathbb{T}_{\mathbf{f}}^{-})=0$.
Since $\mathbb{T}_{\mathbf{f}}^{-}\cong{}\Hom{\I}(\mathbb{T}_{\mathbf{f}}^{+},\I(1))$ (under the duality $\pi$ from Section $\ref{ppq}$),
Tate local duality tells us that $H^{2}(\Q_{p},\mathbb{T}_{\mathbf{f}}^{+})$ is a torsion $\I$-module. Since $\mathbb{T}_{\mathbf{f}}^{+}$
is free of rank one over $\I$, 
Tate's formula for the local Euler characteristic now gives 
$\sum_{k=0}^{2}(-1)^{k}\mathrm{rank}_{\I}H^{j}(\Q_{p},\mathbb{T}_{\mathbf{f}}^{+})
=-1$. Together with what already proved, this allows us to conclude $H^{1}(\Q_{p},\mathbb{T}_{\mathbf{f}}^{+}\otimes_{\I}\Il)\cong{}\Il$.
Taking now $j=1$ and $j=2$ in $(\ref{eq:676754})$ we find exact sequences
\begin{align*}
                   0\fre{}\Il/\varpi\fre{}H^{1}(\Q_{p},\Q_{p}(1))\otimes_{\Q_{p}}L\fre{}
                   H^{2}(\Q_{p},\mathbb{T}^{+}_{\mathbf{f}} &\otimes_{\I}\Il)[\varpi]\fre{}0;   \\
                   H^{2}(\Q_{p},\mathbb{T}^{+}_{\mathbf{f}}\otimes_{\I}\Il)/\varpi\cong{}
                   H^{2}(\Q_{p},\Q_{p}(1))&\otimes_{\Q_{p}}L. 
\end{align*}
Since $\mathrm{dim}_{\Q_{p}}H^{1}(\Q_{p},\Q_{p}(1))=2$ and 
$\dim_{\Q_{p}}H^{2}(\Q_{p},\Q_{p}(1))=1$, and since $\Il/\varpi\cong{}L$,
it follows that both the $\varpi$-torsion  $H^{2}(\Q_{p},\mathbb{T}_{\mathbf{f}}^{+}\otimes_{\I}\Il)[\varpi]$
and the $\varpi$-cotorsion $H^{2}(\Q_{p},\mathbb{T}^{+}_{\mathbf{f}}\otimes_{\I}\Il)/\varpi$ have dimension $1$
over $L=\Il/\varpi$. The structure theorem for finite torsion modules over principal ideal domains then gives 
$H^{2}(\Q_{p},\mathbb{T}_{\mathbf{f}}^{+}\otimes_{\I}\Il)\cong{}\Il/\varpi^{n}$ for some $n\geq{}1$.
To conclude the proof, it remains to show  that $n=1$, 
i.e. that $H^{2}(\Q_{p},\mathbb{T}_{\mathbf{f}}^{+}\otimes_{\I}\Il)$ is semi-simple, or equivalently that the composition  
\[
         \mathcal{H} : 
         H^{1}(\Q_{p},L(1))\twoheadrightarrow{}H^{2}(\Q_{p},\mathbb{T}_{\mathbf{f}}^{+}\otimes_{\I}\Il)[\varpi]\hookrightarrow{}
          H^{2}(\Q_{p},\mathbb{T}_{\mathbf{f}}^{+}\otimes_{\I}\Il)\twoheadrightarrow{}
          H^{2}(\Q_{p},\mathbb{T}_{\mathbf{f}}^{+}\otimes_{\I}\Il)/\varpi
          \cong{}H^{2}(\Q_{p},L(1))\stackrel{\mathrm{inv}_{p}\ }{\cong{}}L
\]
is non-zero. To do this, identify $H^{1}(\Q_{p},L(1))\cong{}\Q_{p}^{\times}\widehat{\otimes}L$
via Kummer theory, and let $q\in{}\Q_{p}^{\times}$. We want to compute the image $\mathcal{H}(q)=\mathcal{H}(q\widehat{\otimes}1)
\in{}L$.
Identify $\mathbb{T}_{\mathbf{f}}^{+}$ with $\I(\mathbf{a}_{p}^{\ast-1}\chi_{\mathrm{cy}}
[\chi_{\mathrm{cy}}]^{1/2})$
(cf. Section $\ref{ppq}$), and  write $c_{q} : G_{\Q_{p}}\fre{}L(1)$ for a $1$-cocycle representing 
$q\widehat{\otimes}1$.
Since $\Il$ is a $L$-algebra and $\phi_{f} : \Il\twoheadrightarrow{}L$
is a morphism of $L$-algebras, one can consider  $c_{q} : G_{\Q_{p}}\fre{}\mathbb{T}_{\mathbf{f}}^{+}\otimes_{\I}\Il$
as $1$-cochain  which lifts $c_{q}$
under $\phi_{f}$. The differential (in $\ctsb(\Q_{p},\mathbb{T}_{\mathbf{f}}^{+}\otimes_{\I}\Il)$)
of $c_{q}$ is then given by
\begin{align*}
    dc_{q}(g,h) & =\mathbf{a}_{p}^{\ast}(g)^{-1}\cdot{}\chi_{\mathrm{cy}}(g)\cdot{}[\chi_{\mathrm{cy}}(g)]^{1/2}\cdot{}c_{q}(h)
    -c_{q}(gh)+c_{q}(g) \\
    &= \chi_{\mathrm{cy}}(g)\cdot{}\lri{\mathbf{a}_{p}^{\ast}(g)^{-1}\cdot{}[\chi_{\mathrm{cy}}(g)]^{1/2}-1}\cdot{}c_{q}(h),
\end{align*}
where we used the cocyle relation (in $\ctsb(\Q_{p},L(1))$) for the second equality.
Retracing the definitions given above, the class $\mathcal{H}(q)$ is then the image
under $\mathrm{inv}_{p}$ of the class represented by the $2$-cocycle
\begin{equation}\label{eq:nbnbnb@}
            \vartheta(g,h):=\chi_{\mathrm{cy}}(g)\cdot{}c_{q}(h)
            \cdot{}\phi_{f}\lri{\frac{\mathbf{a}_{p}^{\ast}(g)^{-1}\cdot{}[\chi_{\mathrm{cy}}]^{1/2}(g)-1}{\varpi}}\in{}L(1).
\end{equation}
Consider the Tate local cup-product pairing $\dia{-,-}^{\mathrm{Tate}}_{\Q_{p}}
: H^{1}(\Q_{p},L)\times{}H^{1}(\Q_{p},L(1))\fre{}L$. Noting that $$\Phi_{\mathbf{f}}:=\phi_{f}\lri{\frac{\mathbf{a}_{p}^{\ast-1}\cdot{}[\chi_{\mathrm{cy}}]^{1/2}-1}{\varpi}}
\in{}\Hom{\mathrm{cont}}(G_{\Q_{p}}^{\mathrm{ab}},L)=H^{1}(\Q_{p},L),$$
the equality $(\ref{eq:nbnbnb@})$ can be rewritten as 
\begin{equation}\label{eq:funtate}
                \mathcal{H}(q)=\mathrm{inv}_{p}\big(\mathrm{class\ of\ }\vartheta\big)=\dia{\Phi_{\mathbf{f}},q}^{\mathrm{Tate}}_{\Q_{p}}\in{}L.
\end{equation}
Let $g_{0}\in{}I_{\Q_{p}}$ be such that $\chi_{\mathrm{cy}}(g_{0})^{1/2}=\gamma_{\mathrm{wt}}$
(where $\varpi=[\gamma_{\mathrm{wt}}]-1$), 
and let $g\in{}I_{\Q_{p}}$. Then $\kappa_{\mathrm{cy}}(g)^{1/2}=\gamma_{\mathrm{wt}}^{z}$
for some $z\in{}\Z_{p}$, satisfying $\frac{1}{2}\log_{p}\lri{\chi_{\mathrm{cy}}(g)}=z\cdot{}\log_{p}(\gamma_{\mathrm{wt}})$. (Recall that $\kappa_{\mathrm{cy}} : G_{\Q_{p}}\twoheadrightarrow{}1+p\Z_{p}$
is the composition of the $p$-adic cyclotomic character $\chi_{\mathrm{cy}} : G_{\Q_{p}}\twoheadrightarrow{}\Z_{p}^{\times}$ with projection to principal units.)
Since $\mathbf{a}_{p}^{\ast}(g)=1$ 
this implies
\begin{equation}\label{eq:nbnbnb+}
                 \Phi_{\mathbf{f}}(g)= \phi_{f}\lri{\frac{\mathbf{a}_{p}^{\ast}(g)^{-1}\cdot{}
                 [\chi_{\mathrm{cy}}]^{1/2}(g)-1}{\varpi}}
                  =\phi_{f}\lri{\frac{[\gamma_{\mathrm{wt}}^{z}]-1}{\varpi}}=z
                  =\frac{1}{2}\frac{\log_{p}\big(\chi_{\mathrm{cy}}(g)\big)}{\log_{p}(\gamma_{\mathrm{wt}})}.
\end{equation}
Let now $\mathrm{Frob}_{p}\in{}\mathrm{Gal}(\Q_{p}^{\mathrm{un}}/\Q_{p})=:G_{\Q_{p}}^{\mathrm{un}}$
be an arithmetic Frobenius,
where $\Q_{p}^{\mathrm{un}}/\Q_{p}$ is the maximal unramified extension of $\Q_{p}$,
and  $G_{\Q_{p}}^{\mathrm{un}}$ is viewed  as a subgroup of  the abelianisation  $G_{\Q_{p}}^{\mathrm{ab}}$
of $G_{\Q_{p}}$ under the canonical 
 decomposition $G_{\Q_{p}}^{\mathrm{ab}}\cong{}\mathrm{Gal}(\Q_{p}(\mu_{p^{\infty}})/\Q_{p})
\times{}G_{\Q_{p}}^{\mathrm{un}}$. 
Using the Mellin transform introduced in Section $\ref{lochidafam}$, and the well-known formula of Greenberg-Stevens
\cite{G-S}:
$\frac{d}{dk}a_{p}(k)_{k=2}=-\frac{1}{2}\mathscr{L}_{p}(A)$, where $\mathscr{L}_{p}(A):=\frac{\log_{p}(q_{A})}{\mathrm{ord}_{p}(q_{A})}$
for the Tate period $q_{A}\in{}p\Z_{p}$ of $A/\Q_{p}$ (see the following section),
one easily computes 
\begin{equation}\label{eq:nbnbnb=}
            \Phi_{\mathbf{f}}(\mathrm{Frob}_{p}^{n})=\phi_{f}\lri{\frac{\mathbf{a}_{p}^{\ast}(\mathrm{Frob}_{p}^{n})^{-1}-1}{\varpi}}
            =\frac{1}{2}\mathscr{L}_{p}(A)\cdot{}\frac{n}{\log_{p}(\gamma_{\mathrm{wt}})}.
\end{equation}
Let $\mathrm{rec}_{p} : \Q_{p}^{\times}\fre{}G_{\Q_{p}}^{\mathrm{ab}}$
be the reciprocity map of local class field theory \cite{Ser}. By combining 
the explicit formula for $\mathrm{rec}_{p}$ given by Lubin-Tate theory with  formulae $(\ref{eq:nbnbnb+})$ and $(\ref{eq:nbnbnb=})$ above
yields
\[
               \Phi_{\mathbf{f}}\big(\mathrm{rec}_{p}(q)\big)= \phi_{f}\lri{\frac{\mathbf{a}_{p}^{\ast}(\mathrm{rec}_{p}(q))^{-1}\cdot{}[\chi_{\mathrm{cy}}]^{1/2}(\mathrm{rec}_{p}(q))-1}{\varpi}}
                  =-\frac{1}{2}\frac{1}{\log_{p}(\gamma_{\mathrm{wt}})}\cdot{}\log_{q_{A}}(q)
\]
for every $q\in{}\Q_{p}^{\times}$,
where $\log_{q_{A}} : \Q_{p}^{\times}\fre{}\Q_{p}$ is the branch of the $p$-adic logarithm vanishing at the Tate period 
$q_{A}$. Equation $(\ref{eq:funtate})$ and another application of local class field theory
then give (cf. \cite{Ser})
\[
             \mathcal{H}(q)=\dia{\Phi_{\mathbf{f}},q}^{\mathrm{Tate}}_{\Q_{p}}=
             \Phi_{\mathbf{f}}\big(\mathrm{rec}_{p}(q)\big)\stackrel{\cdot{}}{=}\log_{q_{A}}(q),
\]
where $\stackrel{\cdot{}}{=}$ denotes equality up to a non-zero factor. This clearly proves that $\mathcal{H}$ is non-zero,
hence (as explained above)
that $H^{2}(\Q_{p},\mathbb{T}_{\mathbf{f}}^{+}\otimes_{\I}\Il)$
is a semi-simple $\Il$-module.
This concludes the proof of the lemma.
\end{proof}

\subsection{Algebraic exceptional zero formulae}\label{exalg}
Since $A/\Q_{p}$ has split multiplicative reduction, it is a \emph{Tate curve} \cite{Tate-2}, \cite[Chapter V]{Sil-2}, i.e. isomorphic 
(as a rigid analytic variety) to a Tate curve $\mathbb{G}_{m}/q_{A}^{\Z}$ over $\Q_{p}$,
where $q_{A}\in{}p\Z_{p}$ is the so called Tate period of $A/\Q_{p}$.
In particular, there exists a $G_{\Q_{p}}$-equivariant isomorphism 
\begin{equation}\label{eq:tatepar}
     \Phi_{\mathrm{Tate}} : \overline{\Q}_{p}^{\times}/q_{A}^{\Z}\cong{}A(\overline{\Q}_{p}).
\end{equation}
Write $K_{\chi,p}:=K_{\chi}\otimes_{\Q}\Q_{p}\cong{}\prod_{v|p}K_{\chi,v}$, and write 
$\iota_{v} : K_{\chi}\hookrightarrow{}K_{\chi,v}\subset{}\overline{\Q}_{p}$ for the resulting embedding of $K_{\chi}$ in its completion at $v$.
Following \cite{M-T-T} and \cite{Be-Da1}, define the \emph{extended Mordell-Weil group} of $A/K_{\chi}$:
\[
                A^{\dag}(K_{\chi}):=\left\{\lri{P,(y_{v})_{v|p}}\in{}A(K_{\chi})\times{}K_{\chi,p}^{\times} : \Phi_{\mathrm{Tate}}(y_{v})=\iota_{v}(P),\ \mathrm{for\ 
                every\ } v|p\right\}.
\]
In concrete terms, an element of $A^{\dag}(K_{\chi})$ is a $K_{\chi}$-rational point os $A$, together with a distinguished lift under 
$\Phi_{\mathrm{Tate}}$ for every prime $v|p$. Then $A^{\dag}(K_{\chi})$ is an extension of the usual Mordell-Weil group $A(K_{\chi})$
by a free $\Z$-module of rank $\#\{v|p\}$. In other words there is  a short exact sequence
\begin{equation}\label{eq:extmor}
          0\fre{}\bigoplus_{v|p}\Z\fre{}A^{\dag}(K_{\chi})\fre{}A(K_{\chi})\fre{}0,
\end{equation}
where the first map sends the canonical $v$-generator  to 
\begin{equation}\label{eq:nbnbnew}
    q_{v}:=\big(0,q^{v}_{A}\big)\in{}A^{\dag}(K_{\chi}),
\end{equation}     $q_{A}^{v}\in{}K_{\chi,p}^{\times}$
being the element having $q_{A}$ as $v$-component and $1$ elsewhere. 
When $K_{\chi}/\Q$ is quadratic, $A^{\dag}(K_{\chi})$ has a natural $\mathrm{Gal}(K_{\chi}/\Q)$-action,
coming from the diagonal action on $A(K_{\chi})\times{}K_{\chi,p}^{\times}$
(with $\mathrm{Gal}(K_{\chi}/\Q)$
acting on $K_{\chi,p}:=K_{\chi}\otimes_{\Q}\Q_{p}$ via its action on the first component).
Recall the Kummer map $A(K_{\chi})\widehat{\otimes}\Q_{p}\hookrightarrow{}\mathrm{Sel}_{p}(A/K_{\chi})$ \cite[Chapter X]{Sil-1}.
The following lemma is proved in   \cite[Section 4]{PhD}
(see  in particular Lemma 4.1 and Lemma  4.3).
For every abelian group $\mathcal{A}$,
 write for simplicity  $\mathcal{A}\otimes{}L:=\lri{\mathcal{A}\widehat{\otimes}\Z_{p}}\otimes_{\Z_{p}}L$.

\begin{lemma}\label{lelap} There exists a unique injective and $\mathrm{Gal}(K_{\chi}/\Q)$-equivariant  morphism of $L$-modules
\[
            i_{A}^{\dag} : A^{\dag}(K_{\chi})\otimes{}L\longrightarrow{}\exsel^{1}(K_{\chi},V_{f}),
\]
satisfying the following properties:
\begin{itemize}
\item[$(i)$] $i_{A}^{\dag}$ gives rise to an injective morphism of short exact sequences of $L[\mathrm{Gal}(K_{\chi}/\Q)]$-modules:
\[
           \xymatrix{         0 \ar[r] & \bigoplus_{v|p}L \ar[r]\ar@{=}[d] & A^{\dag}(K_{\chi})\otimes{}L \ar[r]\ar@{^{(}->}[d]_{i_{A}^{\dag}} &
                                            A(K_{\chi})\otimes{}L \ar@{^{(}->}[d]^{\mathrm{Kummer}}\ar[r]  & 0\ \\
                                             0 \ar[r]   & \bigoplus_{v|p} L \ar[r] & \exsel^{1}(K_{\chi},V_{f}) \ar[r] & \mathrm{Sel}_{p}(A/K_{\chi})\otimes_{\Q_{p}}L
                                             \ar[r] & 0,                                           
                                                      }
\]
the bottom row being $(\ref{eq:exselclassic})$. 
\item[$(ii)$] Let $\mathbb{P}=(P,(y_{v})_{v|p})\in{}A^{\dag}(K_{\chi})$ be such that $y_{v}\in{}\mathcal{O}_{K_{\chi},v}^{\times}$ for every $v|p$.
Then the image of $i_{A}^{\dag}(\mathbb{P})$ under the natural map 
$\exsel^{1}(K_{\chi},V_{f})\fre{}\bigoplus_{v|p}H^{1}(K_{\chi,v},V_{f,v}^{+})$ lies in the finite subspace
$\bigoplus_{v|p}H^{1}_{f}(K_{\chi,v},V_{f,v}^{+})$ \footnote{More precisely,
by the definition of \neko's Selmer complexes, we have a natural surjective morphism of complexes $p_{f}^{+} : \derco(K_{\chi},V_{f})\twoheadrightarrow{}\bigoplus_{v|p}
\dercts(K_{\chi,v},V_{f,v}^{+})$. The map referred to in the lemma is the morphism induced in cohomology by $p_{f}^{+}$.
Moreover, we recall that the \emph{finite (of Bloch-Kato) subspace} $ H^{1}_{f}(K_{\chi,v},-)$
is defined to be the subspace of $H^{1}(K_{\chi,v},-)$ made of crystalline classes,
i.e. classes with trivial image in  $H^{1}(K_{\chi,v},-\otimes{}B_{\mathrm{cris}})$ \cite{B-K}.}. 
\end{itemize}
In particular, $i_{A}^{\dag} : A^{\dag}(K_{\chi})\otimes{}L\cong{}\exsel^{1}(K_{\chi},V_{f})$
is an isomorphism provided that $\sha(A/K_{\chi})_{p^{\infty}}$ is finite. 
\end{lemma}

We will consider from now on $A^{\dag}(K_{\chi})$ (or precisely $A^{\dag}(K_{\chi})/\mathrm{torsion}$)
as a submodule of $\exsel^{1}(K_{\chi},V_{f})$ via the injection $i_{A}^{\dag}$. In particular 
$\dia{P,Q}^{\mathrm{Nek}}_{V_{f},\pi}:=\big<i_{A}^{\dag}(P),i_{A}^{\dag}(Q)\big>^{\mathrm{Nek}}_{V_{f},\pi}$
for every $P,Q\in{}A^{\dag}(K_{\chi})$.

For every $\alpha\in{}\Z_{p}$,
let $\boldsymbol{\alpha}=(\alpha^{1/p},\alpha^{1/p^{2}},\dots)$ be a (fixed) compatible system of 
$p^{n}$-th roots of $\alpha$ in $\overline{\Q}_{p}$.
Using the Tate parametrisation (and recalling that $q_{A}\in{}p\Z_{p}$ has positive $p$-adic valuation),
we can identify $V_{p}(A)$ with the $\Q_{p}$-module generated by $\mathbf{1}\in{}\Z_{p}(1)$ and 
$\boldsymbol{q}_{A}$.
Thanks to our fixed isomorphism $(\ref{eq:isoV_{p}})$, the duality $\pi_{f}:=\pi\otimes_{\Il,\phi_{f}}L$
induces a  duality $\pi_{f} : V_{p}(A)\otimes_{\Q_{p}}V_{p}(A)\fre{}\Q_{p}(1)$.
Denote by $\pi_{f,\mathbf{1}} : V_{p}(A)\otimes_{\Q_{p}}V_{p}(A)\fre{}\Q_{p}$ the composition 
of $\pi_{f}$ with the isomorphism $\Q_{p}(1)\cong{}\Q_{p};\ \mathbf{1}\mapsto{}1$.
We can then state the main result of this section.

\begin{theo}\label{mainPhD} Let $\big(P,\widetilde{P}\big)\in{}A^{\dag}(K_{\chi})$, with $\widetilde{P}=\big(\widetilde{P}_{v}\big)_{v|p}\in{}K_{\chi,p}^{\times}$.
Then
\[
        \dia{q_{v},\big(P,\widetilde{P})}^{\mathrm{Nek}}_{V_{f},\pi}=c(\pi)\cdot{}
        \log_{q_{A}}\lri{N_{K_{\chi,v}/\Q_{p}}\big(\widetilde{P}_{v}\big)},
\]
where $\log_{q_{A}} : \overline{\Q}_{p}^{\times}\fre{}\overline{\Q}_{p}$ 
is the branch of the $p$-adic logarithm vanishing at $q_{A}$, $N_{K_{\chi,v}/\Q_{p}} : K_{\chi,v}^{\times}\fre{}\Q_{p}^{\times}$
is the  norm, and  the \emph{non-zero} constant $c(\pi)\in{}\Q_{p}^{\times}$ (depending on $\pi$,
but \emph{not} on $(P,\widetilde{P})$)
 is given by
$c(\pi)=\frac{1}{2}\pi_{f,\mathbf{1}}\lri{\mathbf{1}\otimes{}\boldsymbol{q_{A}}}$.
\end{theo}
\begin{proof} This is Corollary $4.6$ of \cite{PhD}. (In \emph{loc. cit.}  $\pi : \ppq\otimes_{\I}\ppq\fre{}\I(1)$ is normalised 
in such a way that $\pi_{f,\mathbf{1}}$ takes the value $1$ on $\mathbf{1}\otimes{}\boldsymbol{q_{A}}$,
so that the constant $c(\pi)$ becomes $1/2$.) For a more general statement, see also \cite{Ven}. 
\end{proof}

\subsection{Proof of Theorem $\ref{bchar}$} Assume that $\chi(p)=1$, i.e. that  $p$ splits in $K_{\chi}$. Moreover, assume that
\begin{equation}\label{eq:nuhyp}
                    \rank_{\Z}A(K_{\chi})^{\chi}=1;\ \ \#\Big(\sha(A/K_{\chi})_{p^{\infty}}^{\chi}\Big)<\infty,
\end{equation}
and let $P_{\chi}\in{}A(K_{\chi})^{\chi}$ be a generator of $A(K_{\chi})^{\chi}$ modulo torsion. 
Fix a lift $P^{\dag}_{\chi}=\big(P_{\chi},(\widetilde{P}_{\chi,v})_{v|p}\big)\in{}A^{\dag}(K_{\chi})^{\chi}$ of $P_{\chi}$
under $(\ref{eq:extmor})$, and define a \emph{$\chi$-period}
\[
                q_{\chi}\in{}A^{\dag}(K_{\chi})^{\chi}
\]
as follows. If $\chi$ is the trivial character, i.e. $K_{\chi}=\Q$, then let $q_{\chi}:=(0,q_{A})\in{}A^{\dag}(\Q)
\subset{}A(\Q)\times{}\Q_{p}^{\times}$.
Similarly, if $K_{\chi}/\Q$ is quadratic, let 
$q_{\chi}:=\big(0,(q_{A},q_{A}^{-1})\big)\in{}A^{\dag}(K_{\chi})^{\chi}\subset{}A(K_{\chi})\times{}K_{\chi,\mathfrak{p}}^{\times}\times{}K_{\chi,\overline{\mathfrak{p}}}^{\times}$, where $p\mathcal{O}_{K_{\chi}}=\mathfrak{p}\cdot{}\overline{\mathfrak{p}}$.
By the exact sequence of $\Z[\mathrm{Gal}(K_{\chi}/\Q)]$-modules  
$(\ref{eq:extmor})$, our assumptions, and Lemma $\ref{lelap}$ one has
\begin{equation}\label{eq:qfin}
                \exsel^{1}(K_{\chi},V_{f})^{\chi}\stackrel{i_{A}^{\dag}}{\cong}\lri{A(K_{\chi})\otimes{}L}^{\chi}=L\cdot{}q_{\chi}\oplus{}L\cdot{}P^{\dag}_{\chi}.
\end{equation}
Since $\dia{-,-}^{\mathrm{Nek}}_{V_{f},\pi}$ is a skew-symmetric bilinear form,  $\dia{q_{\chi},q_{\chi}}_{V_{f},\pi}^{\mathrm{Nek}}=0$
and $\dia{P_{\chi}^{\dag},P_{\chi}^{\dag}}^{\mathrm{Nek}}_{V_{f},\pi}=0$.
Moreover, in case $K_{\chi}=\Q$, Theorem $\ref{mainPhD}$ gives
\[
             \dia{q_{\chi},P_{\chi}^{\dag}}^{\mathrm{Nek}}_{V_{f},\pi}\stackrel{\cdot{}}{=}
             \log_{q_{A}}(\widetilde{P}_{\chi,p})=\log_{A}(P_{\chi}),
\]
where $\log_{A}:=\log_{q_{A}}\circ{}\Phi_{\mathrm{Tate}}^{-1} : A(\Q_{p})\cong{}\Q_{p}$ is the formal group logarithm on 
$A/\Q_{p}$, and $\stackrel{\cdot{}}{=}$ denotes equality up to multiplication by a non-zero element of $L^{\times}$.  
In case $K_{\chi}/\Q$ is quadratic, write as above $(p)=\mathfrak{p}\cdot{}\overline{\mathfrak{p}}$,
and  $\iota_{\mathfrak{p}} : K_{\chi}\subset{}K_{\chi,\mathfrak{p}}\cong{}\Q_{p}$ and  $\iota_{\overline{\mathfrak{p}}} : 
K_{\chi}\subset{}K_{\chi,\overline{\mathfrak{p}}}\cong{}\Q_{p}$
 for the completions of $K$ at $\mathfrak{p}$ and $\overline{\mathfrak{p}}$ respectively. 
Then $\iota_{\overline{\mathfrak{p}}}=\iota_{\mathfrak{p}}\circ{}\tau$, where $\tau$ is the non-trivial element of $\mathrm{Gal}(K_{\chi}/\Q)$.
Since $P_{\chi}^{\dag}\in{}A^{\dag}(K_{\chi})^{\chi}$, we have $P_{\chi}^{\tau}=-P_{\chi}$ and 
$\widetilde{P}_{\chi,\mathfrak{p}}=\widetilde{P}_{\chi,\overline{\mathfrak{p}}}^{-1}$.
As $q_{\chi}:=q_{\mathfrak{p}}-q_{\bar{\mathfrak{p}}}$ (by the definitions), another application of Theorem $\ref{mainPhD}$ allows us to compute
\begin{align*}
       \dia{q_{\chi},P_{\chi}^{\dag}}^{\mathrm{Nek}}_{V_{f},\pi}=
        \dia{q_{\mathfrak{p}},P_{\chi}^{\dag}}^{\mathrm{Nek}}_{V_{f},\pi}-
        \dia{q_{\overline{\mathfrak{p}}},P_{\chi}^{\dag}}^{\mathrm{Nek}}_{V_{f},\pi}
        \stackrel{\cdot{}}{=}\log_{q_{A}}\big(\widetilde{P}_{\chi,\mathfrak{p}}\big)
        -\log_{q_{A}}\big(\widetilde{P}_{\chi,\overline{\mathfrak{p}}}\big)  & \\
        =\log_{A}(\iota_{\mathfrak{p}}(P_{\chi})) -\log_{A}(\iota_{\overline{\mathfrak{p}}}(P_{\chi}))
        =\log_{A}\big(\iota_{\mathfrak{p}}\lri{P_{\chi}-P_{\chi}^{\tau}}\big) &
        =2\cdot{}\log_{A}\lri{P_{\chi}},
\end{align*}
where we write again (with a slight abuse of notation) $\log_{A} : A(K_{\chi})
\stackrel{\iota_{\mathfrak{p}}\ }{\longrightarrow{}}A(\Q_{p})\stackrel{\log_{A}\ }{\longrightarrow{}}\Q_{p}$.

The preceding discussion can be summarised by the following formulae (valid for $\chi$ trivial or quadratic):
\[
         \det\dia{-,-}^{\mathrm{Nek},\chi}_{V_{f},\pi}:=\det\begin{pmatrix} \big<q_{\chi},q_{\chi}\big>^{\mathrm{Nek}}_{V_{f},\pi}  & \big<
         q_{\chi},P_{\chi}^{\dag}\big>^{\mathrm{Nek}}_{V_{f},\pi} \\
                &  \\ \big<P_{\chi}^{\dag},q_{\chi}\big>^{\mathrm{Nek}}_{V_{f},\pi} 
                & \big<P_{\chi}^{\dag},P_{\chi}^{\dag}\big>^{\mathrm{Nek}}_{V_{f},\pi}\end{pmatrix}
                \stackrel{\cdot{}}{=}\det\begin{pmatrix} 0 & \log_{A}(P_{\chi}) \\ & \\ -\log_{A}(P_{\chi}) & 0
                \end{pmatrix}\stackrel{\cdot{}}{=}\log_{A}^{2}(P_{\chi})
\]
(where we used again the fact that $\dia{-,-}^{\mathrm{Nek}}_{V_{f},\pi}$ is skew-symmetric to compute 
$\big<P_{\chi}^{\dag},q_{\chi}\big>^{\mathrm{Nek}}_{V_{f},\pi}=-\big<q_{\chi},P^{\dag}_{\chi}\big>^{\mathrm{Nek}}_{V_{f},\pi}$,
and we wrote as above $\stackrel{\cdot{}}{=}$ to denote equality up to multiplication by a non-zero element in $L^{\times}$). 
Since $P_{\chi}\in{}A(K_{\chi})$ is a point of infinite order, and $\log_{A}$ gives an isomorphism 
between $A(\Q_{p})\otimes\Q_{p}$ and $\Q_{p}$, $\log_{A}(P_{\chi})\not=0$,
so that 
\[
       \det\dia{-,-}^{\mathrm{Nek},\chi}_{V_{f},\pi}\not=0.
\]
Recalling that $q_{\chi}$ and $P_{\chi}^{\dag}$ generate $\exsel^{1}(K_{\chi},V_{f})^{\chi}$
as an $L$-vector space by $(\ref{eq:qfin})$, this implies that $\dia{-,-}^{\mathrm{Nek},\chi}_{V_{f},\pi}$
is non-degenerate, and the last statement of Proposition $\ref{nondeg}$ finally gives
\[
             \mathrm{length}_{\mathfrak{p}_{f}}\Big(X_{\mathrm{Gr}}^{\mathrm{cc}}(\mathbf{f}/K_{\chi})^{\chi}\Big)=
             \dim_{\Q_{p}}\Big(\mathrm{Sel}_{p}(A/K_{\chi})^{\chi}\Big)\stackrel{(\ref{eq:kummera}) \text{\ and\ } (\ref{eq:nuhyp})}{=}
             1.
\]
This means that
$X_{\mathrm{Gr}}^{\mathrm{cc}}(\mathbf{f}/K_{\chi})^{\chi}\otimes_{\I}\Il\cong{}\Il/\mathfrak{p}_{f}\Il$
as $\Il$-modules, as was to be shown. 

\section[Proofs]{Proof of the main result}
This section is entirely devoted to the proof of Theorem A
stated  in the introduction. 

\subsection{An auxiliary imaginary quadratic field}
We will need the following crucial lemma, which follows combining  the main result of \cite{BFH},
\neko's proof of the parity conjecture \cite{Ne}, and the KGZ Theorem.

\begin{lemma}\label{klemma} Let $N_{A}=Np$ be the conductor of $A/\Q$ (with $p\nmid{}N$).
Assume that the following properties hold:
\begin{itemize}
\item[$(a)$] there exists a prime $q\not=p$ such that $q\Vert{}N_{A}$;
\item[$(b)$] $\mathrm{rank}_{\Z}A(\Q)=1$ and $\sha(A/\Q)_{p^{\infty}}$ is finite.
\end{itemize}
Then there exists an imaginary quadratic field $F/\Q$, of discriminant $D_{F}$, satisfying the following properties:
\begin{itemize}
\item[$1.$]  $D_{F}$ is coprime to $6N_{A}$;
\item[$2.$] $q$ (resp., every prime divisor of $N_{A}/q$) is  inert (resp., splits) in $F$;
\item[$3.$] $\mathrm{ord}_{s=1}L(A^{F}/\Q,s)=1;$
\item[$4.$] $\mathrm{rank}_{\Z}A(F)=2$ and $\sha(A/F)_{p^{\infty}}$ is finite.
\end{itemize}
(In $3$: $A^{F}/\Q$ is the $\epsilon_{F}$-twist of $A/\Q$, $\epsilon_{F}$ being the quadratic character of $F$.)
\end{lemma}
\begin{proof} By condition $(b)$ and  \neko's proof of the parity conjecture \cite[Section 12]{Ne}
$$\mathrm{sign}(A/\Q)=-1$$ (where $\mathrm{sign}(A/\Q)$ denotes the sign in the functional 
equation satisfied by the Hasse-Weil $L$-series $L(A/\Q,s)$).
Let $\chi$ be a quadratic  Dirichlet character of conductor $c_{\chi}$ coprime with $6N_{A}$ such that:
\begin{itemize}
\item[$(\alpha_{\chi})$] $\chi(q)=-1$ and $\chi(\ell)=+1$ for every prime divisor $\ell$ of $N_{A}/q$;
\item[$(\beta_{\chi})$] $\chi(-1)=+1$,
\end{itemize} 
and let $A^{\chi}/\Q$ be the $\chi$-twist of $A/\Q$. As $q\Vert{}N_{A}$, 
we deduce by \cite[Theorem 3.66]{Sh} and the preceding properties
$$\mathrm{sign}(A^{\chi}/\Q)=\chi(-N_{A})\cdot{}\mathrm{sign}(A/\Q)=-\chi(N_{A})=+1.$$
The main result of \cite{BFH}
then guarantees the existence of a quadratic Dirichlet character $\psi$, of conductor coprime with $6c_{\chi}N_{A}$,  such that
\begin{itemize}
\item[$(\alpha_{\psi})$] $\psi(\ell)=+1$ for every prime divisor $\ell$ of $6c_{\chi}N_{A}$;
\item[$(\beta_{\psi})$] $\psi(-1)=-1$;
\item[$(\gamma_{\psi})$] $\mathrm{ord}_{s=1}L(A^{\chi\psi}/\Q,s)=1$.
\end{itemize}
Define $F=F_{\chi\psi}$ as the quadratic field attached to $\chi\psi$, so $\chi\psi=\epsilon_{F}$
and $L(A^{\chi\psi}/\Q,s)=L(A^{F}/\Q,s)$ is the Hasse-Weil $L$-series of the $F$-twist of $A/\Q$. 
In particular, property $3$ in the statement is satisfied. 
By the KGZ theorem, it follows by  $(\gamma_{\psi})$ 
that $A(F)^{\epsilon_{F}}$ has rank one and $\sha(A/F)^{\epsilon_{F}}$ is finite. Together with  $(b)$, this gives
\[
              \mathrm{rank}_{\Z}A(F)=2;\ \ \#\Big(\sha(A/F)_{p^{\infty}}\Big)<\infty,
\]
i.e. property $4$ in the statement. Property $1$ is clear by construction. Moreover, by $(\alpha_{\chi})$ and $(\alpha_{\psi})$ we deduce 
$\epsilon_{F}(-1)=-1$, $\epsilon_{F}(q)=-1$ and $\epsilon_{F}(\ell)=+1$ for every prime divisor of $N_{A}/q$.
This means precisely that $F/\Q$ is an imaginary quadratic field satisfying property $2$ in the statement,
thus concluding the proof.
\end{proof}

\subsection{Proof of Theorem A} Assume that $A/\Q$ and $p>2$ satisfy the  hypotheses listed in Theorem A, i.e.
\begin{itemize}
\item[$(\alpha)$] $\overline{\rho}_{A,p}$ is an irreducible $G_{\Q}$-representation;
\item[$(\beta)$] there exists a prime $q\not=p$ at which $A$ has multiplicative reduction (i.e. $q\Vert{}N_{A}$);
\item[$(\gamma)$] $p\nmid{}\mathrm{ord}_{q}(j_{A})$;
\item[$(\delta)$] $\mathrm{rank}_{\Z}A(\Q)=1$ and $\sha(A/\Q)_{p^{\infty}}$ is finite.
\end{itemize}  
Let $K/\Q$ be a quadratic imaginary field  such that
\begin{itemize}
\item[$(\epsilon)$] $D_{K}$ is coprime with $6N_{A}$;
\item[$(\zeta)$] $q$ is inert in $K$;
\item[$(\eta)$] every prime divisor of $N_{A}/q$ splits in $K$;
\item[$(\theta)$] $\mathrm{rank}_{\Z}A(K)=2$ and $\sha(A/K)_{p^{\infty}}$ is finite;
\item[$(\iota)$] $\mathrm{ord}_{s=1}L(A^{K}/\Q,s)=1.$
\end{itemize}
The existence of such a $K/\Q$ has been proved in Lemma $\ref{klemma}$ above. 
Finally, let $L/\Q_{p}$ be a finite extension containing 
$\Q_p\lri{D_{K}^{1/2},(-1)^{1/2},1^{1/Np}}/\Q_p$, let $q_{K}\nmid{}6p$ be a rational prime which splits in $K$,
and let  $S$ be the set of  primes of $K$
consisting of  all the prime divisors of $q_{K}N_{A}D_{K}$. Then:

\begin{lemma}\label{hypmain} The data $(\mathbf{f},K,p,L,q_{K},S)$ satisfy Hypotheses $\ref{h2}$, $\ref{h3}$ and $\ref{h4}$.
\end{lemma}
\begin{proof} By construction and properties $(\epsilon)$ and  $(\eta)$, Hypothesis $\ref{h3}$ is satisfied. 
Since $\overline{\rho}_{\mathbf{f}}$ is isomorphic (by definition) to the semi-simplification of $\overline{\rho}_{A,p}$,
assumption $(\alpha)$ 
is nothing but a reformulation of Hypothesis $\ref{h2}$. To prove that Hypothesis $\ref{h4}$ holds true, note that (with the notations of \emph{loc. cit.})
$N^{+}=N_{A}/pq$ and $N^{-}=q$ by $(\zeta)$ and $(\eta)$ above. Then $N^{-}$ is a square-free product of an odd number of primes. 
It thus remains to prove that $\overline{\rho}_{A,p}\cong{}\overline{\rho}_{\mathbf{f}}$ is ramified at $q$. 
By Tate's theory, we know that 
$A/\overline{\Q}_{q}$ is isomorphic to the Tate curve $\mathbb{G}_{m}/t_{q}^{\Z}$
over the quadratic unramified extension of $\Q_{p}$,
where $t_{q}\in{}q\Z_{q}$ is the Tate period of $A/\Q_{q}$,
satisfying $\mathrm{ord}_{q}\lri{t_{q}}=-\mathrm{ord}_{q}(j_{A})$ \cite{Tate-2}, \cite[Chapter V]{Sil-2}. 
Then $$A[p]:=
A(\overline{\Q})[p]\cong{}\left\{t_{q}^{\frac{i}{p}}\cdot{}\zeta_{p}^{j} : 0\leq{}i,j<p
\right\}$$
as $I_{\Q_{q}}$-modules, where $t_{q}^{1/p}\in{}\overline{\Q}_{q}$ and $\zeta_{p}\in\overline{\Q}_{q}$ 
are fixed primitive $p$th roots of $t_{q}$ and $1$ respectively.   
As $\Q_{q}(\zeta_{p})/\Q_{q}$ is unramified, $\overline{\rho}_{A,p}$ is ramified at $q$
precisely if $\Q_{q}(t_{q}^{1/p})/\Q_{q}$ is ramified. Recalling that 
$t_{q}\in{}q\Z_{q}$ and $\mathrm{ord}_{q}(t_{q})=-\mathrm{ord}_{q}(j_{A})$,
this is the case if and only if $p\nmid{}\mathrm{ord}_{q}(j_{A})$. Then Hypothesis $\ref{h4}$ follows from $(\gamma)$.
\end{proof} 

In order to prove Theorem A, we need one more simple lemma. Omitting $S$ from the notations, recall the 
dual Selmer groups 
$X^{\mathrm{cc}}_{\Q_{\infty}}(\mathbf{f}/K):=X_{\Q_{\infty}}^{S,\mathrm{cc}}(\mathbf{f}/K)$
and $X_{\mathrm{Gr}}^{\mathrm{cc}}(\mathbf{f}/K)$
introduced in Sections $\ref{criticalgrsec}$ and $\ref{mainsec}$ respectively.

\begin{lemma}\label{compgr} $\mathrm{length}_{\mathfrak{p}_{f}}\Big(X_{\Q_{\infty}}^{\mathrm{cc}}(\mathbf{f}/K)\Big)\leq{}
\mathrm{length}_{\mathfrak{p}_{f}}\Big(X^{\mathrm{cc}}_{\mathrm{Gr}}(\mathbf{f}/K)\Big)+2$.
\end{lemma}
\begin{proof} 
As remarked in the proof of Lemma $\ref{lalaf}$, the perfect, skew-symmetric duality $\pi : \ppq\otimes_{\I}\ppq\fre{}\I(1)$
induces a natural isomorphism of $\I[G_{\Q_{p}}]$-modules:
$\mathbb{T}^{-}_{\mathbf{f}}\otimes_{\I}\I^{\ast}\cong{}\Hom{\mathrm{cont}}(
\ppq^{+},\mu_{p^{\infty}})=:\mathbb{A}_{\mathbf{f}}^{-}$. By construction and the inflation-restriction sequence, 
there is  then an exact sequence
\[
           0\fre{}\mathrm{Sel}^{\mathrm{cc}}_{\mathrm{Gr}}(\mathbf{f}/K)\fre{}\mathrm{Sel}^{\mathrm{cc}}_{\Q_{\infty}}(\mathbf{f}/K)
           \fre{}
           \bigoplus_{v|p}H^{1}\lri{\mathrm{Frob}_{v},\big(\mathbb{A}_{\mathbf{f}}^{-}\big)^{I_{v}}},
\]
where $I_{v}:=I_{K_{v}}$ is the inertia subgroup of $G_{K_{v}}$,
$\mathrm{Frob}_{v}\in{}G_{K_{v}}/I_{K_{v}}$ is the arithmetic Frobenius at $v$,
and we write for simplicity 
$H^{\ast}(\mathrm{Frob}_{v},-):=H^{\ast}(G_{K_{v}}/I_{K_{v}},-)$. 
(Here the reference to the fixed set $S$ is again omitted, so 
that $\mathrm{Sel}_{\Q_{\infty}}^{\mathrm{cc}}(\mathbf{f}/K)
:=\mathrm{Sel}_{\Q_{\infty}}^{S,\mathrm{cc}}(\mathbf{f}/K)$.)
Taking  Pontrjagin duals and then localising at $\mathfrak{p}_{f}$ gives 
an exact sequence of $\Il$-modules:
\[
              \bigoplus_{v|p}H^{1}\lri{\mathrm{Frob}_{v},\big(\mathbb{A}_{\mathbf{f}}^{-}\big)^{I_{v}}}^{\ast}_{\mathfrak{p}_{f}}
              \fre{}X_{\Q_{\infty}}^{\mathrm{cc}}(\mathbf{f}/K)_{\mathfrak{p}_{f}}\fre{}
              X^{\mathrm{cc}}_{\mathrm{Gr}}(\mathbf{f}/K)_{\mathfrak{p}_{f}}\fre{}0,
\] 
where $\lri{-}^{\ast}_{\mathfrak{p}_{f}}$ is an abbreviation for $\lri{(-)^{\ast}}_{\mathfrak{p}_{f}}=(-)^{\ast}\otimes_{\I}\Il$.
As $p$ splits in $K$, 
one deduces
\begin{equation}\label{eq:niuy}
         \mathrm{length}_{\mathfrak{p}_{f}}\Big(X_{\Q_{\infty}}^{\mathrm{cc}}(\mathbf{f}/K)\Big)\leq{}
         \mathrm{length}_{\mathfrak{p}_{f}}\Big(X_{\mathrm{Gr}}^{\mathrm{cc}}(\mathbf{f}/K)\Big)+
         2\cdot{}\mathrm{length}_{\mathfrak{p}_{f}}\Big(H^{1}\lri{\mathrm{Frob}_{p},\big(\mathbb{A}_{\mathbf{f}}^{-}\big)^{I_{p}}}^{\ast}\Big),
\end{equation}
where $I_{p}:=I_{\Q_{p}}\subset{}G_{\Q_{p}}$ is the inertia subgroup and $\mathrm{Frob}_{p}\in{}G_{\Q_{p}}/I_{\Q_{p}}$
is the arithmetic Frobenius at $p$. 

By equation $(\ref{eq:addequation})$, $\mathbb{T}_{\mathbf{f}}^{+}\cong{}\I\lri{\big(\mathbf{a}^{\ast}_{p}\big)^{-1}\cdot{}\chi_{\mathrm{cy}}\cdot{}[\chi_{\mathrm{cy}}]^{1/2}   }$
as  $G_{\Q_{p}}$-modules. Then its Kummer dual $\mathbb{A}_{\mathbf{f}}^{-}$ is isomorphic to
$\I^{\ast}\lri{\mathbf{a}_{p}^{\ast}\cdot{}[\chi_{\mathrm{cy}}]^{-1/2}}$.
Let $\gamma\in{}1+p\Z_{p}$ be a topological generator, let $[\gamma]\in{}\I$ be its image under the structural morphism 
$[\cdot{}] : \iw\fre{}\I$, and
let $\varpi=[\gamma]-1\in{}\iw$. 
Since $\mathbf{a}_{p}^{\ast}$ is an unramified character 
and $[\rho]\equiv{}1\ \mathrm{mod}\ \varpi$ for every $\rho\in{}1+p\Z_{p}$, one has isomorphisms 
of $\mathrm{Frob}_{p}$-modules
\begin{equation}\label{eq:jjji}
             H^{0}(I_{p},\mathbb{A}_{\mathbf{f}}^{-})=
             \mathbb{A}_{\mathbf{f}}^{-}[\varpi]\cong{}
             \lri{\I/\varpi\I}^{\ast}(\mathbf{a}_{p}^{\ast}).
\end{equation}
Applying $H^{1}(\mathrm{Frob}_{p},-)$ to $(\ref{eq:jjji})$ then yields
$H^{1}\lri{\mathrm{Frob}_{p},\big(\mathbb{A}_{\mathbf{f}}^{-}\big)^{I_{p}}}=
\lri{\frac{\I}{\varpi\cdot{}\I}}^{\ast}/\lri{\mathbf{a}_{p}-1}\lri{\frac{\I}{\varpi\cdot{}\I}}^{\ast}$.
Taking the Pontrjagin duals and then localising at $\mathfrak{p}_{f}$ one deduces
\begin{equation}\label{eq:hghghg}
          H^{1}\lri{\mathrm{Frob}_{p},\big(\mathbb{A}_{\mathbf{f}}^{-}\big)^{I_{p}}}^{\ast}_{\mathfrak{p}_{f}}\cong{}
          \lri{    \lri{\frac{\I}{\varpi\cdot{}\I}}^{\ast\ast}[\mathbf{a}_{p}-1]   }_{\mathfrak{p}_{f}}
          \cong{}\lri{\frac{\Il}{\varpi\cdot{}\Il}}[\phi_{f}(\mathbf{a}_{p})-1]=\Il/\mathfrak{p}_{f}\Il.
\end{equation}
Indeed, as remarked in $(\ref{eq:unif})$, $\varpi$ 
is a uniformiser of $\Il$.
Moreover, $\mathfrak{p}_{f}:=\ker(\phi_{f})$ and $\phi_{f}(\mathbf{a}_{p})=a_{p}(2)=+1$
(as $A/\Q_{p}$ is split multiplicative), so that $\mathbf{a}_{p}-1$ acts trivially on
$\Il/\mathfrak{p}_{f}\Il$ and $(\ref{eq:hghghg})$ follows. In particular,
$(\ref{eq:hghghg})$ yields
\[
               \mathrm{length}_{\mathfrak{p}_{f}}\Big(H^{1}\lri{\mathrm{Frob}_{p},\big(\mathbb{A}_{\mathbf{f}}^{-}\big)^{I_{p}}}^{\ast}\Big)=1.
\]
Together with  equation $(\ref{eq:niuy})$, this  concludes the proof of the lemma. 
\end{proof}

We can finally conclude the proof of Theorem A. To be short, we have
\begin{equation}\label{eq:<=}
         4\stackrel{\mathrm{Cor.}\ \ref{corbdmain}}{\leq{}}
         \mathrm{ord}_{k=2}L_{p}^{\mathrm{cc}}(f_{\infty}/K,k)\stackrel{\mathrm{Cor.}\ \ref{corsumain}}{\leq{}}\mathrm{length}_{\mathfrak{p}_{f}}
         \Big(X_{\Q_{\infty}}^{\mathrm{cc}}(\mathbf{f}/K)\Big)\stackrel{\mathrm{Lemma}\ \ref{compgr}}{\leq}
         \mathrm{length}_{\mathfrak{p}_{f}}\Big(X_{\mathrm{Gr}}^{\mathrm{cc}}(\mathbf{f}/K)\Big)+2
         \stackrel{\mathrm{Th.}\ \ref{bchar}}{=}4.
\end{equation}
Indeed, hypothesis $(\delta)$ gives $\mathrm{dim}_{\Q_{p}}\mathrm{Sel}_{p}(A/\Q)=1$,
and then (as in the proof of Lemma $\ref{klemma}$) \neko's proof of the parity conjecture guarantees that $\mathrm{sign}(A/\Q)=-1$.
Together with Lemma $\ref{hypmain}$, this implies that the hypotheses of Corollary $\ref{corbdmain}$
are satisfied, and then that the first inequality in $(\ref{eq:<=})$ holds true.
Lemma  $\ref{hypmain}$ also allows us to  apply Skinner-Urban's  Corollary $\ref{corsumain}$,
which gives the second inequality in $(\ref{eq:<=})$. The third inequality in $(\ref{eq:<=})$ is the content 
of the preceding lemma.
Finally, let $\chi$ denote either the trivial character or the quadratic character $\epsilon_{K}$ of $K$,
and let $K_{\chi}:=\Q$ or $K_{\chi}:=K$ accordingly. Then
$(\delta)$ and $(\theta)$ imply that (with the notations of Section $\ref{mainsec}$)
\[
             \mathrm{rank}_{\Z}A(K_{\chi})^{\chi}=1;\ \ \#\Big(\sha(A/K_{\chi})_{p^{\infty}}^{\chi}\Big)<\infty.
\]
Moreover, we know that $p$ splits in $K_{\chi}$ (i.e. in $K$, by hypothesis $(\eta)$). Then the hypotheses 
$(i)$, $(ii)$ and $(iii)$ of Theorem $\ref{bchar}$ are
satisfied by both our $\chi$'s, and by applying the  theorem twice yields
\[
               X_{\mathrm{Gr}}^{\mathrm{cc}}(\mathbf{f}/K)_{\mathfrak{p}_{f}}\cong{}X_{\mathrm{Gr}}^{\mathrm{cc}}(\mathbf{f}/\Q)
               _{\mathfrak{p}_{f}}\oplus
               X_{\mathrm{Gr}}^{\mathrm{cc}}(\mathbf{f}/K)^{\epsilon_{K}}_{\mathfrak{p}_{f}}
               \cong{}\Il/\mathfrak{p}_{f}\Il\oplus{}\Il/\mathfrak{p}_{f}\Il
\]  
\footnote{For the first isomorphism, we decomposed $X_{\mathrm{Gr}}^{\mathrm{cc}}(\mathbf{f}/K)$ into its 
`$+$ and $-$' components for the action 
of $\mathrm{Gal}(K/\Q)$, and used the  fact that the $+$-part is naturally isomorphic to 
$X_{\mathrm{Gr}}^{\mathrm{cc}}(\mathbf{f}/\Q)$  under the $K/\Q$-restriction map.},
justifying the last equality in $(\ref{eq:<=})$.

Equation $(\ref{eq:<=})$ proves that  $\mathrm{ord}_{k=2}L_{p}^{\mathrm{cc}}(f_{\infty}/K,k)=4$. 
It then follows by Bertolini-Darmon's Corollary $\ref{corbdmain}$ that 
the Hasse-Weil $L$-function of $A/K$ has a double zero at $s=1$:
\[
       \mathrm{ord}_{s=1}L(A/K,s)=2.
\]
Since $L(A/K,s)=L(A/\Q,s)\cdot{}L(A^{K}/\Q,s)$ is  the product of the Hasse-Weil $L$-functions 
of $A/\Q$ and its $K$-twist $A^{K}/\Q$, and since
$L(A^{K}/\Q,s)$ has a simple zero at $s=1$ by $(\iota)$ above, we finally deduce
$$\mathrm{ord}_{s=1}L(A/\Q,s)=1.$$

\bibliography{myref1}{}
\bibliographystyle{alpha}

\ \\

\end{document}